[1994/12/01]
\documentclass[10pt, reqno]{amsart}
\usepackage{a4wide}
\usepackage[english, activeacute]{babel}
\usepackage{amsmath,amsthm,amsxtra}
\usepackage{epsfig}
\usepackage{amssymb}
\usepackage{latexsym}
\usepackage{amsfonts}
\usepackage{hyperref}
\title{On the soliton dynamics under slowly varying medium for Nonlinear Schr\"odinger equations}
\author{Claudio Mu\~noz C.}
\address{Universit\'e de Versailles Saint-Quentin-en-Yvelines \\ LMV-UMR 8100, 45 av. des Etats-Unis, 78035 Versailles cedex, France}
\email{Claudio.Munoz@math.uvsq.fr}
\thanks{This research was supported in part by a CONICYT-Chile and an \emph{Allocation de Recherche} grants}

\chardef\bslash=`\\ 

\newtheorem{thm}{Theorem}[section]

\newtheorem{lem}[thm]{Lemma}
\newtheorem{prop}[thm]{Proposition}

\newtheorem*{thma}{{\bf Theorem A}}
\newtheorem*{thmap}{{\bf Theorem A'}}
\newtheorem*{thmb}{{\bf Theorem B}}

\newtheorem*{corc}{{\bf Corolary C}}

\theoremstyle{definition}
\newtheorem{defn}{Definition}[section]

\theoremstyle{remark}
\newtheorem{rem}{Remark}[section]

\newtheorem{Cl}{Claim}

\numberwithin{equation}{section}

\newcommand{\R}{\mathbb{R}}
\newcommand{\N}{\mathbb{N}}

\newcommand{\la}{\lambda}

\newcommand{\al}{\alpha}
\newcommand{\ga}{\gamma}

\newcommand{\re}{\operatorname{Re}}
\newcommand{\ima}{\operatorname{Im}}

 \providecommand{\abs}[1]{\lvert#1 \rvert}
 \providecommand{\norm}[1]{\lVert#1 \rVert}

\newcommand{\ve}{\varepsilon}

\newcommand{\be}{\begin{equation}}
\newcommand{\ee}{\end{equation}}
\newcommand{\ba}{\begin{equation*}}
\newcommand{\ea}{\begin{equation*}}
\newcommand{\bea}{\begin{eqnarray}}
\newcommand{\eea}{\end{eqnarray}}
\newcommand{\bee}{\begin{eqnarray*}}
\newcommand{\eee}{\end{eqnarray*}}
\newcommand{\ben}{\begin{enumerate}}
\newcommand{\een}{\end{enumerate}}
\newcommand{\nonu}{\nonumber}

\newcommand{\eval}[2][\right]{\relax
  \ifx#1\right\relax \left.\fi#2#1\rvert}

\let\abs=\envert

\let\norm=\enVert

\begin{document}
\begin{abstract}
We consider the problem of the soliton propagation, in a slowly varying medium, for a generalized, variable-coefficients nonlinear Schr\"odinger equation. We prove global existence and uniqueness of soliton-like solutions for a large class of slowly varying media. Moreover, we describe for all time the behavior of this new generalized soliton solution.
\end{abstract}
\maketitle \markboth{Soliton dynamics for perturbed NLS equations} {Claudio Mu\~noz}
\renewcommand{\sectionmark}[1]{}
\section{Introduction and Main Results}

\smallskip

In this work we continue our study of soliton-propagation under an inhomogeneous medium, started in \cite{Mu2}. Now we consider the following \emph{generalized nonlinear Schr\"odinger equation} (NLS)
\be\label{gKdV0}
iu_t + u_{xx} +  f(x, |u|^2)u =0,\quad \hbox{ in }  \R_t\times \R_x.
\ee

Here $u=u(t,x)$ is a complex-valued function, and $f: \R \times \R \to \R$ a nonlinear function. This equation is a generalization of the --one dimensional-- semilinear \emph{nonlinear Schr\"odinger equation} (NLS)
\be\label{gKdV}
iu_t + u_{xx} +|u|^{m-1} u =0, \quad \hbox{ in }\R_t\times \R_x ;\quad m> 1.
\ee
Concerning the \emph{cubic nonlinear  Schr\"odinger} equation (namely the case $m=3$), it arises in Physics as a model of wave propagation in fiber optics in a nonlinear medium, and also describes the evolution of the envelope of modulated wave groups in water waves. In two dimensions, the cubic NLS also possesses an important physical meaning. 

\smallskip

The Cauchy problem for equation \eqref{gKdV} (namely, adding the initial condition $u(t=0)=u_0$) is \emph{locally well-posed} for $u_0\in H^1(\R)$ (see Ginibre and Velo \cite{GV}). In addition, solutions of (\ref{gKdV}) are invariant under translations in space, time and phase. From the No\"ether theorem, these symmetries are related to \emph{conserved quantities}, invariant under the NLS flow, usually called \emph{mass}, \emph{energy} and \emph{momentum}: 
\be\label{M}
M[u](t):= \int_\R |u|^2(t,x)\,dx = \int_\R |u_0|^2(x)\, dx =M[u](0), \quad \hbox{(Mass),}
\ee
\bea
E[u](t) &:= & \frac 12 \int_\R |u_x|^2(t,x)\,dx - \frac 1{m+1}\int_\R  |u|^{m+1}(t,x)\,dx \label{E}  \\
& & \quad =  \frac 12 \int_\R |(u_0)_x|^2(x)\,dx - \frac 1{m+1}\int_\R |u_0|^{m+1}(x)\,dx =  E[u](0), \; \hbox{(Energy) }\nonumber
\eea
and
\be\label{P}
P[u](t) := \frac 12 \ima \int_\R \bar u u_x (t,x) \, dx =\frac 12 \ima \int_\R \bar u_0 (u_0)_x (x) \, dx = P[u](0),  \; \hbox{(Momentum). }
\ee
In the case $1<m<5$, any $H^1(\R)$ solution is global in time thanks to the conservation of mass and energy (\ref{M})-(\ref{E}), and the Galiardo-Nirenberg inequality
\be\label{GNe}
\int_\R u^{p+1} \leq K(p) \big(\int_\R u^2\big)^{\frac{p+3}{4}}\big(\int_\R u_x^2\big)^{\frac{p-1}{4}}.
\ee

One of the main properties of NLS equations is the existence of localized, exponentially decaying, stable and smooth solutions called \emph{solitons}, or \emph{traveling waves}. Given four real numbers $ x_0, v_0,\ga_0$ and $c_0>0$, traveling waves are solutions of (\ref{gKdV}) of the form
\be\label{(3)}
u(t,x) := Q_{c_0}(x-x_0-v_0t ) e^{i(c_0-\frac 14 v_0^2) t}e^{i\ga_0}e^{\frac i2 v_0 x},
\ee  
with $Q_c(s):=c^{\frac 1{m-1}} Q(c^{1/2} s)$, where $Q$ is the explicit Schwartz function satisfying the second order nonlinear differential equation
\be\label{soliton}
Q'' -Q + Q^m =0,\quad Q>0, \quad Q(x) = \Big[ \frac{m+1}{2\cosh^2(\frac {(m-1)}2 x)}\Big]^{\frac 1{m-1}} \sim e^{-|x|}.
\ee
In particular, for $v_0>0$, this solution represents a \emph{solitary wave}, with \emph{invariant profile}, defined for all time moving to the right with constant velocity.

\smallskip

The study of perturbations of solitary waves lead to the introduction of the concepts of \emph{orbital and asymptotic stability}. Orbital stability of ground states for NLS equations has been widely studied during last decades; we mention the works of Cazenave and Lions \cite{CL}, Weinstein \cite{We1,We2}, Grillakis, Shatah and Strauss \cite{GSS1, GSS2}, Cuccagna \cite{Cu1}, and Martel, Merle and Tsai \cite{MMT}. See references therein for a more detailed bibliography. On the other hand, asymptotic stability of solitary waves and related scattering results have been studied in \cite{SW,TY, BP,P,Cu2,RSS}.

\smallskip

For $m\geq 5$, solitons are shown to be \emph{orbitally unstable} and the Cauchy problem for the corresponding NLS equation has finite-time blow-up solutions, see \cite{Caz} and references there in. In this work, in order to guarantee the stability of soliton solutions, \emph{we will discard high-order nonlinearities}. In other words, we will only consider the case $1<m<5$.

\smallskip

The problem we consider in this paper possesses a large physical and mathematical literature. In the next subsection we briefly describe the main results concerning the propagation of solitons in slowly varying medium. 

\subsection{Statement of the problem, historical review} The dynamical problem of soliton interaction with a slowly varying medium is by now a classical problem in nonlinear wave propagation, representing a simple model of several physical applications. By soliton-medium interaction we mean, loosely speaking, the following problem: In (\ref{gKdV0}), consider a nonlinear function $f=f(t,x,s)$, slowly varying in space and time, possibly of small amplitude, of the form
$$
f(t,x,s^2) \sim |s|^{m-1} \quad \hbox{ as } x\to \pm \infty, \quad \hbox{ for all time;}
$$
(namely (\ref{gKdV0}) behaves like a NLS equation at spatial infinity.) \emph{Consider} a solitary wave solution (note that this assertion must be actually proved) of the corresponding variable coefficient equation (\ref{gKdV0}) with this nonlinearity, at some early time. Then we expect that this solution does interact with the medium, here represented by the nonlinearity $f(t,x,s)$. In a slowly varying medium this interaction, small locally in time, may be significantly important for the long time behavior of the solution. The emerging solution after the interaction is precisely the object of study. In particular, one considers if any change in size, position, or shape, even creation or destruction of solitons, if any, after some large time, may be present.    

\medskip

Let us review some relevant works in this direction. Kaup and Newell \cite{KN1} studied, via inverse scattering methods, slowly varying perturbations of integrable equations. In particular, they considered the following perturbed NLS equation
\be\label{Kaup}
iu_t + u_{xx} + |u|^2 u = a(\ve x) u_x.
\ee
Here the additional term $a(\ve x) u_x$ is intended to describe e.g. depth variations on a surface gravity wave packet. The authors studied the case where $a(\ve x) := \ve x^2$ and showed that, for a small $\ve$, the soliton shape remains unchanged, but both velocity and position parameters evolve following a trapped trajectory of an harmonic oscillator at leading order.  

Subsequently, this problem has been addressed in several other works and for different integrable models. In  \cite{Gr2}, the author considered the time dependent NLS equation
$$
iu_t + u_{xx} + a(\ve t) |u|^2 u =0, \quad \hbox{ in } \; \R_t\times \R_x.
$$
(See \cite{Gr2} for the physical description associated to this model.) Using a perturbative analysis the author found an approximate solution up to second order in $\ve$. This approximate solution is less dispersive than the corresponding solution for the generalized KdV equation studied in \cite{KN1}, in the sense that it does not present a tail behind the soliton solution as in the gKdV case (see also \cite{Gr1, Mu2} for more details).

\smallskip

In this paper we address the problem of soliton dynamics in the case of a slowly varying, inhomogeneous medium, but constant in time. 

\subsection{Setting and hypotheses} 

Let us come back to the general equation (\ref{gKdV0}), and consider $\ve>0$ a small parameter. Along this work we will assume that the nonlinearity $f$ is sufficiently smooth and slowly varying $x$-dependent function of the power cases, independent of time:
\be\label{f}
\begin{cases}
f(x,s^2) := a_\ve(x) |s|^{m-1} , \quad  2\leq m<5, \\
a_\ve (x) : =a(\ve x); \qquad a \in C^3(\R) \hbox{ if } m< 3, \; a \in C^4(\R) \hbox{ if }m\geq 3.
\end{cases}
\ee
Concerning the function $a$ we will assume that there exist constants $K, \mu>0$ such that
\be\label{ahyp} 
\begin{cases}
1< a(r) < 2, \; a'(r)>0,  \, |a^{(k)}(r)| \leq K e^{-\mu |r|} \; \hbox{ for all } r\in \R,\ k=1,2,3,(4); \\
0<a(r)-1  \leq  Ke^{\mu r}, \; \hbox{ for all } r\leq 0, \, \hbox{ and} \\
0<2-a(r)\leq K e^{-\mu r} \; \hbox{ for all } r\geq 0.
\end{cases}
\ee
In particular, $\lim_{r\to -\infty} a(r) =1$ and $\lim_{r\to +\infty} a(r) =2$. 
We emphasize that the special choice ($1$ and $2$) of the limits is irrelevant for the results of this paper. The only necessary conditions are that
$$
0<a_{-\infty}=  \lim_{r\to -\infty} a(r) <\lim_{r\to +\infty} a(r) =: a_{\infty}<+\infty.
$$
Of course the decay hypothesis on $a$ in (\ref{ahyp}) can be relaxed, and the results of this paper still should hold, with more difficult proofs, for asymptotically flat potentials; but for brevity and clarity of the exposition these issues will not be considered in this work.

\medskip

Recapitulating, we will consider the following 1D \emph{aNLS} equation
\be\label{aKdV}
\begin{cases}
iu_t + u_{xx} + a_\ve(x) |u|^{m-1} u =0 \quad \hbox{ in } \R_t \times \R_x,\\
 2\leq m < 5;\quad  0< \ve\leq\ve_0;  \quad a_\ve \hbox{ satisfying } (\ref{ahyp}).
\end{cases}
\ee

The main issue that we will study in this paper is the interaction problem between a soliton and a slowly varying medium, here represented by the \emph{potential} $a_\ve$. In other words, we intend to study for (\ref{aKdV}) whether it is possible to generalize the well-known soliton-like solution $Q$ of NLS. Of course, it is by now well-known that in the case $f(t,x,s^2) =f(s^2)$, and under reasonable assumptions (see for example Berestycki and Lions \cite{BL}), there exist soliton-like solutions, but our objective here will be the study of soliton solutions under a variable coefficient equation.

To support our beliefs, note that at least heuristically, (\ref{aKdV}) behaves at infinity as similar NLS equations:
\be\label{desc}
\begin{cases}
iu_t + u_{xx}  + |u|^{m-1} u =0 \quad \hbox{ as } x\to -\infty, \\
iu_t + u_{xx} + 2 |u|^{m-1} u  =0 \quad \hbox{ as } x\to +\infty.
\end{cases}
\ee
In particular, given $v_0>0$, one should be able to construct a soliton-like solution $u(t)$ of (\ref{aKdV}) such that 
$$
u(t,x) \sim Q(x-v_0 t)  e^{\frac i2 v_0x}e^{i(1-\frac 14 v_0^2)t}, \quad \hbox{ as } t\to -\infty,
$$
in some sense to be defined. Here $Q$ is the standard soliton solution introduced in (\ref{soliton}). 

On the other hand, after passing the interaction region, by stability of the soliton, this solution \emph{should behave} like 
\be\label{1d2Q}
 \sim  2^{-\frac 1{m-1}} Q_{c_\infty}(x-v_\infty t - \rho_\infty(t)) e^{\frac i2 v_\infty x}e^{i\ga_\infty(t)}  + \hbox{ lower order terms in $\ve$}, \ \hbox{ as } t\to +\infty,
\ee
for $\ve$ small enough. Here $c_\infty>0, v_\infty$ are unknown parameters, and $\rho_\infty(t), \ga_\infty(t)$ are \emph{small} perturbations. In fact, note that if $v =v(t)$ is a solution of (\ref{gKdV}) then $u(t) := 2^{-1/(m-1)}v(t) $ is a solution of 
\be\label{simpli}
iu_t + u_{xx} + 2|u|^{m-1} u =0 \quad \hbox{ in  } \R_t \times \R_x.
\ee
In conclusion, this heuristic suggests that even if the potential varies slowly, the soliton should experiment \emph{non trivial} transformations on its shape, scaling and phase, of the same order that of the amplitude of the potential $a$. 

\medskip

Before stating our main results, some important facts are in order. First, unfortunately equation (\ref{aKdV}) is in general not anymore invariant under scaling and spatial translations. Moreover, a nonzero solution of (\ref{aKdV}) \emph{might gain momentum}, in the sense that, at least formally, the quantity $P[u](t)$ defined in (\ref{P}) now satisfies the identity
\be\label{dPa}
 \partial_t P[u](t) = \frac \ve{m+1}  \int_\R a'(\ve x) |u|^{m+1} \geq 0.
\ee
Therefore the momentum is always a non decreasing quantity. This simple fact will have important consequences in our results, in particular we will obtain from this property the \emph{stability} and \emph{uniqueness} of our solution. The hypothesis $a'(\cdot)>0$ is crucial in our arguments, although we think it can be relaxed by considering for example a potential satisfying $a'(r)>0$ for all $|r|>r_0$. We will not pursue on these issues.

\medskip

On the other hand, the mass $M[u](t)$ defined in (\ref{M}) and the novel energy
\be\label{Ea}
E_a [u](t) :=  \frac 12 \int_\R |u_x|^2(t,x)\,dx  - \frac 1{m+1}\int_\R  a_\ve(x) |u|^{m+1}(t,x)\,dx
\ee
remain formally constant for all time. Moreover, a simple balance of mass and energy at $\pm \infty$ allows to determine heuristically the limiting scaling and velocity parameters in (\ref{1d2Q}), if we suppose that the \emph{lower order terms} in (\ref{1d2Q}) are of zero mass at infinity. Indeed, we have (cf. Appendix \ref{AidQ})
\be\label{predM}
M[Q] \sim \frac{c_\infty^{\frac 2{m-1}- \frac 12}}{2^{\frac 2{m-1}}}M[Q],
\ee
and
\be\label{predE}
E[Q] + \frac 14 v_0^2 M[Q] \sim \frac{c_\infty^{\frac 2{m-1} +\frac 12}}{2^{\frac 2{m-1}}} E[Q] + \frac 14 v_\infty^2 \frac{c_\infty^{\frac 2{m-1} -\frac 12}}{2^{\frac 2{m-1}}}M[Q], \quad E[Q]\neq 0,
\ee
This implies that $ c_\infty \sim 2^{\frac 4{5-m}}>1$ and $v_\infty \sim (v_0^2 + 4\frac{(5-m)}{m+3}(c_\infty-1) )^{1/2}$. 

\medskip

These formal arguments suggest the following definition.
\begin{defn}[Pure generalized soliton-solution for aNLS]\label{PSS}~

Let $v_0>0$ be a fixed number. We will say that (\ref{aKdV}) admits a \emph{pure} generalized soliton-like solution (of scaling equals $1$ and velocity equals $v_0$), if there exist $C^1$ real valued functions $\rho=\rho(t),\ga= \ga(t)$ defined for all large times and a global in time $H^1(\R)$ solution $u(t)$ of (\ref{aKdV}) such that 
\bee
& & \lim_{t\to - \infty}\|u(t) -  Q(\cdot -v_0 t) e^{\frac i2 (\cdot) v_0} e^{i(1-\frac 14 v_0^2)t} \|_{H^1(\R)} = 0, \\
& & \lim_{t\to +\infty} \|u(t) - 2^{-\frac 1{m-1}} Q_{c_\infty} (\cdot - v_\infty t - \rho(t)) e^{\frac i2(\cdot ) v_\infty} e^{i\ga(t)}\|_{H^1(\R)} =0,
\eee
with $|\rho'(t)|\ll v_0$ for all large times, and where $c_\infty, v_\infty>0$ are the scaling and velocity predicted by the mass and energy conservation law, as in (\ref{predM})-(\ref{predE}).
\end{defn}

As we will see below, a standard method allows us to construct a generalized soliton solution as $t\to -\infty$, as required in the above definition; however, it is expected that any reasonable soliton-like solution would not be able to satisfy the second assertion, because of some very small dispersive effects produced by the potential $a_\ve$.

\subsection{Previous analytic results on the soliton dynamics under slowly varying medium}

The problem of describing analytically the soliton dynamics of different integrable models under a slowly varying medium has received some increasing attention during the last years. In the framework of NLS equations with non constant potential, the first result in this direction was given by Bronski and Jerrard \cite{BJ}. In this paper it is proved that in the semiclassical limit, the soliton's mass center obeys the Newton's second law with external force given by the potential's gradient. Gustafson et al. \cite{GFJS, FGJS}, Gang and Sigal \cite{GS}, and Holmer et al. \cite{HZ} have considered the dynamics of a soliton under general potentials, for short times, namely $t \sim \frac 1\ve$. Indeed, they consider, among other cases, the cubic NLS equation
$$
iu_t + \Delta u + V(\ve x)  u + |u|^2 u =0, \quad \hbox{ on }Ê\; \R_t \times \R_x,
$$
with an initial condition at time $t=0$ close enough to a solitary wave.  Here $V$ is a bounded, sufficiently smooth potential. The best result in this case states that for any $\delta>0$ and for all time $t\lesssim \delta \ve^{-1}|\log \ve| $, the solution $u(t)$ of the corresponding Cauchy problem remains close in $H^1(\R)$ to a modulated solitary wave, up to an error of order $\ve^{2-\delta}.$
In addition, the dynamical parameters of the solitary wave follow a well defined dynamical system. See also \cite{DV} for a similar result in the case of a generalized Hartree equation. From these results it seems clear that a deeper understanding of the soliton dynamics for very large times strongly depends on the specific character of the considered potential, as we will see below.

\medskip

On the other hand, there is the \emph{opposite} problem of soliton-defect interaction. In this case, one may expect the splitting of the solitary wave, see \cite{GHW,HMZ0, HMZ}. The behavior of perturbations of small solitary waves of NLS equations, and its corresponding dynamics, has been considered e.g. in  \cite{GW,SW2}.

\medskip

Finally, a related problem is the study of the interaction soliton-medium for a generalized Korteweg- de Vries equation, following the physical literature \cite{KZ, KN1, KM1, New}. Dejak, Jonsson and Sigal in \cite{SJ,DS} considered the long time dynamics of solitary waves (solitons) over slowly varying perturbations of KdV and mKdV equations. Recently Holmer et al. have improved these results, up to quadratic order in $\ve$. Finally, we recall that in the case of the generalized Korteweg- de Vries equation
$$
u_t + (u_{xx} -\la u +a_\ve(x) u^m)_x=0, \quad \hbox{ in } \R_t\times \R_x, \quad \la\geq 0,
$$
we have described in \cite{Mu2} the dynamics of a generalized soliton solution. We proved, among other things, that no pure soliton solution is present for any small $\ve>0$ and $\la>0$. In this paper, our main objective is to extend some of these results to (\ref{aKdV}) (cf. Remark \ref{1.1}). 

\subsection{Main Results}

Let
\be\label{Te}
T_\ve := \frac 1{v_0}\ve^{-1 -\frac 1{100}}>0,
\ee
and
\be\label{pm}
p_m := 
\begin{cases}
1, \quad \hbox{ if } \; m\in [2,3),\\
2, \quad \hbox{ if } \;  m\in [3,5).
\end{cases}
\ee
The first parameter can be understood as the \emph{interaction time} between the soliton and the potential. In other words, at time $t=-T_\ve$ the soliton should remain almost unperturbed, and at time $t=T_\ve$ the soliton should have completely crossed the influence region of the potential. Note that the asymptotic $v_0 \sim 0$ depending on $\ve$ is a degenerate case and it will be discarded for this work. 

Second, the parameter $p_m$ measures the degree of accuracy of the main result, based in a Taylor expansion of the nonlinearity involved. In other words, the smoother the nonlinearity, the more accurate the main result.

\smallskip

In what follows, we assume the validity of above hypotheses, namely (\ref{f}) and (\ref{ahyp}). 
Our first result is a complete description, for all times, of the interaction soliton-potential for the aNLS equation \ref{aKdV}.

\medskip

\begin{thma}[Dynamics of a generalized soliton-solution for aNLS equation]\label{MT}~

Assume that $a(\cdot)$ satisfies (\ref{ahyp}). Let $2\leq m <5$, $v_0>0$, $\la_0 := \frac{5-m}{m+3}$ and $p_m$ be as in (\ref{pm}).
There exists a small constant $\ve_0>0$ such that for all $0<\ve<\ve_0$ the following holds. 

\ben

\item \emph{Existence of a soliton-like solution}.

There exists a unique solution $u\in C(\R, H^1(\R))$ of (\ref{aKdV}), global in time, such that 
\be\label{Minfty}
\lim_{t\to -\infty} \|u(t) - Q(\cdot -v_0t) e^{i(\cdot) v_0/2 } e^{i(1-\frac 14 v_0^2)t} \|_{H^1(\R)} =0,
\ee
with conserved mass $M[u] (t)= M[Q]$  and energy $E_a[u](t) = (\frac 14 v_0^2 -\la_0)M[Q].$

\medskip

\item \emph{Stability of interaction soliton-potential}. 
Let 
\be\label{cvinf}
\la_\infty := 2^{-\frac 1{m-1}}, \quad c_\infty := 2^{\frac 4{5-m}}\; (>1), \quad v_\infty := (v_0^2 + 4\la_0 (c_\infty -1))^{\frac 12} \; (>v_0).
\ee
There exist $K>0$, and $C^1$- functions $\rho(t), \ga(t) \in \R $ defined for all $t\geq \frac 12 T_\ve$ such that the function
$$
w(t,x) := u(t,x) - \la_\infty Q_{c_\infty}(x -v_\infty t -\rho(t)) e^{\frac i2 x v_\infty} e^{i\ga(t)},  
$$   
satisfies for all $t\geq \frac 12 T_\ve$,
\be\label{MT2}
\| w(t)\|_{H^1(\R)} + |\rho'(t)| + |\ga'(t)-c_\infty +\frac 14 v_\infty^2 | 
 \leq K \ve^{p_m}.
\ee
\een
\end{thma}

\medskip

\begin{rem}\label{1.1}
One may compare the above result with Theorems 1.1 and 1.2 in \cite{Mu2}, where a bound of order $\ve^{1/2}$ was showed (note that $p_m\geq 1$). Our present result is better due to the absence of a \emph{dispersive tail} behind the soliton, precisely of order $\ve^{1/2}$ in $H^1(\R)$, and present in the gKdV case. In other words, the soliton dynamics in the NLS case is less dispersive than that of the gKdV case, mainly due to the mass conservation law. Let us recall that a first mathematical treatment of the phenomena of dispersive tails can be found in \cite{MMfinal}. Let us finally recall that such dispersive elements in a soliton solution are not present in the case of a pure NLS or gKdV equation. 
\end{rem}

\begin{rem}
We do not discard the existence of very small solitons after the interaction, with size of order at most $\ve^{p_m}$ in $H^1(\R)$. This question is also related to the question of scattering modulo-solitons.
\end{rem}

\begin{rem}
Note that we have improved the preceding results in the literature in several directions: first, we obtain a full $\ve^2$ in the error term for the cubic case; second, our analysis is available beyond the Erhenfest's time $\ve^{-1} |\log \ve|$; third, we obtain a global picture of this generalized soliton solution, by proving a stability result; fourth, we are able to treat all the subcritical nonlinearities of the form $|u|^{m-1}u$; and fifth, we have realized that $\ve^2$ is formally the best possible $H^1$-bound since there exists, at the formal level, a \emph{defect} on the solution, of order $\ve^2$ (see Remark \ref{defe} for more details). 
\end{rem}

One may wonder whether Theorem A is available for other potentials. A first answer in that direction, is given in the next paragraph.

\subsection{Decreasing potential, and reflection}\label{decre} 

Pick now a potential $a(\cdot)$ smooth enough and such that there exist constants $K, \mu>0$ and $0<a_0<1$ such that
\be\label{ahypSr} 
\begin{cases}
a_0 < a(r) < 1, \; a'(r)< 0,  \, |a^{(k)}(r)| \leq K e^{-\mu |r|} \; \hbox{ for all } r\in \R,\ k=1,2,3,(4); \\
0<a(r)-a_0  \leq  Ke^{-\mu r}, \; \hbox{ for all } r\geq 0, \, \hbox{ and} \\
0<1-a(r)\leq K e^{\mu r} \; \hbox{ for all } r\leq 0.
\end{cases}
\ee
In particular, $\lim_{r\to -\infty} a(r) =1$ and $\lim_{r\to +\infty} a(r) =a_0$.

\medskip

Note that  $P[u](t)$ defined in (\ref{P}) now satisfies (\ref{dPa}) with the opposite sign, therefore the momentum is always a non increasing quantity. We claim the following 

\begin{thmap}[Dynamics of a reflected soliton-solution for aNLS equation]\label{MTref}~

Assume that $a(\cdot)$ now satisfies (\ref{ahypSr}). Let $2\leq m <5$, $v_0>0$, $\la_0 := \frac{5-m}{m+3}$ and $p_m$ be as in (\ref{pm}). Suppose in addition that 
\be\label{Smallness}
v_0^2 < 4\la_0( 1-a_0^{\frac 4{5-m}} ) . 
\ee
There exists a small constant $\ve_0>0$ such that for all $0<\ve<\ve_0$ the following holds. 

\ben

\item \emph{Existence of a soliton-like solution}.

There exists a solution $u_\#\in C(\R, H^1(\R))$ of (\ref{aKdV})-(\ref{ahypSr}), global in time, such that 
\be\label{SMinftyr}
\lim_{t\to -\infty} \|u_\#(t) - Q(\cdot -v_0t) e^{i(\cdot) v_0/2 } e^{i(1-\frac 14 v_0^2)t} \|_{H^1(\R)} =0,
\ee
with conserved mass $M[u_\#] (t)= M[Q]$  and energy $E_a[u_\#](t) = (\frac 14 v_0^2 -\la_0)M[Q]<0.$

\medskip

\item \emph{Reflection and stability of soliton-solution}. 

There exist $K=K(v_0)>0$ and $C^1$- functions $\rho(t), \ga(t) \in \R $ defined for all $t\geq K T_\ve$ such that the residual function
$$
w(t,x) := u_\#(t,x) - Q (x  + v_0 t -\rho(t)) e^{-\frac i2 x v_0} e^{i\ga(t)},  
$$   
satisfies, for all $t\geq K T_\ve$,
\be\label{SMT2r}
\| w(t)\|_{H^1(\R)} + |\rho'(t)| + |\ga'(t)-1 +\frac 14 v_0^2 | 
 \leq K \ve^{p_m}.
\ee
\een
\end{thmap}
 
 \begin{rem}
 Let us clarify this last result. Under small but still fixed velocities, the soliton solution is {\bf reflected} by the potential, and modulo a small defect of order $O_{H^1(\R)}(\ve^{p_m})$, it has the {\bf same scaling and opposite velocity} to the initially provided. 
 In addition, the constant $K(v_0)$ becomes unbounded as $v_0$ approaches the equality in (\ref{Smallness}). 
 \end{rem}

\begin{rem}
The \emph{uniqueness} and \emph{stability} for large times of this solution in the case $v_0^2 \geq 4\la_0( 1-a_0^{\frac 4{5-m}} ) $ is not known, mainly due to the fact that the momentum law has now the opposite sign.
\end{rem}
\bigskip

We believe that the analysis in the interaction region can be carried out in a general situation, under asymptotically flat potentials. However, stability and uniqueness properties are probably highly dependent on the considered nonlinearity.

\bigskip

\subsection{Nonexistence of pure soliton-like solution} 

An important problem arises from the above results. Is the solution $u(t)$ constructed in Theorem A above an exactly pure solitary wave for the aNLS equation? (cf. Definition \ref{PSS}.) This question is equivalent to decide whether  
$$
\lim_{t\to +\infty} \|w(t)\|_{H^1(\R)} =0.
$$
We have been unable to solve this problem, due to the lack of \emph{backwards stability} for large times\footnote{Note that in \cite{Mu2} one has backward stability for all $\la>0$.}. In other words, if we suppose $\|w(T)\|_{H^1(\R)} \leq \al$ for small $\al$ and very large time $T\gg T_\ve$, we do not know if a suitable modulation of $w(t)$ is still small (of order $\al$) at time $T_\ve$. This result is equivalent to obtain stability for a decreasing potential $a(\ve x)$.  
\medskip

However, we have formally identify a defect on the soliton solution, which makes the difference between this case and the constant coefficients case. See Remark \ref{defe} for more details.

\medskip

Finally let us recall that Theorem 1.3 in \cite{Mu2} puts in evidence the following conjecture: the presence of a non constant potential induces on any generalized solitary wave nontrivial dispersive effects, contrary to the standard NLS and gKdV equations. We believe that the same phenomenon is present in the Schr\"odinger case.

\bigskip

\begin{rem}[Time depending potentials]
As expected, our results are also valid, with easier proofs, for the following time dependent gKdV equation:
\be\label{time}
iu_t + u_{xx}  + a(\ve t) |u|^{m-1} u =0, \quad \hbox{ in } \R_t\times \R_x.
\ee
Here $a$ satisfies (\ref{ahyp}) now in the time variable. Note that this equation is invariant under scaling and space translations. In addition, the mass $M[u]$ and momentum $P[u]$ remain constants and the energy 
$$
\tilde E[u](t) := \frac 12 \int_\R |u_x|^2  -\frac{a(\ve t)}{m+1} \int_\R |u|^{m+1}
$$
satisfies
$$
\partial_t \tilde E[u](t) = -\frac{\ve a'(\ve t)}{m+1}\int_\R |u|^{m+1}.
$$ 
Furthermore, Theorem A still holds with $\la_\infty = 2^{-1/(m-1)}$, and $c_\infty = 2^{4/(5-m)}$. We left the details to the reader.
\end{rem}

\subsection{The two dimensional case} A natural question arising from the above results is their extension to higher dimensions. Very few results are valable on this topic, apart from the aforementioned works \cite{FGJS,GFJS}.

\smallskip

Here we shall consider the two dimensional case with a potential $a(\cdot ) $ depending only on one spatial variable. Indeed, let $x=(x_1,x_2)\in \R^2$. For $\ve>0$ small, consider the following aNLS equation
\be\label{aKdVN}
iu_t + \Delta u + a(\ve x_1) |u|^{m-1} u =0 \quad \hbox{ in } \R_t \times \R_x^2, \quad 2\leq m < 3.
\ee
We assume $a=a(r)$ satisfying (\ref{ahyp}). The exponent $m$ is chosen to ensure a subcritical regime in $L^2$ and global wellposedness with $L^2$ and $H^1$ data (cf. \cite{GV}). The mass $M(t)$, energy $E_a(t)$ and --vectorial-- momentum $P(t)$ in (\ref{M})-(\ref{P}) are defined in the usual way. From the above assumptions we have mass and energy formally conserved, and
\be\label{dPN}
\partial_t P[u](t) = \frac{\ve e_1}{m+1}\int_\R a'(\ve x_1) |u|^{m+1}(t,x) dx \geq0.
\ee
Here $e_1$ is the first unitary vector in $\R^2$. 

\smallskip

Concerning \emph{solitons} solutions, given $ x_0, \tilde v_0 \in \R^2$, $\ga_0\in \R$ and $c_0>0$, there exists a solution of the two-dimensional version of (\ref{gKdV}) of the form
\be\label{(N)}
u(t,x) := Q_{c_0}(x- x_0-\tilde v_0t ) e^{i(c_0-\frac 14 |\tilde v_0|^2) t}e^{i\ga_0}e^{\frac i2 \tilde v_0\cdot x},
\ee  
with $Q_c(s):=c^{\frac 1{m-1}} Q(c^{1/2} s)$. Here $Q$ is the unique (modulo translations) Schwartz function satisfying the second order nonlinear elliptic equation
\be\label{solitonN}
\Delta Q -Q + Q^m =0,\quad Q>0, \quad |Q(x)| \leq K e^{-|x|}.
\ee

For this case, we have the following positive result.

\begin{thmb}[Dynamics of a two-dimensional generalized soliton-solution]\label{TCa}~

Assume the preceding hypotheses. Let $2\leq m < 3$, and $v_0>0$.
There exists a small constant $\ve_0>0$ such that for all $0<\ve<\ve_0$ the following holds. 

\ben
\item \emph{Existence of a soliton-like solution}.

There exists a \emph{unique} solution $u\in C(\R, H^1(\R^2))$ of (\ref{aKdVN}), global in time, such that 
$$
\lim_{t\to -\infty} \|u(t) - Q(\cdot -v_0e_1 t) e^{i(\cdot) v_0 e_1 /2 } e^{i(1-\frac 14 v_0^2)t} \|_{H^1(\R^2)} =0.
$$

\medskip

\item \emph{Stability of interaction soliton-potential}.
 
Let  $\la_\infty = 2^{-\frac 1{m-1}}$, $c_\infty := 2^{2/(3-m)}$, and
\be\label{cvinfN}
v_\infty= v_\infty(v_0) := (v_0^2 + \al_0 (c_\infty -1))^{\frac 12},  \quad \hbox{ with }\;
\al_0 := \frac{4(3-m)}{m+1} \times \frac{ \int Q^{m+1} }{\int Q^2}.
\ee
There exist $K>0$ and $C^1$- functions $\ga(t) \in \R $, $\rho(t) \in \R^2$ defined for all $t\geq \frac 12 T_\ve$ such that the function
$$
w(t,x) := u(t,x) - \la_\infty Q_{c_\infty }(x -v_\infty e_1 t -\rho(t)) e^{\frac i2 x\cdot v_\infty e_1} e^{i\ga(t)}  
$$   
satisfies for all $t\geq T_\ve$,
\be\label{MT2N}
\| w(t)\|_{H^1(\R^2)} + |\rho'(t)| + |\ga'(t)-c_\infty +\frac 14 v_\infty^2 |  \leq K \ve.
\ee
\een
\end{thmb}

\begin{rem}
The proof of this theorem is close the proof of Theorem A. Note that uniqueness and stability follow from the fact that for any constant $v>0$,
\be\label{sign}
\partial_t \{ v e_1 \cdot P[u](t)\} \geq 0.
\ee
In section \ref{TC} we sketch the main lines of the proof.
\end{rem}

\begin{rem}
The restriction to the two dimensional case is a consequence of the lack of smoothness for the power nonlinearity in higher dimensions. We believe that the above results remain valid for a sufficiently smooth nonlinearity of the form $f(x, |u|^2) u$ (e.g.  $ f(x, s^2)  := a_\ve(x)( s^2 + a_0 s^4)$, with $a_0$ small enough in the one dimensional case.)
\end{rem}

Last, thanks to the invariance of (\ref{aKdVN}) with respect to Galilean boosts on the $x_2$ direction we obtain the following striking result.

\begin{corc}[Description of the soliton dynamics for a \emph{general} incident velocity]\label{CorD}~

Let $\tilde v = (\tilde v_1, \tilde v_2) \in \R^2$ be an initial velocity such that $\tilde v\cdot e_1>0$. Then Theorem B holds with the obvious modifications, and with $\ve_0$ independent of $\tilde v_2$. Moreover, the final velocity is given by $\tilde v_\infty := (v_\infty(\tilde v_1), \tilde v_2)$.
\end{corc}

\medskip

\begin{rem}
Note that in this situation one has the following \emph{refraction law} among the two velocities and the angles of incidence ($\theta_{-\infty}$) and refraction ($\theta_{+\infty}$):
$$
|\tilde v| \sin \theta_{-\infty} = |\tilde v_\infty| \sin \theta_{+\infty}.
$$ 
\end{rem}

\medskip

\begin{proof}[Proof of Corolary C]
Note that $\tilde v_1>0$. Since any solution of (\ref{aKdVN}) is invariant under the Galilean transformation 
$$
\mathcal G[u](t,x)= \mathcal G[u](t,x_1, x_2) = u(t,x_1, x_2 -\tilde v_2 t ) e^{\frac i2 x_2 \tilde v_2 }e^{-\frac i4 \tilde v_2^2 t},
$$
we may suppose without loss of generality that $\tilde v = v_0 e_1$, for $v_0 =|\tilde v|>0$.  We apply Theorem B with this new data. The conclusion follows at once.
\end{proof}

\begin{rem}
Let us remark that the extension of this result to case of a decreasing potential is direct, after the preceding results (cf. Theorem A'). Additionally, the proof of non existence of pure soliton-like solutions for this case remains an open problem. 
\end{rem}

\medskip

Before starting the computations, let us explain the main ideas behind the proof of Theorems A and B.

\subsection{Main ideas of the proof}\label{sop}

Similarly to \cite{Mu2}, the proof of our results are mainly based in the construction of a new {\bf approximate solution} of (\ref{aKdV}) in the interaction region, see e.g. \cite{Martel, MMcol1, MMMcol, MMfinal, Mu1} for similar computations. The construction requires several new computations, up to second order in $\ve$ in the best cases, in order to describe with enough accuracy the behavior of the soliton solution. We recall that this idea is essential in order to improve the existent literature, and to find a formal lower bound on the defect at infinity.

\smallskip

The method is as follows: one separates the analysis among three different time intervals: $t\ll -\ve^{-1} $, $\abs{t} \leq \frac K\ve$ and $\ve^{-1} \ll t$. On each interval the solution possesses a specific behavior, as is now described. 

\smallskip

Indeed, in the first interval of time we prove that $u(t)$ remains very close to a soliton-solution with no change in the scaling, velocity, phase and shift parameters. This result is possible for negative very large times, where the soliton is still far from the interacting region $|t|\leq \ve^{-1}$, and the potential is essentially $a\equiv 1$. The idea is to use a compactness property of the soliton solution to get exponential decay in time of the convergence at infinity in (\ref{Minfty}).

\smallskip

For the second regime, namely $\abs{t}\leq \ve^{-1}$, the soliton-potential interaction leads the dynamics of $u(t)$. The novelty here is the construction of an \emph{approximate solution} of (\ref{aKdV}) with high order of accuracy such that $(a)$ at time $t\sim -\ve^{-1}$ this solution is close to a modulated soliton solution and therefore to $u(t)$; $(b)$ it describes the soliton-potential interaction inside this interval; and $(c)$ it is close to $u(t)$ in the whole interval $[-\ve^{-1}, \ve^{-1}]$, uniformly on time, modulo a modulation on some degenerate directions.

\smallskip

Finally, for times $t\gg \ve^{-1}$, some well known stability properties allow to establish the stability of the solution $u(t)$ as a soliton-like solution, and therefore the proof of Theorem A. These arguments are easy to extrapolate to higher dimensions, giving the proof of Theorem B.

\bigskip

\bigskip

\noindent
{\bf Notation}. Along this paper, both $C, K,\mu>0$ will denote fixed constants, independent of $\ve$, and possibly changing from one line to the other.

\bigskip

Finally, some words about the organization of this paper. First in Section \ref{2} we sketch the proof of Theorem A. Section \ref{4} is devoted to the proof of the main ingredients of Theorem A. Section \ref{Ap} is devoted to the proof of Theorem A'. In Section \ref{TC} we prove Theorem B. Finally Appendices \ref{A} and \ref{B} are devoted to the construction of the soliton-like solution for negative large times and to prove the asymptotic behavior as $t\to +\infty$.

\bigskip

\section{Proof of Theorem A}\label{2}

The proof is similar to the proof of Theorem 1.2 in \cite{Mu2}, and it is based in three independent results: Propositions \ref{Tm1}, \ref{T0} and \ref{Tp1}. Assuming these three results, the proof of Theorem A is straightforward. For the proof of each Proposition, we did as follows. In Section \ref{4} we prove Proposition \ref{T0}, and in Appendices \ref{A} and \ref{B} we prove Propositions \ref{Tm1} and \ref{Tp1}.

\medskip

\noindent
{\bf Step 1. Construction of a soliton-like solution at infinity}.
First we prove the existence and uniqueness of a \emph{pure} soliton-like solution for (\ref{aKdV}) for $t\to -\infty$. See e.g. \cite{Mu2}, Theorem 1.1 for a related result.

\begin{prop}[Existence and uniqueness of a pure soliton-like solution]\label{Tm1}~

There exists $\ve_0>0$ such that for any $0<\ve < \ve_0$, there exists a unique solution $u \in C(\R, H^1(\R))$ of (\ref{aKdV}) such that 
\be\label{lim0}
\lim_{t\to -\infty} \|u(t) -  Q(\cdot -v_0 t)e^{\frac i2 (\cdot ) v_0}e^{i(1-\frac 14v_0^2)t}  \|_{H^1(\R)} =0,
\ee
with mass $M[u](t) = M[Q]$ and energy $E_a[u](t) = (\frac 14v_0^2 -\la_0)M[Q].$
Moreover, there exist constants $K,\mu>0$ such that  for all $t\leq -\frac 1{2}T_\ve$,
\be\label{minusTe}
\|u(t) -  Q(\cdot -v_0 t)e^{\frac i2 (\cdot ) v_0}e^{i(1-\frac 14v_0^2)t} \|_{H^1(\R)} \leq  K e^{\ve \mu t}.
\ee
In particular, 
\be\label{mTep}
\|u(-T_\ve) - Q(\cdot + v_0 T_\ve )e^{\frac i2 (\cdot ) v_0}e^{-i(1-\frac 14v_0^2)T_\ve } \|_{H^1(\R)} \leq K e^{- \mu \ve^{-\frac 1{100}}} \leq K \ve^{10},
\ee
provided $0<\ve<\ve_0$ small enough.
\end{prop}

\begin{proof}
See Appendix \ref{A}.
\end{proof}

Note that the mass and energy identities above follow directly from (\ref{lim0}), Appendix \ref{AidQ} and the energy conservation law from Proposition \ref{Cauchy}. Hereafter, we consider \emph{the} solution $u(t)$ given by the above Proposition.

\medskip

\noindent
{\bf Step 2. Interaction soliton-potential}. The next step in the proof consists on the study of the region of time $[-T_\ve, T_\ve]$, which is the zone where the interaction soliton-potential governs the dynamics. 

Recall the definition of $\la_\infty$, $c_\infty$ and $v_\infty$ in (\ref{cvinf}), and $p_m$ in (\ref{pm}). 

\begin{prop}[Dynamics of the soliton in the interaction region]\label{T0}~

Suppose $v_0>0$. There exist a constant $\ve_0>0$ such that the following holds for any $0<\ve <\ve_0$.
Let $u=u(t)$ be a globally defined $H^1(\R)$ solution of (\ref{aKdV}) such that
\be\label{hypINT}
\| u(-T_\ve) - Q(\cdot + v_0 T_\ve)e^{\frac 12i(\cdot)v_0}  e^{-i(1-\frac 14 v_0^2)T_\ve}  \|_{H^1(\R)}\leq K \ve^{p_m}.
\ee
Then there exist $K_0=K_0(K)>0$, and  $\rho_\ve, \ga_\ve \in \R$ such that
\be\label{INT41}
\|u(T_\ve)  -\la_\infty Q_{c_\infty}(\cdot - \rho_\ve) e^{\frac i2(\cdot)v_\infty} e^{i\ga_\ve} \|_{H^1(\R)} \leq K_0 \ve^{p_m},
\ee
and
\be\label{INT42}
 \frac {99}{100}v_0 T_\ve \leq \rho_\ve \leq \frac {101}{100}(2v_\infty-v_0) T_\ve.   
\ee
\end{prop}

\begin{proof}
See Section \ref{4} for a proof of this Proposition and some additional but not less important properties related to this result. 
\end{proof}

We apply the above Proposition as follows. From  (\ref{mTep}), one has directly (\ref{hypINT}). Then the solution $u(t)$ satisfies (\ref{INT41}) and (\ref{INT42}). We are done.
 
\medskip

\noindent
{\bf Last step. Long time behavior}. The final step of the proof is the use of the following result.  

\begin{prop}[Stability in $H^1(\R)$]\label{Tp1}~

Suppose $2\leq m<5$. There exists $\ve_0>0$ such that if $0<\ve <\ve_0$ the following hold. Suppose that for some time $t_1\geq \frac 12 T_\ve$, $ v_0 t_1 \leq X_0 $ and $\ga_0\in \R$ and $K>0$,
\be\label{18}
\| u(t_1) -  \la_\infty Q_{c_\infty}(\cdot-X_0) e^{\frac i2 xv_\infty} e^{i\ga_0} \|_{H^1(\R)} \leq  K\ve^{p_m}.
\ee
where $u(t)$ is a global $H^1$-solution of (\ref{aKdV}). 

\smallskip
\noindent
Then there exist $K_0>0$ and $C^1$-functions $\rho_2(t), \ga_2(t) \in \R$ defined in $[t_1,+\infty)$ such that 
$$
w(t) := u(t) -   \la_\infty  Q_{c_\infty} (\cdot - v_\infty t -\rho_2(t))e^{\frac i2 (\cdot) v_\infty}  e^{i\ga_2(t)},
$$
satisfies for all $t\geq t_1$,
\be\label{S}
\| w (t) \|_{H^1(\R)} +|\rho_2'(t)| +|\ga_2'(t) -c_\infty +\frac 14 v_\infty^2  |  \leq K_0\ve^{p_m},
\ee
where, for some $K>0$,
$$
|\rho_2(t_1) + v_\infty t_1- X_0 | + |\ga_2(t_1) - \ga_0 | \leq  K \ve^{p_m}.
$$
\end{prop}
\begin{proof}
For the proof, see Appendix \ref{B}.
\end{proof}

\noindent
{\bf End of proof of Theorem A.} 

We conclude in the following form: define $t_1 := T_\ve$, $X_0 := \rho_\ve$ and $\ga_0:=\ga_\ve$. From (\ref{INT41})-(\ref{INT42}) we have (\ref{18}) and therefore (\ref{S}). By renaming $\rho(t) :=\rho_2(t)$, $\ga(t) := \ga_2(t)$, we have that from (\ref{S}) and (\ref{cvinf}) we obtain (\ref{MT2}). The proof is now complete, provided Propositions \ref{Tm1}, \ref{T0} and \ref{Tp1} are valid.

\bigskip

\section{Proof of Proposition \ref{T0}}\label{4}

The proof of Proposition \ref{T0} is divided in four steps. In the first part, we introduce some basic notation. Next, in Step 2 we construct an approximate solution $\tilde u$ solving (\ref{aKdV}) up to second order in $\ve$ in the best cases. Then in Step 3 we prove that $\tilde u$ is close to an actual solution up to order $\ve^{p_m}$ in the whole interval $[-T_\ve, T_\ve]$. Finally, in Step 4 we conclude. 

\medskip

\noindent
{\bf Step 1. Preliminaries.}  
\subsection{Cauchy Problem} First we recall the local well-posedness theory for the Cauchy problem associated to (\ref{aKdV}).

Let  $u_0\in H^1(\R)$.  We consider the following initial value problem
\be\label{Cp1}
\begin{cases}
iu_t + u_{xx}  +  a_\ve (x)|u|^{m-1} u = 0 \quad \hbox{ in } \R_t \times \R_x, \quad  2\leq m<5, \\
u(t=0) =  u_0.
\end{cases}
\ee
Following \cite{Caz}, thanks to the subcritical character of the nonlinearity and the bounds on the potential, we have the following result.

\begin{lem}[Local and global well-posedness in $H^1(\R)$, see \cite{Caz}]\label{Cauchy}~

Suppose $u_0\in H^1(\R)$. Then there exist a unique solution $u\in C(\R, H^1(\R))$ of (\ref{Cp1}). Moreover, for any $t\in \R$ the mass $M[u](t)$ and the energy $E_a[u](t)$ from (\ref{Ea}) remain constant, and the momentum $P[u](t)$ defined in (\ref{P}) obeys (\ref{dPa}). The same result is valid for $L^2(\R)$ data.
\end{lem}

\begin{proof}
The proof is standard, and it is based in a Picard iteration procedure. For the proof see Example 3.2.11, Theorem 4.3.1, Corollary 4.3.3 and Corollary 6.1.2 in \cite{Caz}.
\end{proof}
 
We will also need some properties of the corresponding linearized operator of (\ref{aKdV}). For the proofs, see e.g. \cite{MMcol1}.

\subsection{Spectral properties of the linear NLS operator}

Fix $c>0$,  $m=2,3$ or 4, and let
\be\label{defLy}
    \mathcal{L}_+ w(y) := - w_{yy} +  cw - m Q_c^{m-1}(y) w, \quad\hbox{ and }\quad  \mathcal{L}_- w(y) := - w_{yy} +  cw -  Q_c^{m-1}(y) w;
\ee
where $w=w(y)$.  Then one has

\begin{lem}[Spectral properties of $\mathcal{L}_{\pm}$, see \cite{MMcol2}]\label{surL}~

The operators $\mathcal{L}_{\pm}$ defined (on $L^2(\R)$) by \eqref{defLy} have as domain of definition the space $H^2(\R)$. In addition, they are self-adjoint and satisfy the following properties:
\ben
\item \emph{First eigenvalue}. There exist a unique $\lambda_m>0$ such that  $\mathcal{L}_+  Q_c^{\frac {m+1}2} =-\lambda_m Q_c^{\frac {m+1}2} $. 

\smallskip

\item The kernel of $\mathcal{L}_+$ and $\mathcal L_-$ is spawned by $Q_c'$ and $Q_c$ respectively. Moreover,
\be\label{LaQc}
\Lambda Q_c := \partial_{c'} {Q_{c'}}\big|_{ c'=c} = \frac 1c \Big[\frac 1{m-1} Q_c + \frac 12 xQ'_c \Big],
\ee
satisfies $\mathcal{L}_+ (\Lambda Q_c)=- Q_c$. Finally, the continuous spectrum of $\mathcal L_\pm$ is given by $\sigma_{cont}(\mathcal L_\pm) =[c,+\infty)$.

\smallskip

\item \emph{Inverse}. For all   $h=h(y) \in L^2(\R)$ such that $\int_\R h Q'=0$ (resp. $\int_\R h Q=0$), there exists a unique $ h_+ \in H^2(\R)$ (resp. $h_-\in H^2(\R)$)  such that $\int_\R h_+Q'=0$ (resp. $\int_\R h_- Q =0$) and $\mathcal{L}_+ h_+ = h$ (resp $\mathcal{L}_- h_- = h$). Moreover, if $h$ is even (resp. odd), then $h_{\pm}$ is even (resp. odd).

\smallskip

\item \emph{Regularity in the Schwartz space $\mathcal S$}. For $h\in H^2(\R)$,  $\mathcal{L}_\pm h \in \mathcal{S}$ implies $h\in \mathcal{S}$.

\smallskip

\item\label{6a} \emph{Coercivity}.  There exists $\nu_0>0$ such that the following is satisfied.

\ben
\item For $w=w(y)\in H^1(\R)$, define
$$
\mathcal B[w,w] := 
 \frac 12\int_\R (|w_y|^2+ |w|^2 - Q_c^{m-1} |w|^2 - (m-1) Q_c^{m-1} (\re w)^2) 
$$
Suppose that $\ima \int_\R \bar w Q_c = \re \int_\R \bar w Q_c' =0 $. Then one has 
$$
\mathcal B[w,w] \geq \nu_0 \int_\R |w|^2  -K\abs{\re \int_\R \bar w Q_c }^2.
$$
for some $ K>0$.

\smallskip

\item Suppose now that for $v\neq 0$, and $\theta\in \R$ one has  
$$
 \re \int_\R  \bar w Q_c' e^{iyv/2}e^{i\theta} = \ima \int_\R \bar w Q_c e^{iyv/2}e^{i\theta} =0.
 $$
Then
$$
\tilde{\mathcal B}[w,w] \geq \nu_0 \int_\R |w|^2 -K\abs{\re \int_\R \bar w Q_ce^{iyv/2}e^{i\theta}}^2,
$$
where $\tilde{\mathcal B}[w,w] :={\mathcal B}[w e^{iyv/2}e^{i\theta} ,  w e^{iyv/2}e^{i\theta}]$.
\een
 
\een
\end{lem}

\medskip

We finish this paragraph with a last definition. We denote by $\mathcal{Y}$ the set of $C^\infty$ functions $f$ such that for all $j\in \N$ there exist $K_j,r_j>0$ such that for all $x\in \R$ we have
\be\label{Y}
|f^{(j)}(x)|\leq K_j (1+|x|)^{r_j} e^{-\frac 12|x|}.
\ee
Recall that $Q_c$ is a function in $\mathcal Y$, for $c\geq \frac 14.$

\bigskip

\noindent
{\bf Step 2. Construction of the approximate solution}.

\medskip

\noindent
We look for $\tilde u(t,x)$, an approximate solution for (\ref{gKdV0}),  carrying out a specific structure. We want  $\tilde u$ as a suitable modulation of a solitary wave, solution of the NLS equation
\be\label{orig}
iu_t +u_{xx}  + |u|^{m-1}u =0,
\ee 
plus some extra terms, of small order in $\ve$. 
Indeed, for $t\in [-T_\ve, T_\ve]$, let
$$
 \rho( t), \ga(t) \in \R, \quad c(t), v(t) >0,
$$
to be fixed later. Consider
\be\label{defALPHA}
    y:=x-  \rho(t),  \quad \hbox{and} \quad     \tilde R(t,x): = \frac{Q_{c(t)}(y)}{\tilde a(\ve \rho(t))}e^{i\Theta(t,x)},
\ee
where 
\be\label{param0}
\tilde a := a^{\frac 1{m-1}}, \quad \Theta (t,x) := \int_0^t c(s) ds  + \frac 12 v(t)x -\frac 14 \int_0^t v^2(s) ds + \ga(t).  
\ee
In addition, we will search for \emph{bounded} parameters $(c,v,\rho,\ga)$ satisfying the following constraints:
\be\label{aprio}
\frac 12 \leq c(t)\leq 2^{5}, \quad \frac 12 v_0 \leq v(t)\leq v_0 + 2^5,\quad |\rho'(t) -v(t)| \leq  \frac{v_0}{100},  \quad \ga(t) \in \R.
\ee
By now we only need these hypotheses. Later we will construct a quadruplet $(c,v,\ga,\rho)$ with better estimates, see (\ref{Imp1})-(\ref{Imp2}).

\smallskip

On the other hand, the form of the ansatz $\tilde u(t,x)$ is the sum of the soliton plus a small correction term:
\begin{equation}\label{defv} 
    \tilde u(t,x) := \tilde R(t,x)+ w(t,x),
\end{equation}
where the correction term depends on the nonlinearity we consider:
\be\label{defW}
w(t,x):=\begin{cases}
\displaystyle{ \ve (A_{1,c} (t, y) + i  B_{1,c}(t, y) ) e^{i\Theta}}, \quad \hbox{ in the case $2\leq m< 3$},\\
\displaystyle{ \sum_{k=1,2} \ve^k (A_{k,c} (t, y) + i  B_{k,c}(t, y) ) e^{i\Theta}},  \quad \hbox{ for the case $3\leq m< 5$}, 
\end{cases}
\ee
where  $A_{k,c}(t, y) := A_k (t, \sqrt{c(t)} y) $, and $A_k, B_k$ are unknown real valued functions to be determined. 

\smallskip

Let us be more precise. Given $k=1$ (for $m< 3$), or $k=1$ or $2$ for $m\geq 3$, we will search for functions $(A_{k,c}(t, y), B_{k,c}(t, y))$ such that for all $t\in [-T_\ve, T_\ve]$ and for some fixed constants $K,\mu>0$,
\be\label{IP}
\|A_{k,c}(t, \cdot) \|_{L^\infty(\R)} + \|B_{k,c}(t, \cdot)\|_{L^\infty(\R)} \leq K e^{-\mu\ve |\rho(t)|}, \qquad A_{k,c}(t, \cdot ), B_{k,c}(t, \cdot )\in \mathcal Y.
\ee

We want to measure the size of the error produced by inserting $\tilde u$ as defined in (\ref{defW}) in the equation (\ref{gKdV0}). For this purpose, let 
\be\label{2.2bis}
S[\tilde u](t,x) := i\tilde u_t + \tilde u_{xx} + a_\ve(x) |\tilde u|^{m-1} \tilde u.
\ee

The next result gives the error associated to such an approximated solution. 

\begin{prop}[Decomposition of $S{[}\tilde u{]}$]\label{prop:decomp}~

Let $\Lambda A_{c} := \partial_c A_{c}$. For every $t\in [-T_\ve, T_\ve]$, one has the following nonlinear decomposition of the error term $S[\tilde u]$.
\ben
\item \emph{Case $2\leq m< 3$}. 
\bea\label{Sm}
S[\tilde u](t,x) & : = & \mathcal F_0(t, y) e^{i\Theta} + \tilde S[\tilde u](t,x) \nonu \\
& := &   \Big[ \mathcal F_0(t, y) +  \ve \mathcal F_1(t,y) +  \ve^2 \mathcal F_2(t, y)  + \ve^3 f(t) \mathcal F_c(y) \Big] e^{i\Theta},
\eea
where 
\bea\label{F0}
\mathcal F_0(t, y)&  := & - \frac 12( v'(t) - \ve f_1(t) )y\tilde u + i( c'(t)- \ve f_2(t) ) \partial_c \tilde u \nonumber \\
& & \quad   - ( \ga'(t) + \frac 12 v'(t)\rho (t) ) \tilde u +  i(\rho'(t) -v(t))\partial_\rho \tilde u,
\eea
with $f_1(t)=f_1(c(t),\rho(t))$ and $ f_2(t) =f_2(c(t),v(t),\rho(t))$ given by
\be\label{f1f2}
 f_1(t) :=\frac{8 a'(\ve \rho(t)) c(t)}{(m+3)a(\ve \rho( t))},  \qquad f_2(t)  :=\frac{4 a'(\ve \rho( t)) c(t) v(t)}{(5-m)a(\ve \rho(t))}; 
\ee
\be\label{F1}
\mathcal F_1(t, y)  :=  F_1(t, y) + i G_1(t, y)  - \big[  \mathcal L_+ (A_{1,c})  + i \mathcal L_- (B_{1,c}) \big],
\ee
and
\bea
& & F_1(t, y) :=  \frac{a' (\ve \rho( t))}{\tilde a^m (\ve \rho(t))}yQ_c(y)\big[ Q_c^{m-1}(y) -\frac {4c(t)}{m+3}\big], \label{F1G1} \\
& &  G_1(t, y) := \frac{a' (\ve \rho( t)) v(t) }{\tilde a^m(\ve \rho( t))}\big[ \frac{4c(t)}{5-m}\Lambda Q_c (y)-\frac{1}{m-1}Q_c(y)\big].\label{G1F1}
\eea
Furthermore, suppose that $(A_{1,c}, B_{1,c})$ satisfy (\ref{IP}). Then
\be\label{tSH2}
\|\ve^2 (\mathcal F_2(t, \cdot)+ \ve f(t) \mathcal F_c )e^{i\Theta} \|_{H^1(\R)}\leq K\ve^2(e^{-\ve\mu |\rho(t)|} +\ve),
\ee
uniformly in time.
\medskip

\item \emph{Case $3\leq m<5$}. Define $\partial_\rho \tilde u := \partial_\rho \tilde R - w_y$. Here one has the improved decomposition 
\bea\label{S2}
S[\tilde u](t,x)  &  :=  &  \mathcal F_0(t, y)  e^{i\Theta} + \tilde S[\tilde u](t,x) \nonu \\
& := &  \big[ \mathcal F_0(t, y) +  \ve \mathcal F_1(t,y) +  \ve^2 \mathcal F_2(t, y) + \ve^3 \mathcal F_3(t, y)  +   \ve^4 f(t) \mathcal F_c(y) \big] e^{i\Theta},
\eea
with $\mathcal F_0$ given now by 
\bea\label{F0mod}
\mathcal F_0(t, y)&  := & - \frac 12(v'(t) - \ve f_1(t) ) y\tilde u + i( c'(t)- \ve f_2(t)) \partial_c \tilde u \nonumber \\
& & \quad   - (\ga'(t) + \frac 12 v'(t)\rho (t) -\ve^2 f_3(t)) \tilde u +  i(\rho'(t) -v(t) -\ve^2 f_4(t))\partial_\rho \tilde u,
\eea 
with $f_1, f_2$ as in (\ref{f1f2}), and for $\al_{(\cdot)}$, $\beta_{(\cdot)} \in \R$, 
\be\label{f3}
\quad f_3(t) = f_3(c(t), v(t),\rho(t)) := ( \al_{I}  + \al_{II} \frac{v^2(t)}{c(t)} ) \frac{a'' }{a}(\ve \rho(t))  + (  \al_{III}  + \al_{IV} \frac{v^2(t)}{c(t)} )\frac{a'^2}{a^2}(\ve \rho(t)) ,
\ee
and
\be\label{f4}
f_4(t)= f_4(c(t),v(t),\rho(t)) := \big\{ \beta_I \frac{a'' }{a}(\ve \rho(t)) +\beta_{II} \frac{a'^2 }{a^2 }(\ve \rho(t))\big\} \frac{v(t)}{c(t)}. 
\ee
In addition,
\be\label{F12}
\mathcal F_k(t, y)  :=  F_k(t, y) + i G_k(t, y)  - \big[  \mathcal L_+ (A_{k,c})  + i \mathcal L_- (B_{k,c}) \big],  \quad k=1,2;
\ee
with $F_1$, $G_1$ given by (\ref{F1G1})-(\ref{G1F1}), and
\bea\label{F2}
F_2 & :=& \frac{a''}{2 \tilde a^{m}  } y^2 Q_c^m + m\frac{a' }{a}Q_c^{m-1}y A_{1,c} - \frac 12 f_1yA_{1,c} - \frac 1\ve ( B_{1,c})_t - f_2 \Lambda B_{1,c} \nonu\\
& & \qquad +\frac 12 (m-1)\tilde a Q_c^{m-2} (mA_{1,c}^2 +B_{1,c}^2 )  - \frac{f_3(t)}{\tilde a} Q_c ,
\eea
and   
\bea\label{G2}
G_2 & :=  & \frac 1\ve (A_{1,c})_t + f_2 \Lambda A_{1,c}  +\frac{a' }{a}Q_c^{m-1} y B_{1,c} - \frac 12 f_1 y B_{1,c} \nonu \\
& & \qquad  + (m-1) \tilde a  Q_c^{m-2}  A_{1,c}B_{1,c}  -  \frac{f_4(t)}{\tilde a} Q_c'  ; 
\eea
Moreover, suppose that $(A_{k,c}, B_{k,c})$ satisfy (\ref{IP}) for $k=1$ and $2$. Then 
\be\label{4p5}
\|\ve^3 (\mathcal F_3(t, \cdot)  + \ve f(t) \mathcal F_c) e^{i\Theta}\|_{H^1(\R)}\leq K\ve^3 (e^{-\ve\mu |\rho(t)|} +  \ve),
\ee
uniformly in time.
\een

\end{prop}

\begin{proof}
See Appendix \ref{CDE}.
\end{proof}

\begin{rem}[Important notation]
In what follows, and in order to simplify the notation, we will assume that decomposition (\ref{S2}) is valid for all $2\leq m<5$, with the obvious modifications; in particular, we assume
$f_3 \equiv f_4 \equiv 0$ in (\ref{f3})-(\ref{f4}), for the case $2\leq m<3$. This simplification will be useful in Proposition \ref{prop:I} below, where a stability results is proved.
\end{rem}

From (\ref{Sm})-(\ref{S2}) we see that in order to improve the accuracy of $\tilde u$ as a solution of (\ref{aKdV}), we have to eliminate some terms $\mathcal F_k$.  The next subsection is devoted to the proof of the following assertion: we can choose dynamical parameters $(c,v,\rho, \ga)$ in the interval $[-T_\ve, T_\ve]$ in such a way that $\mathcal F_0(t, \cdot) \sim 0$.

\subsection{Existence for a simplified dynamical system}

Our first result concerns the existence of solutions of the differential system involving the evolution of velocity, scaling and phase parameters. This system is given by imposing the condition $\mathcal F_0 \equiv 0$.

\medskip

We we are able to prove existence and long time behavior for an approximated differential system given by $\mathcal F_0 \equiv 0$. Indeed, 

\begin{lem}[Existence of approximated dynamical parameters, case $2\leq m<5$]\label{ODE}~

Let $v_0>0, \la_0, a(s) $ be as in Theorem A and  (\ref{ahyp}). There exists a unique solution $(C,V,P,G)$ defined for all $t\geq -T_\ve$ with the same regularity than $a(\ve \cdot)$, of the following nonlinear system of differential equations (cf. (\ref{F0}))
\be\label{c}
\begin{cases}
\displaystyle{V'(t) = \ve f_1(C(t),U(t))}, &  V(-T_\ve) = v_0, \\
\displaystyle{C'(t) = \ve f_2(C(t),V(t), U(t)), } &  C(-T_\ve) =1, \\
\displaystyle{U'(t) = V(t),} &  U(-T_\ve) = -v_0 T_\ve, \\
\displaystyle{H'(t) =  - \frac 12V'(t) U(t),} & H(-T_\ve) = 0.
\end{cases}
\ee
In addition, for all $t\in [-T_\ve, T_\ve]$,
\ben
\item $C(t)$ is strictly increasing with $1 \leq C(t) \leq C(T_\ve)$ and $ C(T_\ve) =  c_\infty + O(\ve^{10})  =2^{\frac 4{5-m}} + O(\ve^{10})$.
\item $V(t)$ is strictly increasing with $v_0\leq V(t) \leq V(T_\ve)$, and where 
$$
V(T_\ve) =  v_\infty + O(\ve^{10})= (v_0^2 + 4\la_0 (c_\infty-1))^{1/2}+ O(\ve^{10}),
$$
$$
v_0T_\ve \leq  U(T_\ve) \leq (2v_\infty -v_0)T_\ve.
$$
\een
\end{lem}

\begin{rem}
Note that $(C(t), V(t), U(t), H(t))$ satisfy (\ref{aprio}) and therefore is a admissible set of parameters for $\tilde u.$
\end{rem}

\begin{proof}[Proof of Lemma \ref{ODE}]
The existence of a local solution of (\ref{c}) is consequence of the Cauchy-Lipschitz-Picard theorem. 

Now, in order to prove global existence of such a solution, we derive some a priori estimates. Note that from the first equation in (\ref{c}) we have $C$ strictly increasing in time with $C(t) \geq 1$, $t\in [-T_\ve, T_\ve]$.   
Moreover, after integration, we have
\be\label{c10}
C(t)  = \frac{a^{4/(5-m)}(\ve U(t))}{a^{4(5-m)}(-\ve v_0 T_\ve) } = a^{4/(5-m)}(\ve U(t))(1+ O(\ve^{10})).
\ee
Since $1\leq a \leq 2$, one has that $C$ is bounded and globally well defined with
$$
1\leq C(t) < c_\infty = 2^{\frac 4{5-m}}, \quad t\geq -T_\ve.
$$
Moreover, from the hypothesis on $a$ (cf. (\ref{ahyp})), it is easy to see that $C(T_\ve) = c_\infty + O(\ve^{10})$.

On the other hand, from the second equation in (\ref{c}), we have $V$ strictly increasing in time. Replacing (\ref{c10}), and after multiplication by $V(t)$, one has
$$
V(t) V'(t) = \frac 8{m+3} a^{\frac{m-1}{5-m}}(\ve U(t)) a'(\ve U(t)) V(t) a^{-\frac 4{5-m}}(-\ve v_0 T_\ve)).
$$
After integration in $[-T_\ve, t)$ we obtain $V^2(t) = v_0^2 + 4\la_0 [ C(t) - 1].$
This relation implies the global existence of $V$ and the uniform bound
$$
v_0 \leq V(t) < v_\infty :=  (v_0^2 + 4\la_0 (c_\infty -1))^{1/2}; \quad t\geq -T_\ve.
$$
In addition, one has $V(T_\ve) = v_\infty + O(\ve^{10}).$

\end{proof}

\bigskip

In order to construct a reasonable approximate solution describing the interaction we need to improve the error term $S[\tilde u]$ from Proposition \ref{prop:decomp} to the second order in $\ve$. This is the objective of the next subsection.

\subsection{Resolution of the first order system}\label{5}

In this paragraph we eliminate the term $\mathcal F_1$ in (\ref{Sm})-(\ref{S2}). According to Proposition \ref{prop:decomp}, this can be done for any $2\leq m<5$. We are then reduced to find $(A_{1,c}(t, y), B_{1,c}(t, y))$ satisfying, for all $(t,y)$,
$$
 (\Omega_1)  
\begin{cases}
\mathcal L_+ A_{1,c} (t, y) = F_1(t,y),\\
\mathcal L_- B_{1,c} (t, y) = G_1(t,y).
\end{cases}
$$
When solving problem $(\Omega_1)$, a key property will be the separability between the variables $t$ and $y$ on the source terms $F, G$. This is a surprising property, not necessarily true for more complicated nonlinearities others than pure powers. Let us recall that this property is also present in the case of generalized KdV equations, see \cite{Mu2}.

\subsubsection{Resolution of the linear problem $(\Omega_1)$}

Recall that from Proposition \ref{prop:decomp} and (\ref{F12}) the system $(\Omega_1)$ is more explicitly given by
\be\label{O1a}
 (\Omega_1)  
\begin{cases}
\displaystyle{ \mathcal{L}_+ A_{1,c}(t, y) =\frac {a'}{\tilde a^{m}} (\ve \rho( t)) yQ_c(y)(Q_c^{m-1}(y) -\frac{4c(t)}{m+3} )} ,\\
\displaystyle{ \mathcal{L}_- B_{1,c}(t, y) = \frac{1}{5-m} \frac {a'}{\tilde a^{m}} (\ve \rho( t)) v(t) (Q_c(y) + 2 y Q_c'(y)). }  
\end{cases}
\ee
This system is solvable, as shows the following

\medskip

\begin{lem}[Resolution of $(\Omega_1)$]\label{lem:omega1}~

Suppose $(c(t),v(t),\rho(t), \ga(t))$ satisfy (\ref{aprio}) for $t\in [-T_\ve, T_\ve]$. Then both right hand sides in (\ref{O1a}) are in $\mathcal Y$, and  there exists a unique solution $(A_{1,c}(t, y), B_{1,c}(t, y))$ of $(\Omega_1)$ satisfying (\ref{IP}), given by 
\bea\label{O1}
& & A_{1,c}(t, y) = \frac {a'(\ve \rho (t))}{(m+3) \tilde a^{m}(\ve  \rho(t)) c(t) } \big\{  c(t) y( y Q_c'(y) - Q_c(y))   + \xi Q_c'(y)  \big\},\nonu \\
& & B_{1,c} (t, y)  = - \frac {a'(\ve \rho( t))v(t)}{2(5-m) \tilde a^{m}(\ve \rho( t))c(t) } ( c(t) y^2  +\chi )Q_c(y). 
\eea
for $\xi, \chi$ given by\footnote{See Appendix \ref{AidQ} for more details on the computations.}
$$
\xi := - \frac{ \int_\R (\frac 12Q^2  + y^2 Q'^2)}{\int_\R Q'^2} = -\frac{m+7}{2(m-1)} + \chi, \quad \chi :=  -\frac{\int_\R y^2 Q^2}{\int_\R Q^2}.
$$
Moreover, $A_{1,c}(t, \cdot) $ is odd and $B_{1,c}(t, \cdot) $ is even, and satisfy
\bea\label{Or1}
& & \int_\R A_{1,c}(t,y) Q_c'(y)dy = \int_\R A_{1,c}(t,y) Q_c (y)dy= 0, \nonu \\
& & \int_\R B_{1,c} (t,y)Q_c ' (y)dy= \int_\R B_{1,c} (t,y)Q_c(y)dy =0.
\eea
\end{lem}

\medskip

\begin{proof}
From (\ref{aprio}) we have  $F_1, G_1\in \mathcal Y$.  Using Lemma \ref{surL}, we have the existence of the required solution provided the following two orthogonality conditions
$$
\int_\R F_1(t,y) Q_c'(y)dy = \int_\R G_{1}(t,y) Q_c(y)dy =0,
$$ 
are valid for all $t\in [-T_\ve, T_\ve]$. This is an easy computation. Indeed, up to a function of time, we have (cf. Appendix \ref{AidQ})
$$
\int_\R F_1 Q_c' = \int_\R yQ_c' Q_c(Q_c^{m-1} -\frac{4c}{m+3} ) =c^{2\theta+1}\big[   -\frac{1}{m+1} \int_\R Q^{m+1} + \frac{2}{m+3} \int_\R Q^2\big] =0.
$$
On the other hand,
$$
\int_\R G_1 Q_c =  \int_\R Q_c  (\frac{4c}{5-m} \Lambda Q_c -\frac 1{m-1} Q_c) = \big[ \frac{4\theta }{5-m} -\frac{1}{m-1} \big]c^{2\theta}\int_\R Q^2 =0.
$$
The fact that $A_{1,c}, B_{1,c}$ in (\ref{O1}) solve $(\Omega_1)$ is a simple verification. This finishes the proof.
\end{proof}

\begin{rem}
Note that (\ref{O1}) can be written as follows (we skip the dependence on $t$ of $v$ and $c$, and the dependence on $\ve \rho(t)$ of the function $a$)
\be\label{O1mod}
A_{1,c}(t, y) = \frac {a'}{\tilde a^{m}} c^{\frac 1{m-1}-\frac 12}A_{1}(\sqrt{c} y), \quad 
 B_{1,c} (t, y)  =  \frac {a' v}{\tilde a^{m}} c^{\frac 1{m-1}-1}B_{1}(\sqrt{c} y), 
\ee
for some $A_1, B_1 \in \mathcal Y$ not depending on $c$. More precisely, 
\be\label{AB1}
A_1 (y) :=\frac 1{m+3} (y(yQ' -Q) +\xi Q') , \quad B_{1}(y) := -\frac{1}{2(5-m)}(y^2 +\chi )Q .
\ee
\end{rem}

\subsection{Improvement of the approximate solution}

In this paragraph we consider the case $m\geq 3$. Our objective is to profit of the smoothness of the nonlinearity in this case (see Proposition \ref{prop:decomp}) to go beyond on the computations and solve one more linear system --denoted by $(\Omega_2)$--, and equivalent to solve $\mathcal F_2 \equiv 0$ in (\ref{S2}). As a consequence, the error term $S[\tilde u]$ in (\ref{S2}) will become or order $\sim \ve^3$ (see (\ref{4p5}) and Lemma \ref{CV} below.)

\subsubsection{Improved description of $F_2$ and $G_2$} 

Before solving the second order system $\mathcal F_2 \equiv 0$, we need to simplify some useless terms appearing in the description of $F_2$ and $G_2$ given in (\ref{F2})-(\ref{G2}). Indeed, note that terms like  $(A_{1,c})_t$ or $(B_{1,c})_t$ can be expressed by using the system of equations given by $\mathcal F_0 \equiv 0$, as we state now. 

\medskip

\begin{Cl}[Simplified description of $F_2$ and $G_2$]\label{11}~

We have
\be\label{F2t}
F_2(t,y)  :=  \tilde F_2(t,y) +  O_{H^1(\R)} (\ve |\rho' -v -\ve^2 f_4 | e^{-\ve\mu |\rho(t)|}) 
\ee
and
\be\label{G2t}
G_2(t,y)  =  \tilde G_2(t,y) +  O_{H^1(\R)} (\ve |\rho' -v -\ve^2 f_4| e^{-\ve\mu |\rho(t)|}) 
\ee
where $\tilde F_{2}(t, \cdot)$ is even and $\tilde G_2(t, \cdot )$ is odd. More precisely, they have the form
\be\label{F2new}
\tilde F_2(t,y)  = \frac{a''}{\tilde a^{m}} ( F_{2,c}^I (y)+ \frac{v^2}c F_{2,c}^{II}(y) ) + \frac{a'^2 }{\tilde a^{2m-1} } ( F_{2,c}^{III}(y) + \frac{v^2}{c}F_{2,c}^{IV}(y))   -  \frac{f_3}{\tilde a} Q_c(y),  
\ee
\be\label{G2new}
\tilde G_2(t,y)  :=   \frac{a'' v}{\tilde a^m} G_{2,c}^I(y) + \frac{a'^2 v}{\tilde a^{2m-1}} G_{2,c}^{II}(y) - \frac{f_4}{\tilde a} Q_c'(y) ; 
\ee
with $F_{2,c}^{(\cdot)} (y) = c^{\frac 1{m-1}} F_{2}^{(\cdot)}(\sqrt{c} y)$, and $G_{2,c}^{(\cdot)} (y)= c^{\frac 1{m-1} -\frac 12} G_{2}^{(\cdot)}(\sqrt{c} y)$. Finally, $F_2^{(\cdot)}$ and $G_2^{(\cdot)}$ are explicitly given below. 
\end{Cl}

\begin{rem}
Note that the small order terms in $|\rho' -v -\ve^2 f_4|$ above can be added to $\mathcal F_0$ in Proposition \ref{prop:decomp}. In what follows, we adopt this convention.
\end{rem}

\begin{proof}
First, in order to simplify the notation, let $\rho_1'(t) :=\rho'(t) -v(t)-\ve^2 f_4(t) $. Note that from (\ref{O1mod}) we have
\bee
(A_{1,c})_t & = & \frac {\ve \rho'}{\tilde a^{m} }\big[ a''  - \frac {m a'^2 }{(m-1) a }  \big] c^{\frac 1{m-1} -\frac 12} A_1(\sqrt{c} y) \\
& =& \frac {\ve v}{\tilde a^{m} }\big[ a''  - \frac {m a'^2 }{(m-1) a }  \big] c^{\frac 1{m-1} -\frac 12} A_1(\sqrt{c} y) + O_{H^1(\R)}(\ve |\rho_1'(t) | e^{-\ve\mu |\rho(t)|}), 
\eee 
and similarly, using (\ref{f1f2}),
\bee
(B_{1,c})_t & = &  \frac{\ve }{\tilde a^m} \big[  a'' v^2 + a' f_1   -\frac{m}{m-1} \frac{a'^2 v^2 }{a}  \big] c^{\frac 1{m-1} -1} B_1(\sqrt{c} y) + O(\ve |\rho'_1(t)| e^{-\ve\mu |\rho(t)|})\\
& =&  \frac{\ve }{\tilde a^m} \big[  a'' v^2 + \frac{8 a'^2 c}{(m+3)a}   -\frac{m}{m-1} \frac{a'^2 v^2 }{a}  \big] c^{\frac 1{m-1} -1} B_1(\sqrt{c} y) + O(\ve |\rho'_1(t)| e^{-\ve\mu |\rho(t)|}).
\eee
In addition, we replace (\ref{O1}) in $\tilde F_2$, $\tilde G_2$. We obtain
\bee
\tilde F_2(t,y) &  = & \frac{a''}{\tilde a^{m}} (\ve \rho(t)) \big[ F_{2,c}^I (y)+ \frac{v^2(t)}{c(t)} F_{2,c}^{II}(y) \big] \nonu\\
& &  + \frac{a'^2 }{\tilde a^{2m-1} } (\ve \rho(t))\big[ F_{2,c}^{III}(y) + \frac{v^2(t)}{c(t)}F_{2,c}^{IV}(y)\big]   -  \frac{f_3(t)}{\tilde a(\ve \rho(t))} Q_c(y),  
\eee
with $F_{2,c}^{(\cdot)} (y) = c^{\frac 1{m-1}} F_{2}^{(\cdot)}(\sqrt{c} y)$, and $F_2^I(y) := \frac 12 y^2 Q^m(y)$, $F_2^{II} (y):= -B_1(y),$
$$
F_{2}^{III}(y) := (mQ^{m-1}(y) -\frac 4{m+3}) yA_1(y) +\frac m2 (m-1) Q^{m-2}(y)A_1^2(y) -\frac{8}{(m+3)}B_1(y);
$$
$$
F_2^{IV}(y) := \frac 12 (m-1) Q^{m-2}(y) B_1^2(y) -\frac 2{5-m}yB_1' (y) - \frac {m-8}{5-m} B_1(y).
$$
Note that each term above is even and thus orthogonal to $Q'$.

On the other hand, 
$$
\tilde G_2 (y) :=   v(t) \big[ \frac{a'' }{\tilde a^m}(\ve \rho(t)) G_{2,c}^I(y) + \frac{a'^2 }{\tilde a^{2m-1}} (\ve \rho(t))G_{2,c}^{II}(y) \Big]   -\frac{f_4(t)}{\tilde a(\ve \rho(t))} Q_c'(y);
$$
with $G_{2,c}^{(\cdot)} (y)= c^{\frac 1{m-1} -\frac 12} G_{2}^{(\cdot)}(\sqrt{c} y)$ and $G_2^I (y):= A_1(y)$,
$$
G_2^{II}(y) := \frac{m-6}{5-m}A_1(y) +(Q^{m-1}(y) -\frac 4{m+3}) yB_1(y)  + \frac{2}{5-m}yA_1' (y)+ (m-1)Q^{m-2}(y)A_1(y)B_1(y).
$$
The proof is complete.
\end{proof}

\bigskip

\subsubsection{Resolution of a modified linear problem $(\Omega_2)$}

From Proposition \ref{prop:decomp}, more precisely (\ref{F12}), and the above Claim, we want to solve the modified system $(\tilde{\Omega}_2)$ given by
\be\label{O2}
(\tilde{\Omega}_2)
\begin{cases}
\mathcal{L}_+  A_{2,c}(t,y) = \tilde F_2(t,y),\\
\mathcal{L}_- B_{2,c}(t,y) = \tilde G_2(t,y),
\end{cases}
\ee
where $\tilde F_2$ and $\tilde G_2$ are given in (\ref{F2t})-(\ref{G2t}). The particular choice of $f_3(t)$ and $f_4(t)$ done in (\ref{f3})-(\ref{f4}) will allow us to find a unique solution of this linear system satisfying suitable orthogonality conditions. Recall the terms $F_2^{(\cdot)}$ and $G_2^{(\cdot)}$ introduced in the above Claim, and $\theta = \frac 1{m-1} -\frac 14.$

\begin{lem}[Resolution of $(\tilde \Omega_2)$]\label{lem:omega2}~

Suppose $m\geq 3$ and $f_3(t)$, $f_4(t)$ given by (\ref{f3})-(\ref{f4}), with 
\be\label{albe}
\al_{(\cdot)} := \frac 1{2\theta M[Q]} \int_\R \Lambda Q F_2^{(\cdot)}, \qquad \beta_{(\cdot)} := -\frac 1{M[Q]} \int_\R yQ G_2^{(\cdot)}.
\ee
There exists a unique solution $(A_{c,2}(t, y), B_{c,2}(t, y))$ of $(\tilde \Omega_2)$ satisfying (\ref{IP}). In addition, $A_{2,c}$ is even and $B_{2,c}$ is odd and satisfy the following decomposition:
\be\label{A2}
A_{2,c}(t, y) = \frac{a''}{\tilde a^{m}} ( A_{2,c}^I(y) + \frac {v^2}c A_{2,c}^{II}(y)  ) + \frac{a'^2}{\tilde a^{2m-1}} ( A_{2,c}^{III}(y) + \frac{v^2}c A_{2,c}^{IV}(y))  + \frac{f_3}{\tilde a} \Lambda Q_c,
\ee
with $A_{2,c}^{(\cdot)} (y) = c^{\frac 1{m-1} -1} A_{2}^{(\cdot)}(\sqrt{c} y)$, $A_{2}^{(\cdot)} $ even, and
\be\label{B2}
B_{2,c}(t, y) = \frac{a''v}{\tilde a^{m}} B_{2,c}^I(y) +\frac{a'^2v}{\tilde a^{2m-1}} B_{2,c}^{II}(y) +  \frac{f_4}{2\tilde a} y Q_c,
\ee
with $B_{2,c}^{(\cdot)} (y) = c^{\frac 1{m-1} -\frac 32} B_{2}^{(\cdot)}(\sqrt{c} y)$ and $B_{2}^{(\cdot)} $ odd.
Moreover, both $A_{2,c}$ and $B_{2,c}$ satisfy
\bea\label{Or2}
& & \int_\R A_{2,c}(t,y) Q_c'(y) dy =\int_\R A_{2,c} (t,y)Q_c(y) dy  = 0, \nonu \\
& &  \int_\R B_{2,c} (t,y)Q_c' (y) dy=\int_\R B_{2,c}(t,y) Q_c (y) dy=0.
\eea
\end{lem}

\medskip

\begin{rem}\label{defe}
Note that thanks to the introduction of $f_3(t)$ and $f_4(t)$ in (\ref{f3})-(\ref{f4}), and from (\ref{defW}), (\ref{Or1}) and (\ref{Or2}) one has
\be\label{Or3}
\int_\R w(t,x) Q_c(y)e^{-i\Theta} = \int_\R w(t,x) Q_c'(y)e^{-i\Theta} =0. 
\ee
Let us remark that $f_3(t)$ and $f_4(t)$ formally represent the {\bf lack of symmetry}, or the {\bf defect} of the soliton solution, with respect to the pure soliton solution considered in Definition \ref{PSS}. Indeed, a formal integration of the dynamical system present in (\ref{S2}) and (\ref{ODE}) gives us
$$
\rho(t) \sim U(t) + \hbox{ non-zero order $\ve$ correction,}
$$
where this last correction comes from the integral of the term $\ve^2 f_4(t)\neq 0$. Therefore, the solution $u(t)$ should differ from a hypothetic pure soliton solution, at the formal level, by an $O(\ve)$ correction term in their trajectories. 
Finally, let us recall that in this case, and similarly to \cite{MMfinal}, the main order in the defect of the soliton solution is present on the trajectory (and phase), rather than in the scaling, as in \cite{MMcol1,MMMcol,Mu1}.
\end{rem}

\begin{rem}
The exact value of $\al_{(\cdot)}$ and $\beta_{(\cdot)}$ can be computed explicitly but their are not necessary  for our purposes. Nevertheless, it is simple to see that from Claim \ref{11}, Lemma \ref{surL}, (\ref{AB1}) and Appendix \ref{AidQ} one has
$$
\al_I = \frac{1}{2(m+1)M[Q]}\int_\R y^2 Q^{m+1} >0; \quad \al_{II}=-\frac{(m-1)}{2M[Q]}\int_\R y^2 Q^2<0;
$$
and
$$
\beta_I = \frac{1}{(m+3)M[Q]} \big[ \frac 52 \int_\R y^2 Q^2 + \frac \xi 2 \int_\R Q^2\big] = \frac{-1}{m+3}(\frac{7+m}{2(m-1)} +4\chi)>0,
$$
for $m\in [3, 5)$.
\end{rem}

\begin{proof}
Note that $\tilde F_2(t, \cdot )$ is even and $\tilde G_2(t, \cdot )$ is an odd function, and both functions are in $\mathcal Y$, uniformly in time. Therefore $\tilde F_2$ is orthogonal to $Q_c'$, and $\tilde G_2$ is orthogonal to $Q_c$. From Lemma \ref{surL} part (3), we obtain the conclusion. 

Moreover, note that $\mathcal L_+ \Lambda Q_c =-Q_c$, $\mathcal L_- (yQ_c) =-2Q_c'$. Thanks to the choice of $f_3$ and $f_4$ one has (see Appendix \ref{AidQ} for more details on the computations) 
\bee
\int_{\R} A_{2,c} Q_c & =&  -\int_\R \mathcal L_+ A_{2,c} \Lambda Q_c = -\int_\R \tilde F_{2} \Lambda Q_c \\
& =&-\frac{a''}{\tilde a^{m}}  \int_\R \Lambda Q_c( F_{2,c}^I  + \frac{v^2}c F_{2,c}^{II}) - \frac{a'^2 }{\tilde a^{2m-1} } \int_\R \Lambda Q_c ( F_{2,c}^{III} + \frac{v^2}{c}F_{2,c}^{IV})   +  \frac{f_3}{\tilde a} \int_\R \Lambda Q_c Q_c  \\
& =& - \frac{c^{2\theta -1}}{\tilde a}\Big[  \frac{a''}{a} ( \al_I  +\al_{II} \frac{v^2}c ) + \frac{a'^2 }{a^2 }  ( \al_{III} +  \al_{IV}\frac{v^2}{c} )   -  f_3 \Big]  \\
& =& 0.
\eee
Similarly,
\bee
\int_{\R} B_{2,c} Q_c' &  = & -\frac 12 \int_\R \mathcal L_- B_{2,c}  y Q_c = -\frac 12 \int_\R \tilde G_{2} yQ_c  \\
& =&   - \frac{a'' v}{2\tilde a^m} \int_\R yQ_c G_{2,c}^I - \frac{a'^2 v}{2\tilde a^{2m-1}} \int_\R yQ_c G_{2,c}^{II} + \frac{f_4}{2\tilde a} \int_\R yQ_c Q_c'  \\
& =&   \frac{c^{2\theta}}{2\tilde a} \Big[  (\beta_I \frac{a'' }{ a}  + \beta_{II} \frac{a'^2}{ a^2} )\frac{v}{c}  - f_4  \Big] \\
& =& 0.
\eee
The proof is complete.
\end{proof}

From Proposition \ref{prop:decomp} and the singular behavior of the nonlinearity $|z|^{m-1} z$ around $z=0$ for $2\leq m <4$, $m\neq 3$, we cannot perform a new expansion to improve our estimates. We stop here the search of an approximate solution for the case $3\leq m<5$.

\subsection{Error estimates}

As a consequence of Proposition \ref{prop:decomp} and Lemma \ref{ODE} and Lemma \ref{lem:omega1}, we have the following estimates on the error associated to the approximate solution $\tilde u$. Recall the definition of $\tilde S[\tilde u]$ and $p_m$ given in (\ref{Sm})-(\ref{S2}) and (\ref{pm}) respectively.

\begin{lem}[Estimation of the error $\tilde S{[}\tilde u{]}$]\label{CV}~

There exist constants $\ve_0, K>0$ such that for all $0<\ve <\ve_0$ the following holds. The error associated to the function $\tilde u$ satisfies  
\be\label{SH2}
\| \tilde S[\tilde u](t) \|_{H^1(\R)} \leq K\ve^{p_m +1}(\ve + e^{-\ve\mu |\rho(t)|}),
\ee
and the following integral estimate holds
\be\label{integra}
\int_\R\| \tilde S[\tilde u](t) \|_{H^1(\R)} dt \leq K\ve^{p_m}.
\ee
\end{lem}

\begin{proof}
First we prove the case $2\leq m< 3$. Here $p_m =1$. From Proposition \ref{prop:decomp} and Lemma \ref{lem:omega1} we have
$$
\tilde S[\tilde u] = \ve^2 [\mathcal F_2(t, y) + \ve f(t) \mathcal F_c(y)]e^{i\Theta},
$$
From estimate (\ref{tSH2}) we have (\ref{SH2}) in this case.

Let us consider the case $m\geq3$, with $p_m=2$. Here we invoke Proposition \ref{prop:decomp} and Lemmas \ref{lem:omega1} and \ref{lem:omega2} to get
$$
\tilde S[\tilde u] = \ve^4 [\mathcal F_4(t, y) + \ve f(t) \mathcal F_c(y)]e^{i\Theta},
$$
From (\ref{4p5}), the rest of the proof and (\ref{integra}) are direct since from (\ref{aprio}) one has $\rho'(t) \geq \frac 12 v_0>0$.
\end{proof}

\medskip

\subsection{Recomposition of the solution}

In this subsection we will show that $\tilde u$ at time $t=- T_\ve$ behaves as a modulated soliton. We begin with some $H^1$-estimates.
\smallskip

\begin{lem}[First estimates on $\tilde u$]~

Suppose $0<\ve<\ve_0$ small enough, and $(c,v,\rho, \ga)$ satisfying (\ref{aprio}). Then the following auxiliary estimates hold.
\ben
\item \emph{Decay away from zero}. Suppose $f_c=f_c(y)\in \mathcal Y$. Then there exist $K,\mu>0$ constants such that for all $t\in [-T_\ve, T_\ve]$
\be\label{Est1}
\norm{a'(\ve x) f_c(y)}_{H^1(\R)} \leq K e^{-\mu\ve|\rho(t)|}.
\ee
\item \emph{Almost soliton-solution}. The following estimates hold for all $t\in [-T_\ve, T_\ve]$.
\be\label{Est2}
\tilde u _{xx} - (c(t)+\frac 14v^2(t)) \tilde u  + a_\ve |\tilde u|^{m-1} \tilde u -iv(t) \tilde u_x = O_{H^1(\R)}(\ve e^{-\mu\ve |\rho(t)|}),
\ee
and
\be\label{Est2a}
i\tilde u_{t} + iv\tilde u_x  + (c(t)+\frac 14 v^2(t)) \tilde u = O_{H^1(\R)}(\ve e^{-\mu\ve |\rho(t)|}).
\ee
\een
\end{lem}

\begin{proof}
(\ref{Est1}) is a classical result, see \cite{Mu2} Lemma 4.7 for a complete proof. On the other hand, to prove (\ref{Est2}), note that after some simplifications, and by using (\ref{defv})-(\ref{IP}) with  $\|w(t)\|_{H^1(\R)} \leq K\ve e^{-\ve\mu |\rho(t)|}$, we have
\bee
& & \tilde u _{xx} - (c+\frac 14v^2) \tilde u  + a_\ve |\tilde u|^{m-1} \tilde u -iv \tilde u_x \\
& & \qquad =  \frac 1{\tilde a}(Q_c''  -cQ_c  + Q_c^m)e^{i\Theta}+ (\frac{a_\ve(x)}{a_\ve(\rho)} -1)\frac{Q_c^m}{\tilde a} e^{i\Theta} + O_{H^1(\R)}( \ve e^{-\mu\ve |\rho(t)|}) \\
& & \qquad =  O_{H^1(\R)}(\ve e^{-\mu\ve |\rho(t)|}).
\eee
Let us prove (\ref{Est2a}). From the definition of $S[\tilde u]$ and estimate (\ref{SH2}),
\bee
i\tilde u_t + iv\tilde u_x + (c+\frac 14 v^2)\tilde u & = & S[\tilde u]  - \{ \tilde u _{xx} - (c+\frac 14v^2) \tilde u  + a_\ve |\tilde u|^{m-1} \tilde u -iv \tilde u_x \}    \\
& =& O_{H^1(\R)}(\ve e^{-\mu\ve |\rho(t)|}).
\eee
The proof is complete.
\end{proof}

The next result describes the behavior of the almost solution $\tilde u$ at the endpoint $t=-T_\ve$.

\begin{lem}[Behavior at $t = - T_\ve$]\label{atpmT}~

There exist constants $K,\ve_0>0$ such that for every $0<\ve <\ve_0$ the following holds. Let $\hat u(t) := \tilde u (t; C(t),V(t),  U(t), H(t))$ be the approximate solution constructed in Section 3 Step 2 with modulation parameters $(C,V,U,H)$ given by Lemma \ref{ODE}. Then one has
\be\label{mTe}
\| \hat u(-T_\ve) -  Q (\cdot + v_0T_\ve ) e^{\frac i2 (\cdot)v_0} e^{i\ga_{-1}} \|_{H^1(\R)} \leq K \ve^{10},
\ee
with 
\be\label{gam1}
\ga_{-1} := - \int_{-T_\ve}^0 C(s) ds  + \frac 14 \int_{-T_\ve}^0 V^2(s) ds.
\ee
\end{lem}

\begin{proof}
By definition,
$$
\hat u (-T_\ve) - Q(\cdot +v_0 T_\ve) e^{\frac i2 (\cdot)v_0} e^{i\ga_{-1}} = \tilde R(-T_\ve) - Q(\cdot +v_0 T_\ve) e^{\frac i2 (\cdot)v_0} e^{i\ga_{-1}} + w(-T_\ve).
$$
From (\ref{defW}), (\ref{IP}) and Lemma \ref{ODE} we have
$$
\| w(\pm T_\ve) \|_{H^1(\R)}   \leq K \ve e^{-\mu \ve^{-\frac 1{100}}} \leq K \ve ^{10},
$$
for $\ve$ small enough. Since from (\ref{c1}) $U(-T_\ve) = v_0 T_\ve $, $V(-T_\ve) =v_0$ and $C(- T_\ve) =1$, we have
$$
\| \tilde R(-T_\ve) - Q(\cdot +v_0 T_\ve) e^{\frac i2 (\cdot)v_0} e^{i\ga_{-1}} \|_{H^1(\R)} \leq K \ve^{10}, 
$$
as desired. 
\end{proof}

Resuming, we have constructed an approximate solution $\tilde u$ formally describing the interaction soliton-potential. In the next subsection we will show that a suitable modification of the the \emph{solution} $u$ constructed in Theorem \ref{Tm1} actually behaves like $\tilde u$ inside the interaction region $[-T_\ve, T_\ve]$.

\bigskip


\noindent
{\bf Step 3. Stability results.} 

\medskip

In this paragraph our objective is to prove that the approximate solution $\tilde u(t)$ describes the dynamics of the soliton in the interaction interval $[-T_\ve, T_\ve]$. We will prove the following result, cf. Propositions 5.1 in \cite{Mu2} for a similar result for a gKdV equation.

\begin{prop}[Exact solution close to the approximate solution $\tilde u$]\label{prop:I}~

Let $2\leq m<5$, $p_m$ defined in (\ref{pm}). There exists $\ve_0>0$ such that the following holds for any $0<\ve <\ve_0$.
Suppose that for $\hat u (-T_\ve)$ as defined in Lemma \ref{atpmT} one has
\be\label{hypINTa}
\| u(-T_\ve) - \hat u(-T_\ve) \|_{H^1(\R)}\leq K \ve^{p_m},
\ee
with $u=u(t)$ the $H^1(\R)$ solution of (\ref{aKdV}) constructed in Proposition \ref{Tm1}. Then there exist $K_0=K_0(m,K)>0$ and  $C^1$-functions $c,v, \rho, \ga :[-T_\ve,T_\ve ]\rightarrow \R$ such that, for all $t\in [-T_\ve,T_\ve]$,
\be\label{INT41a}
\|u(t) -\tilde u(t; c(t), v(t), \rho(t), \ga(t)) \|_{H^1(\R)} \leq K_0 \ve^{p_m},
\ee
and
\be\label{garho}
 |\rho'(t) -v(t) -\ve^2 f_4(t) |+ |\ga'(t) -\frac 12 v'(t)\rho(t) -\ve^2 f_3(t) |\leq K_0 \ve^{p_m},
\ee
\be\label{cc}
|v'(t) - \ve f_1(t)|+|c'(t) - \ve f_2(t)| \leq K_0 ( \ve^{2p_m} + \ve^{p_m+1} ).
\ee
\end{prop}

\noindent
{\bf Proof of Proposition \ref{prop:I}}.
Let $K^*>1$ be a constant to be fixed later. Since $\|u(-T_\ve)-\hat u(-T_\ve)\|_{H^1(\R)}\leq K \ve^{p_m} $, by continuity in time in $H^1(\mathbb{R})$, there exists $-T_\ve<T^*\leq T_\ve$ with
\bee
    T^*& := & \sup\big\{T\in [-T_\ve,T_\ve], \hbox{ such that for all }  t\in [-T_\ve,T], \hbox{ there exists }  \rho(t), \ga (t) \in \mathbb{R},  \\
    & & \quad \quad \quad  \hbox{ with }  \|  u(t)  -\tilde u(t; C(t), V(t), \rho(t), \ga(t))\|_{H^1(\R)}\leq K^* \ve^{p_m } \big\}.
\eee
The objective is to prove that $T^*=T_\ve$ for $K^*$ large enough. To achieve this, we argue by contradiction, assuming that $T^*<T_\ve $ and reaching a contradiction with the definition of $T^*$ by proving some independent estimates for $  \|  u(t)  -\tilde u(t; C(t), V(t), \rho(t), \ga(t))\|_{H^1(\R)}$ on $[-T_\ve,T^*]$, for a special modulation parameters $\rho(t), \ga(t)$.

\subsubsection{Modulation} 
By using the Implicit function theorem we will construct some modulation parameters and estimate their variation in time:
 
\begin{lem}[Modulation in time]\label{DEFZ}~

Assume $0<\ve<\ve_0(K^*)$ small enough.
There exist \emph{unique} $C^1$ functions $ c(t),v(t), \rho(t), \ga(t)$ such that, for all $t\in [-T_\ve,T^*]$, the function
\be\label{defz}
z(t) :=u(t) -\tilde u(t; c(t),v(t), \rho(t), \ga(t)), 
\ee
satisfies
\be\label{defz1}
\int_\R \bar z(t,x) Q_c(y)e^{i\Theta}  dx = \int_\R \bar z(t,x)Q_c'(y)e^{i\Theta}dx =0, 
\ee
and
\bea
& & |\rho(-T_\ve) - U(-T_\ve)| + |\ga(-T_\ve)-H(-T_\ve)| \nonu \\ 
&  & \qquad +|c(-T_\ve)-C(-T_\ve)| +|v(-T_\ve) -V(-T_\ve)| +\|z(-T_\ve)\|_{H^1(\R)}\leq K \ve^{p_m }.\label{minTe}
\eea
Moreover, we have,  for all $t\in [-T_\ve,T^*]$,
\be\label{TRANS3}
\|z(t)\|_{H^1(\R)} + |c(t) -C(t) | + |v(t) -V(t)| \leq  K K^* \ve^{p_m }.
\ee
In addition, $z(t)$ satisfies the following equation
\bea\label{Eqz1}
& &  iz_t +  z_{xx}  + a_\ve [ |\tilde u +z|^{m-1}(\tilde u +z) - |\tilde u|^{m-1} \tilde u] + \tilde S[\tilde u] + i (c' -\ve f_2) \partial_c \tilde u \nonumber  \\
& & \qquad -\frac12(v' -\ve f_1) y\tilde u   + i(\rho' -v - \ve^2 f_4) \partial_\rho \tilde u  - (\ga' + \frac 12 v' \rho -\ve^2 f_3 )\tilde u =0.
\eea
Finally, there exist $K,\mu>0$ independent of $K^*$ such that for every $t\in [-T_\ve, T^*]$
\bea\label{rho1}
& & |\rho'(t) -v(t) - \ve^2 f_4(t)| + |\ga'(t) -\frac 12 v'(t) \rho(t) -\ve^2 f_3(t)| \leq \nonumber \\
& & \qquad  \leq K\Big[ \|z(t)\|_{L^2(\R)} +  \ve e^{-\mu\ve|\rho(t)| } \|z(t)\|_{L^2(\R)} +  \|z(t)\|_{L^2(\R)}^2 + \| \tilde S[\tilde u](t)\|_{L^2(\R)}\Big], 
\eea
and
\be\label{c1}
|v'(t) -  \ve f_1(t) | +|c'(t) -  \ve f_2(t) | \leq K\Big[  \ve e^{-\mu\ve|\rho(t)| } \|z(t)\|_{L^2(\R)} +  \|z(t)\|_{L^2(\R)}^2 + \|\tilde S[\tilde u](t)\|_{L^2(\R)}\Big].
\ee
\end{lem}

\begin{proof}
The proof of (\ref{defz})-(\ref{TRANS3}) is by now well-know and it is a consequence of an Implicit Function Theorem application. See e.g. \cite{MMcol1} for a detailed proof.
On the other hand, the proof of (\ref{Eqz1}) follows after a simple calculation using (\ref{aKdV}).

\smallskip

The proof of  (\ref{rho1}) and (\ref{c1}) follow from (\ref{defz1})-(\ref{Eqz1}) after taking time derivative and replacing $z_t$. We skip the details.
\end{proof}

\medskip

\subsubsection{Improvement of (\ref{aprio})}

In this paragraph we prove that the parameters $(c(t),v(t),\rho(t),\ga(t))$ constructed in Lemma \ref{DEFZ} satisfy the assumptions required in (\ref{aprio}). Recall that under these hypotheses, all the results of Section 3, Step 2 are valid. In particular, one has (\ref{SH2}).

First of all, from (\ref{minTe}) we have
\be\label{Imp1}
\frac 9{10} \leq c(t)\leq \frac{11}{10}c_\infty <2^5, \qquad \frac 9{10} v_0 \leq v(t) < v_0 + 2^5. 
\ee
On the other hand, from (\ref{rho1}) we have for $\ve$ small,
\be\label{Imp2}
|\rho'(t) -v(t)| \leq K^*\ve^{p_m} \leq \frac{v_0}{100}.
\ee
We are done.

\subsubsection{Energy functional for $z$}\label{EFz}

Consider the $H^1(\R)$ functional
\bea\label{F}
\mathcal F(t) & := & \frac 12 \int_\R |z_x|^2 +\frac12 (c+\frac 14 v^2) \int_\R |z|^2- \frac 12 v \ima \int_\R \bar z z_x \nonumber  \\
& & \quad - \frac{1}{m+1}\int_\R a_\ve(x) [|\tilde u+ z|^{m+1} -|\tilde u|^{m+1} - (m+1)|\tilde u|^{m-1}\re \{ \tilde u\bar z\}] .
\eea

\begin{lem}[Modified coercivity for $\mathcal F$]\label{Coer2}~

There exist $K,\nu_0>0$, independent of $K^*$ and $\ve$ such that for every $t\in [-T_\ve, T_\ve]$
$$
\mathcal F(t) \geq \nu_0 \|z(t)\|_{H^1(\R)}^2 - K \ve (e^{-\mu\ve |\rho(t)|} + 1)\|z(t)\|_{L^2(\R)}^2 - K  \|z(t)\|_{L^2(\R)}^3. 
$$ 
In particular, for $\ve$ small enough, one has
$$
\mathcal F(t) \geq \frac 9{10} \nu_0 \|z(t)\|_{H^1(\R)}^2 . 
$$ 
\end{lem}

\begin{proof}
The proof is similar to the proof of Lemma 5.5 in \cite{Mu2}. First of all it is easy to see that 
\bee
\mathcal F(t) & = & \frac 12 \int_\R |z_x|^2 + \frac 12  (c +\frac 14 v^2 ) \int_\R  |z|^2  - \frac 12v\ima \int_\R z_x\bar z  \\
& & -  \int_\R \frac{a(\ve x) }{a(\ve \rho)} Q_c^{m-1}(y) [  |z|^2 + (m-1) [\re( e^{i\Theta} \bar z)]^2 ]  + O(\ve \|z(t)\|_{H^1(\R)}^2+ \|z(t)\|_{H^1(\R)}^3)
\eee
On the other hand, it is clear that $\abs{\ve \frac { a'(\ve \rho)}{ a (\ve \rho)} \int_\R   y Q_c^{m-1} |z|^2}\leq K  \ve e^{-\mu\ve |\rho(t)|} \|z(t)\|_{L^2(\R)}^2.$ Thus we have
\bea
 \mathcal F(t)&  = &  \frac 12 \int_\R |z_x|^2 + \frac 12  (c +\frac 14 v^2 ) \int_\R  |z|^2  - \frac 12v\ima \int_\R z_x\bar z\nonu \\
 & &   -  \int_\R Q_c^{m-1}(y) [  |z|^2 + (m-1) [\re( e^{i\Theta} \bar z)]^2 ]   \nonumber \\
& &  + O( \ve(1+ e^{-\ve\mu |\rho(t)|}) \|z(t)\|_{H^1(\R)}^2  + \|z(t)\|_{H^1(\R)}^3).  \label{Un}
\eea
Finally, from Lemma \ref{surL} and (\ref{defz1}), we have the existence of constants $K, \nu_0>0$ such that for all $t\in [-T_\ve, T^*]$
$$
(\ref{Un}) \geq \nu_0 \|z(t)\|_{H^1(\R)}^2 -K\ve (1+ e^{-\ve\mu |\rho(t)|}) \|z(t)\|_{H^1(\R)}^2-K\|z(t)\|_{H^1(\R)}^3.
$$
The proof is now complete.
\end{proof}

Now we use a coercivity argument, similar to Lemma 5.6 in \cite{Mu2} to obtain independent estimates for $\mathcal F(T^*)$.  

\begin{lem}[Estimates on $\mathcal F(T^*)$]\label{Ka}~

The following properties hold for any $t\in [-T_\ve, T^*]$.
\ben
\item First time derivative. 
\bea
& & \mathcal F'(t)  = \nonu \\
&   & = \ima \int_\R \overline{iz_t} \big\{ z_{xx} - (c+\frac 14 v^2) z + a_\ve [|\tilde u+z|^{m-1}(\tilde u+z)  - |\tilde u|^{m-1}\tilde u]  -iv z_x\big\} \label{Fp} \\
& & \quad  + \ima \int_\R a_\ve \overline{i\tilde u_t} [ |\tilde u+z|^{m-1}(\tilde u + z)  - |\tilde u|^{m-1}\tilde u -  \frac 12 (m+1) |\tilde u|^{m-1} z - \frac 12(m-1) | \tilde u|^{m-3} \tilde u^2 \bar z ] \nonumber  \\
 & & \quad + ( c' +  \frac 14 v' v) \int_\R |z|^2 -  \frac 12 v' \ima\int_\R \bar z z_x. 
\eea
\item Integration in time. There exist constants $K,\mu>0$ such that
\be\label{IntF}
\mathcal F(t) -\mathcal F(-T_\ve)  \leq   K(K^*)^4\ve^{4p_m-1  -\frac 1{100}}  + KK^* \ve^{2p_m } + K\int_{-T_\ve}^t  \ve e^{-\ve\mu |\rho(s)|}\|z(s)\|_{H^1(\R)}^2 ds. 
\ee
\een
\end{lem}

\begin{proof}
First of all, (\ref{Fp}) follows after derivation in time.
Let us  consider (\ref{IntF}). In order to simplify the computations, let  $v_1' := \frac 12 (v' -\ve f_1)$, $c_1' := c' -\ve f_2$, $\rho_1' := \rho' -v -\ve^2 f_4$ and $\ga_1' := \ga' +\frac 12 v' \rho -\ve^2 f_3$. In addition, consider
$$
L[z]:=  z_{xx} - (c+\frac 14 v^2) z   + \frac 12 a_\ve [ (m+1) |\tilde u |^{m-1} z +(m-1)|\tilde u|^{m-3}\tilde u^2\bar z] -ivz_x,
$$
and
$$
N[z] :=a_\ve(x)[  |\tilde u+z|^{m-1}(\tilde u + z)  - |\tilde u|^{m-1}\tilde u -  \frac 12 (m+1) |\tilde u|^{m-1} z - \frac 12(m-1) | \tilde u|^{m-3} \tilde u^2 \bar z].
$$
Replacing (\ref{Eqz1}) in (\ref{Fp}) we get
\bea
& &\!\!\!\!\!\! \mathcal F'(t) =\nonu \\
&  &\!\!\!\!\!\! =  \ima \int_\R  \{ (c+\frac 14 v^2) z + iv z_x\}  \{  \frac 12 a_\ve[ (m+1) |\tilde u |^{m-1} \bar z +(m-1)|\tilde u|^{m-3}\overline{\tilde u}^2 z] +\overline{N[z]}  \}  \label{Fp0} \\
& & -  \ima \int_\R \overline{\tilde S[\tilde u]}  L[z]   -   \ima \int_\R\overline{  L[z]} [ \ga_1'  \tilde u+ v_1' y \tilde u - c'_1 i\partial_c \tilde u  - \rho_1' i\partial_\rho \tilde u]  \nonu   \\
& &  -  \ima \int_\R \overline{ N[z]} [  i\tilde u_t -\tilde S[\tilde u] + \ga_1' \tilde u + v_1' y \tilde u - c'_1  i\partial_c \tilde u - \rho_1'i\partial_\rho \tilde u ] \label{Fp4}\\
& & +  ( c' +  \frac 14 v' v) \int_\R |z|^2 -  \frac 12 v' \ima\int_\R \bar z z_x. \label{Fp4b}
\eea 
From (\ref{c1}), 
$$
|(\ref{Fp4b})|\leq K \ve e^{-\ve\mu |\rho(t)|} \|z(t)\|_{H^1(\R)}^2 + K \|z(t)\|_{H^1(\R)}^4.
$$
On the other hand, note that $L[z] = L_Q [z] + O(\ve e^{-\ve\mu|\rho(t)|}\|z(t)\|_{H^1(\R)})$, with
$$
L_Q[z] :=  z_{xx} - (c+\frac 14 v^2) z   + \frac 12  Q_c^{m-1}(y) [ (m+1)  z +(m-1) e^{2i\Theta} \bar z] -ivz_x.
$$
Therefore a simple computation using (\ref{defz1}) and (\ref{c1})-(\ref{rho1}) gives us
$$
\abs{  \ima \int_\R\overline{  L[z]} [ \ga_1'  \tilde u+ v_1' y \tilde u - c'_1 i\partial_c \tilde u  - \rho_1' i\partial_\rho \tilde u]} \leq  K \ve e^{-\ve\mu|\rho(t)|} \|z(t)\|_{H^1(\R)}^2.
$$
We also have
\be\label{E1b}
\abs{ \ima \int_\R \overline{\tilde S[\tilde u]}  L[z]  } \leq K \|z(t)\|_{L^2(\R)} \|\tilde S[\tilde u](t)\|_{H^1(\R)}.
\ee
Next, note that from (\ref{2.2bis}), Proposition \ref{prop:decomp} and (\ref{Est2}) one has
$$
 i\tilde u_t -\tilde S[\tilde u] + \ga_1' \tilde u + v_1' y \tilde u - c'_1  i\partial_c \tilde u - \rho_1'i\partial_\rho \tilde u  =   -(c+\frac 14 v^2)\tilde u -  iv \tilde u_x + O_{H^1(\R)}(\ve e^{-\ve\mu|\rho(t)|}),
$$
therefore we obtain
\bee
& &(\ref{Fp0})+(\ref{Fp4})= \\
&  & = \ima \int_\R   \{ (c+\frac 14 v^2) z + iv z_x\}  \{ \frac 12 a_\ve[ (m+1) |\tilde u |^{m-1} \bar z +(m-1)|\tilde u|^{m-3}\overline{\tilde u}^2 z] +\overline{N[z]} \} \\
& & + \ima  \int_\R \{ (c+\frac 14 v^2) \tilde u + iv \tilde u_x\} \overline{N[z]}  +  O(\ve e^{-\ve\mu|\rho(t)|}\|z(t)\|_{H^1(\R)}^2).
\eee
Now we claim that
\be\label{Mir1}
 \ima \int_\R   \frac 12 a_\ve z [ (m+1) |\tilde u |^{m-1} \bar z +(m-1)|\tilde u|^{m-3}\overline{\tilde u}^2 z] +  \ima  \int_\R \tilde u  \overline{N[z]} = \ima \int_\R \bar zN[z],
\ee
and
\bea\label{Mir2}
& &  \ima \int_\R   \frac 12 a_\ve iz_x [ (m+1) |\tilde u |^{m-1} \bar z +(m-1)|\tilde u|^{m-3}\overline{\tilde u}^2 z] +  \ima  \int_\R i\tilde u_x  \overline{N[z]} = \nonumber  \\
& & \qquad =  -\ima \int_\R i\bar z_x  N[z] + O(\ve e^{-\ve\mu|\rho(t)|} \|z(t)\|_{H^1(\R)}^2).
\eea
Assuming these two identities, we have finally
$$
(\ref{Fp0})+(\ref{Fp4})  =   O(\ve e^{-\ve\mu|\rho(t)|}\|z(t)\|_{H^1(\R)}^2).
$$
Therefore
\bee
\abs{\mathcal F'(t)} & \leq&   K \ve e^{-\ve\mu |\rho(t)|}\|z(t)\|_{L^2(\R)}^2 +  K \|\tilde S[\tilde u](t)\|_{H^1(\R)} \|z(t)\|_{L^2(\R)} + K \|z(t)\|_{L^2(\R)}^4 \\
&\leq & K \ve e^{-\ve\mu |\rho(t)|}\|z(t)\|_{L^2(\R)}^2 +  K K^* e^{-\mu\ve|\rho(t)|} \ve^{1+ 2p_m} + K (K^*)^4 \ve^{4p_m}. 
\eee
After integration between $-T_\ve$ and $t$  we obtain (\ref{IntF}). 

\smallskip

Let us prove (\ref{Mir1}) and (\ref{Mir2}). First of all, note that (\ref{Mir1}) is consequence of the identity
$$
 \ima  \big\{ \frac 12 a_\ve  z [ (m+1) |\tilde u |^{m-1} \bar z +(m-1)|\tilde u|^{m-3}\overline{\tilde u}^2 z] - \overline{\tilde u}N[z] \big\} = \ima\{ \bar zN[z] \}.
$$
On the other hand, (\ref{Mir2}) is an easy consequence of the following identity
\bee
& & \re \Big\{ \frac 12 a_\ve z_x [ (m+1) |\tilde u |^{m-1} \bar z +(m-1)|\tilde u|^{m-3}\overline{\tilde u}^2 z] +  \overline{\tilde u}_x  N[z] \Big\} =\\
& &   \qquad = -\re \{z_x \overline{N[z]} \} + \frac{a_\ve}{m+1} \partial_x\Big\{ |\tilde u+z|^{m+1} -|\tilde u|^{m+1} -(m+1)|\tilde u|^{m-1}\re \{\tilde u \bar z\} \Big\};
\eee
and integration by parts  and (\ref{Est1}).
\end{proof}

\noindent
{\bf End of proof of Proposition \ref{prop:I}.} 
Using Gronwall's inequality (see e.g. \cite{Mu2} for more details) in (\ref{IntF}), the fact that $\rho'(t) \geq \frac 12 v_0>0$, estimate (\ref{TRANS3}), and Lemma \ref{Coer2} we conclude that for some large constant $K>0$, but independent of $K^*$ and $\ve$,
\bee
\|z(t)\|_{H^1(\R)}^2 \leq  K\ve^{2p_m } +  K(K^*)^4\ve^{4p_m-1  -\frac 1{100}} + KK^* \ve^{2p_m }.
\eee
From this estimate and taking $\ve$ small, and $K^*$ large enough, we obtain that for all $t\in [-T_\ve, T^*]$,
$$
\|z(t)\|_{H^1(\R)}^2 \leq  \frac 12 (K^*)^2 \ve^{2p_m }.
$$
Next, from the mass conservation law and (\ref{minTe}) one has
$$
|c(t) -C(t)|\leq K \ve^{p_m} + K(K^*)^2\ve^{2p_m} \leq K\ve^{p_m}, 
$$
and from the energy conservation law and once again (\ref{minTe})
$$
|v(t) -V(t)|\leq K \ve^{p_m} + K(K^*)^2\ve^{2p_m} \leq K\ve^{p_m}. 
$$
Therefore, for $K^*$ large enough and all $\ve>0$ small,
$$
\| u(T^*) - \tilde u(C(T^*),V(T^*),\rho(T^*),\ga(T^*)) \|_{H^1(\R)} \leq \frac 23 K^* \ve^{p_m}.
$$
This estimate contradicts the definition of $T^*$, and therefore $T^* =T_\ve$. In addition, from (\ref{TRANS3}) we obtain (\ref{INT41a}). Finally (\ref{garho}) and (\ref{cc}) are consequence of (\ref{rho1})-(\ref{c1}). The proof of Proposition \ref{prop:I} is now complete.

\bigskip

\noindent
{\bf Final Step. Conclusion and Proof of Proposition \ref{T0}.} Now we prove the main result of this section, which describes the core of interaction soliton-potential. 

\begin{proof}[Proof of Proposition \ref{T0}]
Consider $u(t)$ a solution of (\ref{aKdV}) satisfying (\ref{hypINTa}). We first compare $u(t)$ with the approximate solution $\hat u(t)$ from Lemma \ref{atpmT}, at time $t=-T_\ve$.

\subsubsection{Behavior at $t= - T_\ve$} We claim that a suitable modification of $u$ matches with our approximate solution $\hat u(t)$.  Indeed, for $\ga_{-1}$ introduced in (\ref{gam1}), let
$$
v(t,x) := u(t,x)e^{i\tilde\ga}, \qquad \tilde\ga :=  (1- \frac 14 v_0^2) T_\ve  + \ga_{-1},
$$
which still satisfies (\ref{aKdV}).  From (\ref{hypINTa}) and (\ref{mTe}) we have that
$$
\|v(-T_\ve ) - \hat u(-T_\ve) \|_{H^1(\R)} \leq K \ve^{10}.
$$

\subsubsection{Behavior at $t= T_\ve$} Thanks to the above estimate we can invoke Proposition \ref{prop:I} to obtain the existence of $K_0,\ve_0>0$ such that for all $0<\ve<\ve_0$
$$
\| v( T_\ve )  - \tilde u(T_\ve, c(T_\ve), v(T_\ve),\rho(T_\ve) , \ga(T_\ve))  \|_{H^1(\R)} \leq K_0\ve^{p_m},
$$
with $|c(T_\ve) - C(T_\ve)| + |v(T_\ve) - V(T_\ve)| \leq K_0 \ve^{p_m}$.
On the other hand, note that from Lemma \ref{ODE}, (\ref{defALPHA}), (\ref{defv}), (\ref{defW}), (\ref{IP}) and the last estimates
$$
\| \tilde u(T_\ve, c(T_\ve), v(T_\ve),\rho(T_\ve) , \ga(T_\ve)) - \la_\infty Q_{c_\infty} (\cdot -\rho(T_\ve)) e^{\frac i2(\cdot ) v_\infty} e^{i\bar\ga(T_\ve)}  \|_{H^1(\R)} \leq K K_0\ve^{p_m},
$$
where
$$
|\rho(T_\ve) -U(T_\ve) | \leq \frac{T_\ve}{100}, \quad \bar{\ga}(T_\ve) :=  \int_0^{T_\ve} c(s) ds  -\frac 14\int_0^{T_\ve} v^2(s)ds + \ga(T_\ve).
$$
The proof of this last estimate is similar to the proof of Lemma \ref{atpmT}. Therefore 
$$
\| v( T_\ve ) - \la_\infty Q_{c_\infty}(\cdot  - \rho( T_\ve))e^{\frac i2 (\cdot )v_\infty}  e^{i \bar{\ga}(T_\ve) }  \|_{H^1(\R)} \leq K K_0\ve^{p_m},
$$
Returning to the original function $u$, we obtain that
$$
\| u(T_\ve) -\la_\infty Q_{c_\infty}(\cdot  -\rho( T_\ve)) e^{\frac i2 (\cdot )v_\infty}  e^{i(\bar{\ga}(T_\ve)-\tilde \ga)} \|_{H^1(\R)} \leq K K_0\ve^{p_m}.
$$
Finally, note that  $ \frac {99}{100}v_0 T_\ve \leq \rho(T_\ve) \leq \frac{101}{100}(2v_\infty - v_0)T_\ve. $ By defining $\rho_\ve := \rho( T_\ve)$, and $\ga_\ve :=  \bar{\ga}(T_\ve)-\tilde \ga,$ we obtain (\ref{INT41})-(\ref{INT42}). This finishes the proof.
\end{proof}


\bigskip


\section{Decreasing potential and reflection}\label{Ap}

In this section we sketch the proof of Theorem A'.  First of all, it is not difficult to prove the following result, by following Proposition \ref{Tm1}.
 
 \begin{prop}[Existence of a pure soliton-like solution]\label{STm1r}~

There exists $\ve_0>0$ such that for any $0<\ve < \ve_0$, there exists a solution $u_\# \in C(\R, H^1(\R))$ of (\ref{aKdV})-(\ref{ahypSr})  such that 
\be\label{lim0r}
\lim_{t\to -\infty} \|u_\#(t) -  Q(\cdot -v_0 t)e^{\frac i2 (\cdot ) v_0}e^{i(1-\frac 14v_0^2)t}  \|_{H^1(\R)} =0,
\ee
with mass $M[u_\#](t) = M[Q]$ and energy $E_a[u_\#](t) = (\frac 14v_0^2 -\la_0)M[Q]<0.$
Moreover, there exist constants $K,\mu>0$ such that  for all $t\leq -\frac 1{2}T_\ve$,
\be\label{minusTer}
\|u_\#(t) -  Q(\cdot -v_0 t)e^{\frac i2 (\cdot ) v_0}e^{i(1-\frac 14v_0^2)t} \|_{H^1(\R)} \leq  K e^{\ve \mu t}.
\ee
In particular, 
\be\label{mTepr}
\|u_\#(-T_\ve) - Q(\cdot + v_0 T_\ve )e^{\frac i2 (\cdot ) v_0}e^{-i(1-\frac 14v_0^2)T_\ve } \|_{H^1(\R)} \leq K e^{- \mu \ve^{-\frac 1{100}}} \leq K \ve^{10},
\ee
provided $0<\ve<\ve_0$ small enough.
\end{prop}

\begin{rem}
Note that from this result, we have proved the first part of Theorem A', namely (\ref{SMinftyr}).
\end{rem}

\begin{rem}
The main difference between Propositions \ref{Tm1} and \ref{STm1r} is in the proof of Proposition \ref{Uealpha}. Now condition (\ref{Ue1alpha}) must be replaced by the stronger assumption
$$
\|u_n(t) - Q(\cdot - v_0 t)e^{\frac i2 (\cdot) v_0}e^{i(1-\frac 14 v_0^2)t} \|_{H^1(\R)}\leq 2K e^{\ve \mu t}.
$$
This assumption is intended to restore the lack of monotonicity of the momentum in the good direction. Note that a similar condition is present in the construction of $N$ solitary waves for NLS equations, see \cite{Martel}.  

Finally, from this last fact the proof of  {\bf uniqueness} of this solution is an open issue. 
\end{rem}

\bigskip

The next step is the study of the interaction soliton-potential. This is the part of the proof where we need some completely new computations. Our objective is to prove the following result.

\begin{prop}[Dynamics of the soliton in the interaction region]\label{ST0r}~

Suppose $v_0>0$ satisfying the smallness condition
$$
v_0^2 < 4\la_0( 1-a_0^{\frac 4{5-m}} ), \qquad \hbox{ $\la_0$ given in Theorem A.} 
$$
There exist constants $K_0,\ve_0>0$ such that the following holds for any $0<\ve <\ve_0$. There exist $\rho_\ve, \ga_\ve \in \R$ such that $u_\#$ defined in Proposition \ref{STm1r} satisfies
$$
\|u_\#(\tilde T_\ve)  - Q(\cdot - \rho_\ve) e^{-\frac i2(\cdot)v_0} e^{i\ga_\ve} \|_{H^1(\R)} \leq K_0 \ve^{p_m},
$$
and $ \rho_\ve \sim -v_0T_\ve$ and $\tilde T_\ve \sim T_\ve$, independent of $\ve$.
\end{prop}

\begin{rem}
From the proof of this result it will be clear that in the case $v_0^2 > 4\la_0( 1-a_0^{\frac 4{5-m}} )$ the soliton exits by the right hand side of the potential. However, no stability result for large time is known in the regime of a decreasing potential, so we do not know the asymptotic behavior of this solution. The proof of this result will require some new ideas.
\end{rem}

\bigskip

Let us assume for the moment the validity of Proposition \ref{ST0r}, in order to describe the last step of the proof, namely the asymptotic analysis.

\subsection{Stability of the reflected soliton solution}

The final step towards the proof of Theorem \ref{MTref} is a stability result for a reflected soliton. Let us recall that this result is consequence of the good sign of the following derivative:
\be\label{dPneg}
\partial_t \{ (-v_0) P[u_\#](t)\} \geq 0,
\ee
which does not hold for the case of final velocity $v_\infty>0$.

\begin{prop}[Stability in $H^1(\R)$, reflected case]\label{STAAd}~

Suppose $2\leq m<5$. There exists $\ve_0>0$ such that if $0<\ve <\ve_0$ the following hold. Suppose that for some time $t_1\geq \frac 12 T_\ve$, $ X_0 \leq -v_0 t_1 $ and $\ga_0\in \R$ and $K>0$,
\be\label{S18Ad}
\| u_\#(t_1) -  Q(\cdot-X_0) e^{-\frac i2 xv_0} e^{i\ga_0} \|_{H^1(\R)} \leq  K\ve^{p_m}.
\ee
where $u_\#(t)$ is the global $H^1$-solution of (\ref{aKdV}) with decreasing potential (\ref{ahypSr}). 

\smallskip
\noindent
Then there exist $K_0>0$ and $C^1$-functions $\rho(t), \ga(t) \in \R$ defined in $[t_1,+\infty)$ such that 
$$
w(t) := u_\#(t) -  Q (\cdot +v_0 t -\rho_2(t))e^{-\frac i2 (\cdot) v_0}  e^{i\ga(t)},
$$
satisfies for all $t\geq t_1$,
\be\label{SSAd}
\| w (t) \|_{H^1(\R)} +|\rho_2'(t)| +|\ga_2'(t) -1 +\frac 14 v_0^2  |  \leq K_0\ve^{p_m},
\ee
where, for some $K>0$,
$$
|\rho_2(t_1) - v_0 t_1- X_0 | + |\ga_2(t_1) - \ga_0 | \leq  K \ve^{p_m}.
$$
\end{prop}

\bigskip

\noindent
{\bf End of proof of Theorem A'.}
The second part of Theorem A', namely \ref{SMT2r}, follows from (\ref{SSAd}) in the preceding proposition, after defining $\rho(t) := \rho_2(t)$ and $\ga(t) := \ga_2(t)$. The proof is now complete, provided Proposition \ref{STAAd} holds.

\begin{proof}[Proof of Proposition \ref{STAAd}]
The proof of this result is based in a standard Weinstein argument, similar to the proof of Proposition \ref{Tp1}. The main difference is now given by the introduction of the Weinstein functional
$$
\tilde{\mathcal F}(t) := E_a[u](t) +  (\tilde c + \frac 14 v_0^2) M[u](t)  + v_0 P[u](t),
$$ 
for a suitable positive constant $\tilde c$ (e.g. one may take $\tilde c =1$). From (\ref{dPneg}), this functional satisfies $\tilde{\mathcal F}'(t) \leq 0$, which is the key property to establish the stability result.  For a complete proof, see e.g. \cite{Mu4}.

\end{proof}

\bigskip

The next step is the proof of Proposition \ref{ST0r}. In order to obtain this result, the first step is a detailed study of the dynamical system (\ref{c}) for the case of a decreasing potential. That is the objective of the next subsection.

\subsection{Study of a dynamical system, revisited}

Similarly to Proposition \ref{T0}, the dynamics of a soliton soliton is mainly described by its velocity $V(t)$, position $U(t)$, scaling $C(t)$ and phase $H(t)$. The dynamical system governing these variables is the same as in Lemma \ref{ODE}, with the key difference on the sign of the derivative of the potential $a$. Indeed, recall that 
\be\label{f1f2r}
 f_1(C,U) =\frac{8 a'(\ve U) C}{(m+3)a(\ve U)},  \qquad f_2(C,V,U)  =\frac{4 C V a'(\ve U) }{(5-m)a(\ve U)}.
\ee
Our first result is as follows:

\begin{lem}[Existence of approximated dynamical parameters, case $2\leq m<5$]\label{ODEr}~

Let $v_0>0, \la_0, a(s) $ be as in Theorem \ref{MTref} and  (\ref{ahypSr}). There exists a unique solution $(V,C, U, H)$ defined for all $t\geq -T_\ve$ with the same regularity than $a(\ve \cdot)$, of the following nonlinear system of differential equations 
\be\label{crS}
\begin{cases}
\displaystyle{V'(t) = \ve f_1(C(t),U(t))}, &  V(-T_\ve) = v_0, \\
\displaystyle{C'(t) = \ve f_2(C(t),V(t), U(t)), } &  C(-T_\ve) =1, \\
\displaystyle{U'(t) = V(t),} &  U(-T_\ve) = -v_0 T_\ve, \\
\displaystyle{H'(t) =  - \frac 12V'(t) U(t),} & H(-T_\ve) = 0.
\end{cases}
\ee
In addition, for all $t\geq -T_\ve$, $C(t), V(t)$ are strictly decreasing with 
\be\label{c10}
C(t)  = \frac{a^{4/(5-m)}(\ve U(t))}{a^{4(5-m)}(-\ve v_0 T_\ve) } = \frac{a^{4/(5-m)}(\ve U(t))}{a^{4(5-m)}(-\ve^{-1/100}) },
\ee
and satisfy the parabola
\be\label{parab}
C(t) = c_0  +\frac{V^2(t)}{4\la_0}, \quad  c_0 := 1-\frac{v_0^2}{4\la_0}<1.
\ee
\end{lem}

\begin{proof}
The existence of a local solution of (\ref{crS}) is consequence of the Cauchy-Lipschitz-Picard theorem. 

Now, in order to prove global existence of such a solution, we derive some a priori estimates. Note that from the first equation in (\ref{c}) we have $C$ \emph{strictly decreasing} in time with $C(t) \leq 1$, $t\geq -T_\ve$.   
Moreover, after integration, we have (\ref{c10}). Since $\frac 12 <a < 1$, one has that $C$ is bounded and globally well defined with
\be\label{c11}
\frac{a_0^{4/(5-m)}}{a^{4(5-m)}(-\ve^{-1/100}) }  \leq  C(t) < 1 , \quad t\geq -T_\ve.
\ee

On the other hand, from the second equation in (\ref{crS}), we have $V$ \emph{strictly decreasing} in time. Replacing (\ref{c10}), and after multiplication by $V(t)$, one has
$$
V(t) V'(t) = \frac 8{m+3} a^{\frac{m-1}{5-m}}(\ve U(t)) a'(\ve U(t)) V(t) a^{-\frac 4{5-m}}(-\ve^{-1/100}).
$$
After integration in $[-T_\ve, t)$ we obtain  (\ref{parab}). This last relation and the fact that $C(t)\leq 1$ implies the global existence of $V$ and the uniform bound
$$
\abs{V(t)} \leq v_0, \quad t\geq -T_\ve.
$$
The proof is complete.
\end{proof}

\medskip
Now we describe the behavior of $(C,V,U)$ for large times.  Interestingly enough, here the long time behavior may be different depending on the initial velocity $v_0$. Recall that $c_0 = 1-\frac{v_0^2}{4\la_0}$. 

\medskip

\noindent
{\bf First case: $c_0 <  a_0^{4/(5-m)}$.} In this case, the soliton is formally refracted by the potential.

\begin{lem}[Long time behavior, refracting case]\label{FirstC}~

Suppose $c_0 < a_0^{4/(5-m)}$. Then $\lim_{t\to +\infty} (C(t),V(t),U(t)) =(c_\infty, v_\infty,+\infty)$, with
\be\label{K3}
c_\infty =a_0^{\frac 4{5-m}} (1+O(\ve^{10})), \quad \hbox{ and }\quad v_\infty = [ 4\la_0(c_\infty -c_0) ]^{1/2}>0.
\ee
Moreover, there exists $-T_\ve<\tilde T_\ve<K(v_0)T_\ve$ such that $U(\tilde T_\ve) = -U(-T_\ve).$
\end{lem}

\begin{proof}
We prove that for $\ve$ small $\lim_{t\to +\infty } C(t) = a_0^{\frac 4{5-m}} (1+O(\ve^{10})). $ Indeed, note that from (\ref{c10}) and (\ref{parab}) we have
$$
0\leq a_0^{\frac 4{5-m}} -c_0 <\frac{V^2(t)}{4\la_0};
$$
thus $V(t)>0$ for all $t\geq -T_\ve$. Moreover, if $\lim_{+\infty}V(t) =0$, we have from (\ref{c10})
$$
\lim_{t\to +\infty}C(t) =c_0\leq a_0^{\frac 4{5-m}}, 
$$ 
a contradiction with (\ref{c11}) (after taking $\limsup$). In concluding we have $\lim_{+\infty}V(t) >0$, and $\lim_{+\infty} U(t) =+\infty$. Passing to the limit in (\ref{c10}), one obtains
$$
\lim_{t\to +\infty} C(t) = \frac{a_0^{4/(5-m)}}{a^{4/(5-m)}(-\ve^{-1/100})} > a_0^{4/(5-m)}>c_0.
$$
This proves the first assertion in (\ref{K3}). The last assertion in (\ref{K3}) follows from (\ref{parab}); we get
$$
\lim_{t\to +\infty} V(t) = [ 4\la_0( \lim_{t\to +\infty} C(t) -c_0) ]^{1/2}>0.
$$
Now, let us define the {\bf exit time} $\tilde T_\ve > -T_\ve$ such that $U(\tilde T_\ve) = -U(-T_\ve)$. Note that this time is unique since $U$ is a strictly increasing function. We have
\be\label{Ueee}
-U(-T_\ve)  =  U(-T_\ve) + \int_{-T_\ve}^{\tilde T_\ve } V(t)dt  \geq    U(-T_\ve) + V(+\infty) (\tilde T_\ve +T_\ve), 
\ee
then, if $c_0<  a_0^{4/(5-m)}$ we have $\tilde T_\ve \leq KT_\ve$, since $V(+\infty)=v_\infty$ is bounded by below independent of $\ve$.
\end{proof}

\medskip

\noindent
{\bf Second case: $c_0 > a_0^{4/(5-m)}$.} In this case, the soliton is formally reflected by the potential.

\begin{lem}[Long time behavior, reflecting case]\label{KKK}~

Suppose now $c_0 >  a_0^{4/(5-m)}$. Then there exists a unique $t_0>-T_\ve$, satisfying $V(t_0)=0$  and  $t_0\leq K(v_0)T_\ve $ for some constant $K(v_0)>0$ but independent of $\ve$. Moreover, one has  
$$
\lim_{t\to +\infty} (C(t),V(t),U(t)) =(1, v_0,-\infty).
$$
Moreover, there exist $\tilde T_\ve >-T_\ve $ and $\tilde K(v_0)>0$ such that $U(\tilde T_\ve) = U(-T_\ve)$, with $\tilde T_\ve \leq \tilde K(v_0)T_\ve$.
\end{lem}

\begin{proof}
First we prove the existence of $t_0> -T_\ve$ such that $V(t_0)=0$. Note that its existence implies its uniqueness. By contradiction, let us suppose that $V(t) >0$ for all $t>-T_\ve$. Then $U(t)$ is strictly increasing. Here we have two cases. First, suppose $\lim_{+\infty} V(t) >0$. Then we have
$U(+\infty) =+\infty$ and therefore $C(t)$ is strictly decreasing with
$$
\lim_{t\to+\infty}C(t) = \frac{a_0^{4/{5-m}}}{a^{4/(5-m)}(-\ve^{-1/100})}.
$$  
Passing to the limit in (\ref{parab}), one has
$$
\frac{a_0^{4/{5-m}}}{a^{4/(5-m)}(-\ve^{-1/100})} >  c_0  > a_0^{4/(5-m)},
$$
a contradiction for $\ve$ small enough.

Now suppose $\lim_{t\to +\infty} V(t) =0$. Here we have two cases: either $\lim_{t\to +\infty}U(t) =+\infty$, or $U(-T_\ve) \leq \lim_{t\to +\infty}U(t) =: U_\infty<+\infty$. For the first case, similarly to the recent analysis, one has
$$
\lim_{t\to +\infty}C(t) =: C_\infty = \frac{a_0^{4/{5-m}}}{a^{4/(5-m)}(-\ve^{-1/100})} \geq  c_0  > a_0^{4/(5-m)}.
$$
This is a contradiction for $\ve$ small enough.

In the second case, one has
$$
C_\infty =  \frac{a^{4/(5-m)}(\ve U_\infty)}{a^{4(5-m)}(-\ve^{-1/100}) } >0,
$$
and
$$
\lim_{t\to +\infty}V'(t) = \ve f_1(C_\infty, U_\infty)<0,
$$
a contradiction with $\lim_{t\to+\infty} V(t) =0$ (since $\lim_{t\to +\infty} V' (t)= \lim_{t\to +\infty} \frac{V(t)}{t} =0$.)

\smallskip

Therefore, there exists $t_0\in \R$ such that $V(t_0) =0$, with $V(t)<0$ for $t>t_0$. In addition, $C(t_0) = c_0$ and $U'(t_0)= 0.$ From (\ref{c10}) one has $|\ve U(t_0) | \leq K(v_0)$.

\smallskip

Moreover, $\lim_{t\to +\infty}U(t) = -\infty$.  Indeed, note that for $\nu>0$ small, since $V''(t) =O(\ve^2)$,
$$
V(t_0 \pm  \frac \nu\ve) = \pm V'(t_0) \frac \nu\ve + O(\nu^2) =\pm \nu f_1(c_0,U(t_0)) +  O(\nu^2) =\pm \kappa_0 \nu + O(\nu^2).
$$
with $ \kappa_0 := f_1(c_0,U(t_0)) = \frac{8c_0}{(m+3)} \frac{a' (\ve U(t_0))}{a(\ve U(t_0))}<0$. Then for $t>t_0 + \frac\nu\ve$,
\bee
U(t) & = & U(-T_\ve) + \int_{-T_\ve}^{t_0-\frac{\nu}{\ve}} V(t)dt +\int_{t_0-\frac{\nu}{\ve}}^{t_0+\frac{\nu}{\ve}} V(t)dt + \int_{t_0+ \frac{\nu}{\ve}}^{t} V(t)dt \\
& \leq  &  U(-T_\ve) + v_0 (T_\ve + t_0-\frac{\nu}{\ve})  + v_0 \frac \nu\ve  + \nu\kappa_0 ( t -t_0 - \frac{\nu}{\ve}). 
\eee
and thus $\lim_{t\to +\infty}U(t) = -\infty$. In addition, $\lim_{t\to +\infty}C(t) =a^{-\frac {4}{5-m}}(-\ve^{-1/100}) = 1+O(\ve^{10})$ and  $\lim_{t\to +\infty}V(t) = -v_0+ O(\ve^{10})$.

\smallskip

Similarly,
\bee
U(t_0) & = & U(-T_\ve) + \int_{-T_\ve}^{t_0-\frac \nu\ve} V(t)dt +  \int_{t_0-\frac \nu\ve}^{t_0} V(t)dt  \\
&\geq &  U(-T_\ve)  - \nu \kappa_0 (t_0-\frac \nu\ve +T_\ve) -K\frac \nu\ve. 
\eee
This inequality implies that $t_0 \leq K T_\ve$, with $K = K(v_0)>0$ independent of $\ve.$ Note that $K$ becomes singular as $v_0$ approaches $c(v_0)= a_0^{4/(5-m)}$.

Let us define $\tilde T_\ve> -T_\ve$ such that $U(\tilde T_\ve ) = U(-T_\ve).$ Then we have $C(\tilde T_\ve) =1$ and $V(\tilde T_\ve) = -v_0$. We finally have
\bee
0 & = & \int_{-T_\ve}^{\tilde T_\ve} v(t)dt =  \int_{-T_\ve}^{t_0-\frac{\nu}{\ve}} V(t)dt +\int_{t_0-\frac{\nu}{\ve}}^{t_0+\frac{\nu}{\ve}} V(t)dt - \int_{t_0+ \frac{\nu}{\ve}}^{\tilde T_\ve} |V(t)|dt \\
& \leq  &  v_0 (t_0-\frac{\nu}{\ve} +T_\ve) + v_0 \frac\nu\ve +   v(t_0 +\frac \nu \ve) (\tilde T_\ve -t_0 +\frac \nu\ve).
\eee
In conclusion $\tilde T_\ve \leq \tilde K(v_0) T_\ve.$
\end{proof}

\medskip

\noindent
{\bf Final case: $c_0 =  a_0^{4/(5-m)}$.} In this case, following the proof of Lemma \ref{FirstC}, we have

\medskip

\begin{lem}[Long time behavior, critical case]~

Suppose $c_0 = a_0^{4/(5-m)}$. Then $\lim_{t\to +\infty} (C(t),V(t),U(t)) =(c_\infty, v_\infty,+\infty)$, with
\be\label{K3}
c_\infty =c_0(1+O(\ve^{10})), \quad \hbox{ and }\quad v_\infty = [ 4\la_0(c_\infty -c_0) ]^{1/2}>0.
\ee
\end{lem}

\begin{rem}
The main difference between this last result and Lemma \ref{FirstC} is in the fact that we cannot bound by below $v_\infty$ in a suitable way, independent of $\ve$. As a consequence, we cannot estimate the time of interaction as in Lemma \ref{FirstC}. The behavior of the soliton in this case remains an interesting open problem.
\end{rem}

\bigskip

\noindent
{\bf Proof of Proposition \ref{ST0r}}
The proof of this result is direct, just follow the lines of the proof of Proposition \ref{T0}, now applied in the interval $[-T_\ve, \tilde T_\ve]$. The main modification is in the a priori assumptions (\ref{aprio}). Now we assume that
$$
|c(t) -C(t)| + |v(t) -V(t) | + |\rho'(t) -U'(t)| \leq \ve^{1/100}.
$$
It is clear that these estimates are improved in the new version of Lemma \ref{DEFZ}.  The last issue is the integrability of the term $\ve e^{-\ve \ga |\rho(t)|}$, since $\rho'(t) \sim v(t)$ may change of sign during the dynamics. However, from the last assumptions, and a correct splitting of the interval of integration as in the proof of Lemma \ref{KKK}, it is easy to see that 
$$
\int_\R \ve e^{-\ve \ga |\rho(t)|} dt \leq K,
$$
with $K$ independent of $\ve$. We left the details to the reader.

\bigskip

\section{The two dimensional case}\label{TC}

In this section we sketch the proof of Theorem A for dimension $2$, namely Theorem B. More precisely, our objective is to  adapt the proof of Propositions \ref{Tm1}, \ref{T0} and \ref{Tp1} to the two dimensional case. Recall that $ 2\leq m< 3$.

\medskip

\noindent
{\bf Step 1. Proposition \ref{Tm1} revisited.}  The proof of this result is identical to the one dimensional case (see Appendix \ref{A}), with the novelty that $\rho_0(t)$ is now a $\R^2$-valued vector. The uniqueness follows essentially from (\ref{sign}). No additional modifications are required.

\medskip

\noindent
{\bf Step 2. Proposition \ref{T0} revisited.} Here we need to introduce several modifications on the computations.

\smallskip

First of all, the Cauchy problem (\ref{Cp1}) in the higher dimensional case is globally well posed for $1<m<3$ for $L^2$ and $H^1$ data, see \cite{Caz, GV}. The conservations laws (\ref{M}), (\ref{Ea}) and identity (\ref{dPN}) hold without modifications. 

On the other hand, (\ref{defLy}) now reads
\be\label{defLyN}
    \mathcal{L}_+ w(y) := - \Delta_y w + c w - m Q_c^{m-1}(y) w, \quad\hbox{ and }\quad  \mathcal{L}_- w(y) := - \Delta_y w +  cw -  Q_c^{m-1}(y) w;
\ee
where $w=w(y)$. Lemma \ref{surL} is also valid in higher dimensions. In particular, one has the following.  Assume that $v\in \R^2$, $v\neq 0$, $\theta\in \R$, and for $k=1, 2$, one has
$$
 \re \int_{\R^2}  \bar w \partial_{y_k} Q_c e^{iy\cdot v/2}e^{i\theta} = \ima \int_{\R^2} \bar w Q_c e^{iy\cdot v/2}e^{i\theta} = \re \int_{\R^2} \bar w Q_c e^{iy\cdot v/2}e^{i\theta}  =0.
 $$
Then
$$
\tilde{\mathcal B}[w,w] \geq \sigma_c \int_{\R^2} |w|^2,
$$
where $\tilde{\mathcal B}[w,w] $ is the standard $2$-dimensional generalization of the functional $\tilde{\mathcal B}$ defined in Lemma \ref{surL}. Finally, the space $\mathcal Y$ in (\ref{Y}) is easily generalizable to higher dimensions.

\medskip

Let us consider now the approximate solution $\tilde u$. From the fact that the potential $a$ depends only on $x_1$, the relevant dynamical system depends only on this variable. 
Indeed, for $t$ in $[-T_\ve, T_\ve]$, let
$$
c(t),  \ga(t) \in \R, \qquad v(t) =(v_1(t), v_2(t)) \in \R^2, \; \rho( t) = (\rho_1(t), \rho_2(t))\in\R^2,
$$
to be fixed later. Consider $y :=(y_1,y_2)$, where
\be\label{defALPHAN}
y:=x -  \rho(t), \quad \hbox{and} \quad     \tilde R(t,x): = \frac{Q_{c(t)}(y)}{\tilde a(\ve \rho_1(t))}e^{i\Theta(t,x)},
\ee
where, as in the one-dimensional case 
\be\label{param0N}
\tilde a := a^{\frac 1{m-1}}, \quad \Theta (t,x) := \int_0^t c(s) ds  + \frac 12 v(t)\cdot x -\frac 14 \int_0^t |v|^2(s) ds + \ga(t).  
\ee
In addition, we will search for \emph{bounded} parameters $(c,v,\ga)$ satisfying the same  constraints (\ref{aprio}), with the obvious modifications.

By now we only need these hypotheses. As in Lemma \ref{ODE} and Proposition \ref{prop:I}, we will construct a quadruplet $(c,v,\rho,\ga)$ with better estimates.

\medskip

On the other hand, the form of the ansatz $\tilde u(t,x)$ is given by (\ref{defv}), with 
\be\label{defWN}
w(t,x):= \ve (A_{1,c} (t, y) + i  B_{1,c}(t, y) ) e^{i\Theta},
\ee
with $A_{1,c}, B_{1,c} $ satisfying condition (\ref{IP}) in $\R^2$. Proposition \ref{prop:decomp} now reads
\be\label{SmN}
S[\tilde u](t,x)   =  \Big[ \mathcal F_0(t, y) +  \ve \mathcal F_1(t,y) +  \ve^2 \mathcal F_2(t, y)  + \ve^3 f(t) \mathcal F_c(y) \Big] e^{i\Theta (t, x)},
\ee
where $\mathcal F_0$ is given now by
\bea\label{F0N}
\mathcal F_0(t, y)&  := &  - \frac 12  (v'(t) - \ve f_1 (t) ) \cdot y\tilde u  +i ( c'(t)- \ve f_2(t)  ) \partial_c \tilde u \nonumber \\
& & \quad   -(\ga'(t) + \frac 12 v'(t)\cdot \rho ( t)) \tilde u  + i (\rho'(t) -v(t) ) \cdot \partial_\rho \tilde u,
\eea
$$
f_1(t) :=(\frac{4 \kappa a'(\ve \rho_1(t)) c(t)}{(m+1)a(\ve \rho_1( t))}, 0) , \quad f_2(t) :=  \frac{2 a'(\ve \rho_1( t)) c(t) v_1(t)}{(3-m)a(\ve \rho_1(t))};
$$
\be\label{F1N}
\mathcal F_1(t, y)  :=  F_1(t, y) + i G_1(t, y)  - \big[  \mathcal L_+ (A_{1,c})  + i \mathcal L_- (B_{1,c}) \big],
\ee
with
\bee
& & F_1(t, y) :=  \frac{a' (\ve \rho_1( t))}{\tilde a^m (\ve \rho_1(t))}y_1Q_c(y)\big[ Q_c^{m-1}(y) -\frac {2\kappa c(t)}{m+1}\big], \\
& &  G_1(t, y) := \frac{a' (\ve \rho_1( t)) v_1(t) }{\tilde a^m(\ve \rho_1( t))}\big[ \frac{2c(t)}{3-m}\Lambda Q_c (y)-\frac{1}{m-1}Q_c(y)\big],
\eee
and $\kappa := \frac{\int Q^{m+1}}{\int Q^2}$. Furthermore
\be
\|\ve^2 \mathcal F_2(t, \cdot) \|_{H^1(\R^2)} \leq K\ve^2 e^{-\ve\mu |\rho(t)|}; \quad  \|\ve^3 f(t) \mathcal F_c\|_{H^1(\R^2)}\leq K\ve^3,
\ee
uniformly in time, provided $(A_{1,c}, B_{1,c})$ satisfy (\ref{IP}).

\medskip

Now, let us describe the main differences on the dynamical system concerning the essentially important variables for the dynamics: $c(t), v_1(t), \rho(t)$ and $\ga(t)$. The result is the following.

\begin{lem}[Existence of dynamical parameters]\label{ODEN}~

Suppose $2\leq m< 3$. Let $v_0>0, \la_0, a(s) $ be as in Theorem C and  (\ref{ahyp}). There exists a unique solution $(c,v,\rho,\ga)$ defined for all $t\geq -T_\ve$ with the same regularity than $a(\ve \cdot)$, of the following nonlinear system of differential equations  
\be\label{cn}
\begin{cases}
\displaystyle{C'(t) = \frac{2\ve a' (\ve U_1(t)) }{(3-m)a(\ve U_1( t))} C(t) V_1(t), } &  C(-T_\ve) =1, \\
\displaystyle{V_1'(t) = \frac {4\ve \kappa }{m+1} \frac{a'  (\ve U_1( t))}{a(\ve U_1(t))} C(t)}, &  V_1(-T_\ve) = v_0, \\
\displaystyle{U_1'(t) = V_1(t),} &  U_1(-T_\ve) = -v_0 T_\ve, \\
\displaystyle{H'(t) =  - \frac 12 V_1'(t)U_1(t),} & H(-T_\ve) = 0.
\end{cases}
\ee
In addition,
\ben
\item $C(t)$ is strictly increasing with $1 \leq C(t) \leq C(T_\ve),$, with 
$$
C(T_\ve) = c_\infty  + O(\ve^{10}) = 2^{\frac 2{3 -m}} + O(\ve^{10}).
$$
\item $V(t)$ is strictly increasing with $v_0\leq V(t) \leq V(T_\ve)$, with 
$$
V(T_\ve) =  v_\infty + O(\ve^{10}) =  (v_0^2 + 4\al_0 (c_\infty-1))^{1/2}+ O(\ve^{10}),
$$
with $\al_0$ given in (\ref{cvinfN}).
\een
\end{lem}

On the other hand, the first linear system $(\Omega_1)$ is easily solvable, because
$$
\int_{\R^2} F_1 \partial_{y_i} Q_c(y) = \int_{\R^2} G_1 Q_c(y) =0.
$$
Moreover, the solution $(A_{1,c}, B_{1,c})$ satisfies (\ref{IP}). In addition, Lemma \ref{CV} now reads
$$
\|\tilde S[\tilde u](t)\|_{H^1(\R^2)} \leq K\ve^2(e^{-\ve \mu |\rho(t)|} + \ve).
$$
Similarly Lemma \ref{atpmT} holds with no major modifications.

\medskip

Let us sketch the proof of Proposition \ref{prop:I} in the higher dimensional case. As in Lemma \ref{DEFZ} we consider
$$
z(t):= u(t) - \tilde u(t, c(t),v(t),\rho(t),\ga(t)),
$$
satisfying for $k=1,2$, and for all $t\in [-T_\ve, T_\ve]$,
$$
\int_{\R^2} \bar z(t,x) Q_c(y) = \int_{\R^2} \bar z(t,x) \partial_{y_k} Q_c(y) =0,
$$
and the equation
\bee
& & iz_t + \Delta z + a_\ve(x_1) [  |\tilde u + z|^{m-1}(\tilde u + z) -|\tilde u|^{m-1}\tilde u  ] + \tilde S[\tilde u]  \\
& & \qquad \qquad \qquad  -\frac 12 (v' -\ve f_1) \cdot y \tilde u + i(c'-\ve f_2) \partial_c \tilde u + i(\rho' -v) \cdot \partial_\rho \tilde u + (\ga' +\frac 12 v' \cdot \rho) \tilde u  =0,
\eee
in addition to (\ref{rho1})-(\ref{c1}).

\medskip

Finally, the functional $\mathcal F$ in (\ref{F}) remains the same, up to the obvious modifications: we replace $z_x$ by $\nabla z$ and $v$ by its vectorial version. Following these steps, we finally conclude (\ref{INT41a}) and therefore the two dimensional version of Proposition \ref{T0}.

\medskip

\noindent
{\bf Step 3. Proposition \ref{Tp1} revisited.} The proof of this result is identical to the one dimensional case. No additional modifications to the standard ones are required.

From this analysis we conclude the proof of Theorem B.

\bigskip

\appendix

\section{Proof of Proposition \ref{Tm1}}\label{A}

In this section we sketch the proof of Proposition \ref{Tm1}. For a similar proof, see e.g. \cite{Mu2}.

\medskip

Let $(T_n)_{n\in \N}\subseteq \R$ an increasing sequence with $T_n\geq \frac 1{2}T_\ve$ for all $n$ and $\lim_{n\to +\infty} T_n =+\infty$. Consider $u_n(t)$ the solution of the following Cauchy problem  
\be\label{CPn}
\begin{cases}
i(u_n)_t + (u_n)_{xx}+ a_\ve(x)|u_n|^{m-1}u_n=0, \quad \hbox{ in }\; \R_t\times \R_x, \\
u_n(-T_n)=Q(\cdot + v_0 T_n)e^{\frac i2 (\cdot) v_0}e^{-i(1-\frac 14 v_0^2)T_n}.
\end{cases}
\ee
In other words, $u_n$ is a solution of aNLS that at time $t=-T_n$ corresponds exactly to a solitary wave. It is clear that this function is in $H^1(\R)$; moreover, there exists a uniform constant $C=C(v_0)>0$ such that
$$
\|Q(\cdot - v_0 t)e^{\frac i2 (\cdot) v_0}e^{i(1-\frac 14 v_0^2)t} \|_{H^1(\R)}\leq C.
$$
According to Lemma \ref{Cauchy}, we have that $u_n$ is globally well-defined in $H^1(\R)$. 

The next step is to establish uniform estimates starting from a fixed time $t =-\frac 1{2} T_\ve <0$ large enough such that the soliton is sufficiently away from the region where the influence of the potential $a_\ve$ is present. This is the purpose of the following

\begin{prop}[Uniform estimates in $H^1$ for large times, see also \cite{Martel1}]\label{Ue1}~

There exist constants  $K,\mu>0$ and $\ve_0>0$ small enough such that for all $0<\ve <\ve_0$ and for all $n\in \N$ we have and for all $t\in [-T_n, -\frac 1{2}T_\ve]$,
\be\label{Ue2}
\|u_n(t) - Q(\cdot - v_0 t)e^{\frac i2 (\cdot) v_0}e^{i(1-\frac 14 v_0^2)t} \|_{H^1(\R)} \leq K e^{\mu \ve t}.
\ee
In particular, there exists a constant $C>0$ such that for all $t\in [-T_n, -\frac 1{2}T_\ve]$,
\be\label{Ue3}
\|u_n(t) \|_{H^1(\R)} \leq C.
\ee
\end{prop}
Using Proposition \ref{Ue1} we will obtain the existence of a \emph{critical element} $u_{0,*}\in H^1(\R)$, with several interesting properties. Indeed, let us consider the sequence $(u_n(-\frac 1{2}T_\ve))_{n\in \N}\subseteq H^1(\R)$. We claim the following result.

\begin{lem}[Compactness property]\label{CompProp}~

Given any number $ \delta >0$, there exist $\ve_0>0$ and a constant $K_0>0$ large enough such that for all $0<\ve<\ve_0$ and for all $n\in \N$, 
\be\label{CP1}
\int_{|x| >K_0} |u_n|^2(-\frac 1{2}T_\ve) <\delta.
\ee
\end{lem}

\begin{proof}
The proof is by now a standard result. See \cite{Martel} for the details.
\end{proof}

Let us come back to the proof of Theorem \ref{Tm1}. From (\ref{Ue3}) we have that 
$$
\|u_n(-\frac 12 T_\ve)\|_{H^1(\R)}\leq  C,
$$
independent of $n$. Thus, up to a subsequence we may suppose $u_n(-\frac 1{2} T_\ve ) \rightharpoonup u_{*,0} $ in the $H^1(\R)$ weak sense, and $u_n(-\frac 1{2} T_\ve ) \to u_{*,0}$ in $L^2_{loc}(\R)$, as $n\to +\infty$. In addition, from (\ref{CP1}) we have the strong convergence in $L^2(\R)$. 

\smallskip

Let $u_*=u_*(t)$ be the solution of (\ref{gKdV0}) with initial data $u_*(-\frac 1{2} T_\ve ) = u_{*,0}$. From Proposition \ref{Cauchy}  we also have $u_*\in C(\R,L^2(\R))$ (that is, $L^2$ local well-posedness plus conservation of mass). Thus, using the continuous dependence of $u_n$ and $u_*$, and the bound (\ref{Ue3}), we obtain $u_n(t) \to u_*(t)$ in $H^1(\R)$ for every $t \leq -\frac 1{2}T_\ve $.  Passing to the limit in (\ref{Ue2}) we obtain for all $t\leq -\frac 1{2}T_\ve$,
$$
\|u_*(t) -  Q(\cdot - v_0 t)e^{\frac i2 (\cdot) v_0}e^{i(1-\frac 14 v_0^2)t} \|_{H^1(\R)} \leq K e^{\ve \mu t},
$$
as desired. This finish the proof of the existence part of Theorem \ref{Tm1}. 
\subsection{Uniform $H^1$ estimates. Proof of Proposition \ref{Ue1}}
In this paragraph we explain the main steps of the proof of Proposition \ref{Ue1} in the $H^1$ case; for the general case the reader may consult \cite{Martel}. 

The first step in the proof is the following bootstrap property:

\begin{prop}[Bootstrap]\label{Uealpha}~

There exist constants $K,\mu, \ve_0>0$ such that for all $0<\ve<\ve_0$ the following is true. 
There exists $\al_0>0$ such that for all $0<\al<\al_0$, if for some $-T_{n,*} \in [-T_n, -\frac 1{2}T_\ve]$ and for all $t\in [-T_n, -T_{n,*}]$ one has
\be\label{Ue1alpha}
\|u_n(t) - Q(\cdot - v_0 t)e^{\frac i2 (\cdot) v_0}e^{i(1-\frac 14 v_0^2)t} \|_{H^1(\R)}\leq 2\al,
\ee
then, for all $t\in [-T_n, -T_{n,*}]$
\be\label{Ue2alpha}
\|u_n(t) - Q(\cdot - v_0 t)e^{\frac i2 (\cdot) v_0}e^{i(1-\frac 14 v_0^2)t} \|_{H^1(\R)}\leq K e^{\ve \mu t}.
\ee
\end{prop}

\begin{proof}[Proof of Proposition \ref{Ue1}, assuming the validity of Proposition \ref{Uealpha}]
Let $0<\al<\al_0$. Note that from (\ref{CPn}) there exists $t_0=t_0(n,\alpha)>0$ such that (\ref{Ue1alpha}) holds true for all $t\in [-T_n, -T_n +t_0]$. 
Now let us consider (we adopt the convention $T_{*,n}>0$)
\bee
& & -\tilde T_{*,n}  :=   \sup \{ t\in [-T_n, -\frac 12T_\ve] \ | \   \hbox{ for all $t' \in [-T_n, t ] $, }\; \\
& & \qquad \qquad \qquad    \| u_n(t') -Q(\cdot - v_0 t')e^{\frac i2 (\cdot) v_0}e^{i(1-\frac 14 v_0^2)t'} \|_{H^1(\R)} \leq  2\al \}.
\eee
Assume, by contradiction, that $-\tilde T_{*,n} < -\frac 12 T_\ve$. From Proposition \ref{Uealpha}, we have
$$
\| u_n(t') -Q(\cdot - v_0 t')e^{\frac i2 (\cdot) v_0}e^{i(1-\frac 14 v_0^2)t'} \|_{H^1(\R)} \leq K e^{\mu\ve t} \leq \alpha ,
$$
for $\ve$ small enough (recall that $t\leq -\frac 12T_\ve =-\frac 1{2v_0} \ve^{-1-\frac 1{100}}$), a contradiction with the definition of $\tilde T_{*,n}$.
\end{proof}
Now we are reduced to prove Proposition \ref{Uealpha}. 

\begin{proof}[Proof of Proposition \ref{Uealpha}]
The first step in the proof is to decompose the solution preserving a standard orthogonality condition. To obtain this, without loss of generality, by taking $T_{n,*}$ larger we may suppose that for all $t\in [-T_n, -T_{n,*}]$ 
\be\label{Uer}
\|u_n(t) -  Q(x-v_0 t - r_n(t)) e^{it} e^{\frac i2 v_0 x}e^{-\frac 14i v_0^2 t} e^{i g_n(t)}  \|_{H^1(\R)}\leq 2\al,
\ee
for all smooth $r_n, g_n$ satisfying $r_n(-T_n)=g_n(-T_n)=0$ and $\abs{r_n'(t)}\leq \frac 1{t^2}$. A posteriori we will prove that this condition can be improved and extended to any time $t\in [-T_n, -\frac 12T_\ve ]$.

For notational simplicity, in what follows we will drop the index $n$ on $-T_{*,n}$ and $u_n$, if no confusion is present.

\begin{lem}[Modulation]\label{FM}~

There exist $K,\mu, \ve_0>0$ and unique $C^1$ functions $\rho_0, \ga_0:[ -T_n, -T_*]\to \R $ such that for all $0<\ve<\ve_0$ the function $z$ defined by
\be\label{Ortho0}
z(t,x) := u(t,x) - \tilde R_{v_0}(t,x) ; \qquad \tilde R_{v_0}(t,x) := Q(y) e^{i\theta},  
\ee
with
\be\label{Ortho001}
y := x-v_0 t -\rho_0(t), \qquad \theta := t +\frac 12 v_0 x -\frac 14v_0^2 t +\ga_0(t),
\ee
satisfies for all $t\in [-T_n, -T_*]$,
\be\label{Ortho01}
\re \int_\R \bar z(t,x) Q'(y)e^{i\theta} dx= \ima \int_\R\bar z(t,x) Q(y)e^{i\theta} dx=0, 
\ee
\be\label{Ortho2}
\|z(t)\|_{H^1(\R)} \leq K\alpha, \quad \rho_0(-T_n) = \ga_0(-T_n) =0.
\ee

\smallskip

In addition, $z$ satisfies the following modified Schr\"odinger equation,
\bea\label{EqZ0}
& & i z_t + z_{xx}  + a_\ve(x)|\tilde R_{v_0} + z|^{m-1} (\tilde R_{v_0} + z)  - a_\ve(x) |\tilde R_{v_0}|^{m-1} \tilde R_{v_0} \nonumber \\ 
& & \qquad\qquad  - i\rho_0' (t) Q'(y)e^{i\theta}  -\ga_0'(t) \tilde R_{v_0}  +  (a_\ve(x)-1)  |\tilde R_{v_0}|^{m-1} \tilde R_{v_0} =0,
\eea
and
\be\label{rho0}
|\rho_0'(t)| + |\ga_0'(t)| \leq K\big[ e^{\ve \mu t} + \|z(t)\|_{H^1(\R)} + \|z(t)\|_{L^2(\R)}^2\big].
\ee
\end{lem}

\begin{proof}[Proof of Lemma \ref{FM}]

The proof of (\ref{Ortho01}) is a standard consequence of the Implicit Function Theorem, the definition of $T_*$  $(=T_{*,n})$, and the definition of $u_n(-T_n)$ given in (\ref{CPn}), see for example \cite{Martel} for a detailed proof. 
Similarly, the proof of (\ref{EqZ0}) follows after a simple computation.

Now we deal with (\ref{rho0}). Taking time derivative to the first identity in (\ref{Ortho01}) and using (\ref{EqZ0}), we get 
\bee
0 & = & -\ima \int_\R \overline{iz}_t Q'(y)e^{i\theta} +\re \int_\R \bar z  (Q'(y)e^{i\theta})_t  \\
& =& \ima \int_\R \big\{ \bar z_{xx}  + a_\ve(x)|\tilde R_{v_0} + z |^{m- 1}\overline{(\tilde R_{v_0} + z)} - a_\ve(x)|\tilde R_{v_0}|^{m-1}\overline{\tilde R_{v_0}}  \big\} Q'(y)e^{i\theta}    \\
& &  + \rho_0'(t)  \int_\R Q'^2 
+  \ima \int_\R (a_\ve(x)-1) |\tilde R_{v_0}|^{m-1} \overline{\tilde R_{v_0}} Q'(y)e^{i\theta} \\
& & +\re \int_\R \bar z \big\{ -(v_0 + \rho_0'(t)) Q''(y) + i(1-\frac 14 v_0^2 +\ga_0'(t))Q'(y) \big\} e^{i\theta} 
\eee
First of all, note that
\bee
& &  \ima \int_\R \big\{ z_{xx}  + a_\ve(x)| \tilde R_{v_0} + z|^{m-1} (\tilde R_{v_0} + z)  - a_\ve(x) |\tilde R_{v_0}|^{m-1} \tilde R_{v_0}  \big\} Q'(y)e^{-i\theta} =  \\
  & & \qquad \qquad = O(\|z(t)\|_{L^2(\R)} + \|z(t)\|_{L^2(\R)}^2).
\eee
On the other hand, from (\ref{ahyp}), (\ref{Ortho01}), the uniform bound on $\rho_0'(t)$ in the definition of $T_*$  and the exponential decay of $R$, we have
\be\label{dec1}
\abs{\ima \int_\R (a_\ve(x)-1) |\tilde R_{v_0}|^{m-1} \overline{\tilde R_{v_0}} Q'(y)e^{i\theta} } \leq  K e^{\ve \mu t}.
\ee
Indeed, first note that from (\ref{Uer}), by integrating between $-T_n $ and $t$ and using (\ref{Ortho01}) we get 
$$
\rho_0(t) \leq - \frac 1{T_n} -\frac 1t \leq \frac 2{T_\ve}\leq K\ve^{1+\frac 1{100}} .
$$
Thus $v_0 t +\rho_0(t) \leq v_0 t  +K \ve^{1+\frac 1{100}} \leq \frac 9{10}v_0t .$ Therefore, by possibly redefining $\mu>0$, we have from (\ref{ahyp}),
\bee
\abs{\int_\R  (a_\ve(x)-1) |\tilde R_{v_0}|^{m-1} \overline{\tilde R_{v_0}} Q'(y)e^{i\theta} }& \leq & K\int_{-\infty}^0e^{\mu \ve x}e^{- |x- v_0 t -\rho_0(t) |} dx \\
& & \qquad +  Ke^{v_0t+\rho_0(t)} \int_0^\infty e^{- x} dx  \\
&  \leq & K \exp \big[  \mu\ve (v_0t+\rho_0(t) )\big] + K  \exp \big[  \mu (v_0t+\rho_0(t) ) \big]\\
&  \leq &  Ke^{\mu\ve t}.
\eee
Finally,
\bee
& & \abs{\re \int_\R \bar z \big\{ -(v_0 + \rho_0'(t)) Q''(y) + i(1-\frac 14 v_0^2 +\ga_0'(t))Q'(y) \big\} e^{i\theta} }  \qquad \qquad \\
& & \qquad \qquad \qquad\qquad \qquad  \leq  K\|z(t)\|_{L^2(\R)} (1+|\rho_0'(t)| + |\ga_0'(t)|).
\eee
We arrive, for $\al$ small enough, to the following estimate
\be\label{Ae1}
\abs{\rho_0'(t)} \leq K( e^{\ve\mu t} + \|z(t)\|_{L^2(\R)} (1+ |\ga_0'(t)|) + \|z(t)\|_{L^2(\R)}^2).
\ee
\smallskip

Now we consider the second identity in (\ref{Ortho01}). Proceeding in a similar way as above, we obtain 
\be\label{Ae2}
\abs{\ga_0'(t)} \leq K( e^{\ve\mu t} + \|z(t)\|_{L^2(\R)} (1+|\rho_0'(t)|) + \|z(t)\|_{L^2(\R)}^2).
\ee
Collecting  estimates (\ref{Ae1})-(\ref{Ae2}) we obtain (\ref{rho0}). 
\end{proof}

\subsubsection{Almost conservation of mass, energy and momentum} 
Now let us recall that for all $-T_n\leq t\leq -\frac 12T_\ve$ we have $M[u](t)$ and $E_a[u](t)$ conserved. In addition, from (\ref{P}) we have   
$$
\partial_t P[u](t)  =  \frac \ve{m+1} \int_\R a'(\ve x) |u|^{m+1} \geq 0.
$$
Therefore
\be\label{Ine1}
E_a[u](t) - E_a[u](-T_n)  + (1+\frac 14 v_0^2)[M[u](t) -M[u](-T_n)] - v_0[P[u](t) -P[u](-T_n)]  \leq 0.
\ee
Similarly, note that in the considered region the solitary wave $\tilde R_{v_0}(t)$ is an almost solution of (\ref{aKdV}), in particular it must almost conserve the mass $M$ (\ref{M}) and the energy $E_a$ (\ref{Ea}), at least for large negative time. Indeed, arguing as in Lemma \ref{C2} (but with easier proof), one has 
\bea\label{dE01a}
 E_a[\tilde R_{v_0}](-T_n) - E_a[\tilde R_{v_0}](t) + (1 + \frac 12 v_0^2)\big[ M[\tilde R_{v_0}](-T_n)  -M[\tilde R_{v_0}](t)\big] \qquad & &  \nonumber \\
 \qquad \qquad \qquad  - v_0 \big[ P[\tilde R_{v_0}](-T_n) -P[\tilde R_{v_0}](t) \big]   \leq  K e^{\mu\ve t } ,& &  
\eea
for some constant $K>0$ and all time $t\in [-T_n, T_*]$.

The next step is the use the mass conservation law to provide a control of the $\tilde R_{v_0}(t)$ direction. Indeed, one has  
\be\label{CRD}
\abs{\re\int_\R \tilde R_{v_0} \bar z(t)} \leq  K\|z(-T_n)\|_{L^2(\R)}^2+ K\|z(t)\|_{L^2(\R)}^2 \leq K \sup_{t\in [-T_n, T_*]}\|z(t)\|_{L^2(\R)}^2.
\ee
for a constant $K>0$, independent of $\ve$.
On the other hand, note that
\bea\label{Expans1}
E_a[u](t) + (1 +\frac 14 v_0^2) M[u](t) -v_0 P[u](t) & = & E_a[\tilde R_{v_0}](t) + (1 +\frac 14 v_0^2) M[\tilde R_{v_0}](t) -v_0 P[\tilde R_{v_0}] \nonumber \\
& & - \re \int_\R (a_\ve-1) |\tilde R_{v_0}|^{m-1} \tilde R_{v_0} \bar z   + \mathcal F_0(t),
\eea
where $\mathcal F_0$ is the following quadratic functional
\bee
 \mathcal F_0(t) & := &  \frac 12\int_\R \big[ |z_x|^2 + (1+\frac 14v_0^2) |z|^2\big] - \frac {v_0}2 \ima\int_\R \bar zz_x \\
 & & \quad- \frac 1{m+1} \int_\R a_\ve(x) [ | \tilde R_{v_0} + z|^{m+1} -|\tilde R_{v_0}|^{m+1} - (m+1)|\tilde R_{v_0}|^{m-1} \re (\tilde R_{v_0} \bar z) ].
\eee
In addition, for any $t\in [-T_n, -T_*]$,
\be\label{dec3}
\abs{ \re\int_\R (a_\ve-1) |\tilde R_{v_0}|^{m-1} \tilde R_{v_0} \bar z }\leq K e^{\mu \ve  t}  \|z(t)\|_{L^2(\R)}. 
\ee
The proof of (\ref{Expans1}) is essentially an expansion of the energy-mass functional using the relation $u(t) =\tilde R_{v_0}(t) + z(t)$. The proof of (\ref{dec3}) is similar to (\ref{dec1}).

On the other hand, the functional $\mathcal F_0(t)$ above mentioned enjoys the following coercivity property: there exist $K,\la_0>0$ independent of $\ve$ such that for every $t\in [-T_n, -T_*]$
\be\label{F00}
\mathcal F_0(t) \geq \la_0 \|z(t)\|_{H^1(\R)}^2 -\Big|\re \int_\R \tilde R_{v_0}(t)  \bar z(t)\Big|^2 - K e^{\mu\ve t} \|z(t)\|_{L^2(\R)}^2 - K \|z(t)\|_{L^2(\R)}^3. 
\ee
This bound is a consequence of (\ref{Ortho01}) and Lemma \ref{surL}.

\subsubsection{End of proof of Proposition \ref{Uealpha}}
Now by using (\ref{Expans1}), (\ref{F00}), and the estimates (\ref{Ine1})-(\ref{dE01a}) and (\ref{CRD}) we finally get (\ref{Ue2alpha}). Indeed, note that 
\bee
Ke^{\mu \ve t} & \geq &  E_a[\tilde R_{v_0}](-T_n) - E_a[\tilde R_{v_0}](t) + (1 + \frac 12 v_0^2)\big[ M[\tilde R_{v_0}](-T_n)  -M[\tilde R_{v_0}](t)\big]    \\
& & \qquad \qquad  - v_0 \big[ P[\tilde R_{v_0}](-T_n) -P[\tilde R_{v_0}](t) \big]\\
&  \geq &  \mathcal  F_0(t) -   Ke^{\mu \ve t} -Ke^{\mu \ve t}\|z(t)\|_{L^2(\R)} -K\|z(t)\|_{L^2(\R)}^4. 
\eee
Finally, from (\ref{F00}) and \ref{CRD} we conclude that for some $K,\mu>0$,
$$
\|z(t)\|_{H^1(\R)} \leq K e^{\mu\ve t}.
$$
Plugging this estimate in (\ref{rho0}), we obtain that $\abs{\rho_0'(t)} \leq K e^{\mu\ve t},$
and thus after integration and by taking $\mu>0$ smaller if necessary, we get the final uniform estimate (\ref{Ue2alpha}) for the $H^1$-case. Note that we have also improved the estimate on $\rho_0'(t)$ assumed in (\ref{Uer}). This finishes the proof.
\end{proof}

\subsection{Proof of Uniqueness}\label{Uni}

First of all let us recall that the solution $u$ above constructed is in $C(\R, H^1(\R))$ and satisfies the exponential decay (\ref{minusTe}). Moreover, \emph{every} solution converging to a soliton satisfies this property. This property is crucial to obtain the uniqueness.

\begin{prop}[Exponential decay, see also \cite{Martel1, Mu2}]\label{Uni1}~

There exists $\ve_0>0$ such that for all $0<\ve<\ve_0$ the following holds. Let $v$ be a $C(\R,H^1(\R))$ solution of (\ref{gKdV0}) satisfying
$$
\lim_{t\to -\infty} \|v(t) - Q(\cdot - v_0 t)e^{\frac i2 (\cdot) v_0}e^{i(1-\frac 14 v_0^2)t} \|_{H^1(\R)} =0.
$$
Then there exist $K,\mu>0$ such that for every $t\leq -T_\ve$ one has
$$
\|v(t) -Q(\cdot - v_0 t)e^{\frac i2 (\cdot) v_0}e^{i(1-\frac 14 v_0^2)t} \|_{H^1(\R)}\leq K e^{\mu\ve t}.
$$
\end{prop}

\begin{proof}
Fix $\alpha>0$ small. Let $\ve_0=\ve_0(\al)>0$ small enough such that for all $0<\ve\leq \ve_0$ and $t\leq -T_\ve,$
$$
\|v(t) -Q(\cdot - v_0 t)e^{\frac i2 (\cdot) v_0}e^{i(1-\frac 14 v_0^2)t} \|_{H^1(\R)} \leq \alpha.
$$
Possibly choosing $\ve_0$ even smaller, we can apply the arguments of Proposition \ref{Uealpha} to the function $v(t)$ on the interval $(-\infty, -\frac 1{2}T_\ve]$ to obtain the desired result. Recall that a key fact to obtain this result is that 
$$
\partial_t P[v](t) \geq 0,
$$
which is not valid in the case of a pure soliton solution going to $x\sim +\infty$ as $t\to +\infty$.
\end{proof}

Now we are ready to prove the uniqueness part.

\begin{proof}[Sketch of proof of uniqueness]
Let $w(t) := v(t) -u(t)$. Then $w(t) \in H^1(\R)$ and satisfies the equation
\be\label{W}
\begin{cases}
iw_t + w_{xx} + a_\ve(x)|u+ w|^{m-1}(u+ w) - a_\ve(x)|u|^{m-1}u =0, \quad \hbox{ in } \; \R_t \times \R_x,\\
\|w(t)\|_{H^1(\R)}\leq K e^{\mu \ve t }\quad \hbox{ for all  } t\leq -\frac 1{2}T_\ve.  \quad \hbox{ (cf. Proposition \ref{Uni1}).}
\end{cases}
\ee
The idea is to prove that $w(t)\equiv 0$ for all $t\in \R$. For this purpose, one defines the second order functional
\bee
 \mathcal F_0[w](t)&  := & \frac 12\int_\R |w_x|^2 + \frac 12 (1+\frac 14v_0^2) \int_\R |w|^2   -\frac 12 v_0\ima\int_\R w_x \bar w \\
& &  -\frac 1{m+1}\int_\R a_\ve(x)[ |u+w|^{m+1} -|u|^{m+1} - (m+1) |u|^{m-1}\re(u \bar w) ].
\eee
It is easy to verify that 
\ben
\item Asymptotic at $-\infty$.
\be\label{LimF}
\lim_{t\to -\infty} \mathcal F_0[w] (t)  = 0. 
\ee
\item Lower bound. There exists $K>0$ such that for all $t\leq -\frac 1{2}T_\ve$,
$$
\mathcal F_0[w](t) \geq \tilde{\mathcal F}_0[w](t)  - K e^{\mu\ve t}\sup_{t'\leq t} \|w(t')\|_{H^1(\R)}^2,
$$
where
\bee
\tilde{\mathcal F}_0[w](t)& := & \frac 12\int_\R |w_x|^2 + \frac 12 \int_\R(1+\frac 14v_0^2)|w|^2 -\frac 12 v_0\ima\int_\R w_x \bar w \\
& & - \int_\R a_\ve(x)[ (m-1)|u|^{m-3}(\re(u\bar w))^2 +|u|^{m-1}|w|^2 ]. 
\eee
\item First derivative.
\bee
& & \mathcal F_0' [w](t)   =  \ima\int_\R \overline{iw}_t \big\{ w_{xx} -(1+\frac 14 v_0^2) w  +  |u+w|^{m-1}(u+w) -|u|^{m-1}u -iv_0 w_x \big\} \nonumber \\
& &\!\!\!\!\! \!\!\!\! \!\!+ \ima \int_\R a_\ve(x) \overline{iu}_t \big\{  |u+w|^{m-1}(u+w) -|u|^{m-1}u  - \frac 12(m+1) |u|^{m-1} w  -\frac 12 (m-1) |u|^{m-3}u^2 \bar w  \big\}.
\eee
\item Upper bound. There exists $K,\mu>0$ such that 
$$
\mathcal F_0[w](t) \leq K e^{\mu\ve t} \sup_{t'\leq t}\|w(t')\|_{H^1(\R)}^2.
$$
\een 
These estimates are proved similarly to the proof of Lemma \ref{Ka}, see also \cite{Martel1} for a similar proof. However, the functional $\mathcal F_0(t)$ is not necessarily coercive; so in order to obtain a satisfactory lower bound on $\mathcal F_0$, one has to modify the function $w$ in $(-\infty, -\frac 1{2}T_\ve]$ as follows. Let
$$
\tilde w(t,x) := w(t,x) +b_1(t)Q(x-v_0t) e^{i(1-\frac 14 v_0^2)t} e^{\frac i2 v_0x}  + b_2(t) Q'(x-v_0t) e^{i(1-\frac 14 v_0^2)t} e^{\frac i2 v_0x}, 
$$
with
\bee
b_1(t) & \! := \! &  -  \frac{1}{2M[Q] } \ima\int_\R \bar w(t,x) Q(x-v_0t) e^{i(1-\frac 14 v_0^2)t}e^{\frac i2 v_0x}dx ;\\
 b_2(t) & \! : =\! & - \frac{1}{2M[Q']} \re\int_\R \bar w(t,x) Q'(x-v_0t) e^{i(1-\frac 14 v_0^2)t} e^{\frac i2 v_0x}dx.
\eee
This new function satisfies
\ben
\item Orthogonality on the $Q$ and $Q'$ directions:
$$ 
\ima\int_\R \overline{\tilde w}(t) Q(x-v_0t) e^{i(1-\frac 14 v_0^2)t}e^{\frac i2 v_0x} =\re\int_\R \overline{\tilde w}(t)  Q'(x-v_0t) e^{i(1-\frac 14 v_0^2)t}e^{\frac i2 v_0x} = 0.
$$
\item Equivalence. There exists $C_1,C_2>0$ independent of $\ve$ such that
$$
C_1 \|w(t)\|_{H^1(\R)} \leq \|\tilde w(t)\|_{H^1(\R)} + \abs{b_1(t)} +  \abs{b_2(t)} \leq C_2 \|w(t)\|_{H^1(\R)}.
$$
Moreover,
$$
\tilde{\mathcal F}_0[\tilde w](t) = \tilde{\mathcal F}_0[w](t) + O(e^{\ve\mu t}\|w(t)\|_{H^1(\R)}^2).
$$
\item Control on the $Q$ direction: for some $K, \mu>0$,
$$
\abs{\re \int_\R \overline{\tilde w}(t,x) Q(x-v_0t) e^{i(1-\frac 14 v_0^2)t}e^{\frac i2 v_0x}} \leq K e^{\ve \mu t} \sup_{t'\leq t} \|w(t')\|_{H^1(\R)}. 
$$
This property is proved similarly to the proof of (\ref{c2rho2}): We use the fact that variation in time of the above quantity is of quadratic order in $\tilde w$.
\item Coercivity. There exists $\la>0$ independent of $t$ such that
$$
\tilde{\mathcal F}_0[\tilde w](t) \geq \la \| \tilde w(t)\|_{H^1(\R)}^2 -K\abs{\re \int_\R \overline{\tilde w}(t,x)Q(x-v_0t) e^{i(1-\frac 14 v_0^2)t}e^{\frac i2 v_0x}}^2.
$$
\item Sharp control. From the equivalence $w$-$\tilde w$ and the coercivity property we obtain, for some $K, \mu>0$,
\be\label{Sc}
\|\tilde w(t)\|_{H^1(\R)}  \leq  K e^{\ve \mu t/2} \sup_{t'\leq t} \|w(t')\|_{H^1(\R)},
\ee
and therefore
\be\label{Sc1}
|b_1(t)| +  |b_2(t)| \leq K e^{\ve \mu t/2} \sup_{t'\leq t} \|w(t')\|_{H^1(\R)}.
\ee
Note that the bounds on $b_1(t)$ and $b_2(t)$ are proved similarly to (\ref{rho2c2}).
\een

The proof of all these affirmations follows the argument of Proposition 6 in \cite{Martel1}, with easier proofs. 
Finally, from (\ref{Sc})-(\ref{Sc1}) we have for $\ve $ small enough and $t\leq -\frac 12T_\ve$,
$$
\|w(t)\|_{H^1(\R)} \leq  K e^{\ve\mu t} \sup_{t'\leq t} \|w(t')\|_{H^1(\R)} <\frac 12 \sup_{t'\leq t} \|w(t')\|_{H^1(\R)} .
$$
This inequality implies $w\equiv 0$, which gives the uniqueness. 
\end{proof}

\bigskip

\section{Proof of Proposition \ref{Tp1}}\label{B}

The proof of the stability result  (\ref{S}) is based in a standard Weinstein argument. Let us assume that for some $K>0$ fixed,
\be\label{48bon}
\|u(t_1)-  \la_\infty Q_{c_\infty} (\cdot  - X_0)e^{\frac i2 v_\infty x } e^{i\ga_0} \|_{H^1(\R)}\leq K \ve^{p_m},
\ee
with $\la_\infty, v_\infty, c_\infty$ defined in Theorem A, $p_m$ defined in (\ref{pm}), and $\ga_0 \in \R.$
From the local and global Cauchy theory (cf. Lemma \ref{Cauchy}), we know that $u$ is well defined for all $t\geq t_1$.

\medskip
\noindent
{\bf Step 0. Preliminars.}
In order to simplify the calculations, note that from (\ref{simpli}) the function $v(t,x):= \la_\infty^{-1} u(t,x)$ solves
$$
iv_t + v_{xx}   +  \frac{a_\ve}{2} |v|^{m-1} v=0 \quad \hbox{ on } \R_t\times \R_x.
$$
The energy is now given by
\be\label{tEa}
\tilde E_a[v] := \frac 12 \int_\R |v_x|^2  -\frac 1{m+1} \int_\R \frac{a_\ve}2|v|^{m+1};
\ee
the mass (\ref{M}) and momentum (\ref{P}) remain unchanged. In addition (\ref{48bon}) now becomes
\be\label{48bon1}
\| v(t_1)- Q_{c_\infty}(\cdot  - X_0) e^{\frac i2 xv_\infty} e^{i\ga_0}  \|_{H^1(\R)}\leq  \tilde K \ve^{p_m}.
\ee
With a slight abuse of notation we will {\bf rename} $v:=u$, $\tilde K:= K$, and we will assume the validity of (\ref{48bon1}) for $u$. In addition, and if no confusion is present, we will drop the tilde in (\ref{tEa}). The parameters $X_0$ and $c_\infty$ remain unchanged.

Let $D_0>2K$ be a large number to be chosen later,  and set 
\bea\label{Tprime}
   T^* \!\!\! & : = &   \sup\Big\{t\geq t_1 \ | \ \hbox{ for all }  \ t'\in [t_1, t), \ \hbox{ there exist } \  r_2(t'), g_2(t') \in\R  \hbox{ smooth} \nonumber \\
 & & \quad  \hbox{  such that  }   
| r_2'(t')| + | r_2 (t_1)+v_\infty t_1 - X_0 | \leq \frac {v_\infty}{100}, \hbox{ and } \;   \nonumber \\
& & \;   \| u(t')- Q_{c_\infty}(\cdot -v_\infty t -r_2(t'))  \exp \big\{ \frac i2 xv_\infty -\frac i4 v_\infty^2 t  + i g_2(t) \big\}   \|_{H^1(\R)}\le D_0 \ve^{p_m} \Big\}.
\eea
Observe that $T^*>t_1$ is well-defined since $D_0>2K$, (\ref{48bon}) and the continuity of $t\mapsto u(t)$ in ${H^1(\R)}$. The objective is to prove that $T^*= +\infty$, and thus (\ref{S}). Therefore, for the sake of contradiction, in what follows {\bf we shall suppose} $T^* <+\infty$.

The first step to reach a contradiction is now to decompose the  solution on $[t_1,T^*]$ using modulation theory around the soliton. In particular, we will find some special $\rho_2(t), \ga_2(t)$ satisfying the hypothesis in (\ref{Tprime}) but with  
\be\label{onehalf}
\sup_{t\in [t_1, T^*]} \big\| u(t)- Q_{c_\infty}(\cdot - v_\infty t- \rho_2(t)) \exp \big\{ \frac i2 x v_\infty -\frac i4v_\infty^2 t +i\ga_2(t)\big\}  \big\|_{H^1(\R)}\le \frac 12 D_0 \ve^{p_m},
\ee
a contradiction with the definition of $T^*$.

\medskip
\noindent
{\bf Step 1. Modulation on the degenerate directions.} We will prove the following

\begin{lem}[Modulated decomposition]\label{3Dr}~

For $\ve>0$ small enough, independent of $T^*$, there exist  $C^1$ functions $\rho_2, c_2, \tilde \ga_2$, defined on $[t_1,T^*]$, with $c_2(t)>0$ and  such that the function $z(t)$ given by
\be\label{eta1a}
z(t,x):=u(t,x)- \tilde R(t,x),
\ee
where $\tilde R(t,x):= Q_{c_2(t) }(y) e^{i\Gamma}$, with
$$
y: = x- v_\infty t - \rho_2(t) \; \hbox{ and } \; \Gamma := \frac 12 xv_\infty + \int_{t_1}^t c_2(s) ds - \frac 14 v_\infty^2 t + \tilde \ga_2(t),
$$
satisfies for all $t\in [t_1,T^*],$
\bea 
&& \re \int_\R  \tilde R(t) \bar z(t) = \ima \int_\R \tilde R(t) \bar z(t) =  \re \int_\R Q_{c_2(t)}'(y) e^{i\Gamma} \bar z(t)=0,  \label{10a}\\ 
&& \|z(t)\|_{H^1(\R)}+ |c_2(t) - c_\infty |    \leq  K D_0\ve^{p_m},\; \hbox{ and }\label{11a}\\
&& \!\! \!\!\!\! \!\!\!\!\!\!\!\! \!\!\|z(t_1)\|_{H^1(\R)}+ |\rho_2(t_1) + v_\infty t_1 - X_0 | +   |c_2(t_1) -c_\infty |  + | \tilde \ga_2(t_1)  -\frac14 v_\infty t_1 -\ga_0| \leq  K \ve^{p_m} \label{12a},
\eea
where $K$  is not depending on $D_0$.  In addition, $z (t)$ now satisfies the following modified NLS equation
\bea\label{13a}
& &   iz_t + z_{xx}  + \frac 12a_\ve(x)\big[ |\tilde R + z|^{m-1}(\tilde R + z)  -  |\tilde R|^{m-1} \tilde R \big] \qquad\nonumber \\  
& & \qquad  + \; i c_2'(t)\Lambda Q_{c_2}e^{i\Gamma} - \tilde \ga'_2(t) Q_{c_2}e^{i\Gamma}- i\rho_2'(t)Q'_{c_2}e^{i\Gamma} +  (\frac 12a_\ve(x) -1)Q_{c_2}^m e^{i\Gamma}  =0.
\eea
Furthermore, for some constant $\mu>0$ independent of $\ve$, we have the following estimates:
\be\label{rho2c2}
|\rho_2'(t)|  \leq  K \Big[ \int_\R e^{-\mu |y|}|z|^2(t,x) dx\Big]^{\frac 12}+ K\int_\R e^{-\mu |y|}|z|^2(t,x) dx + K e^{-\mu\ve t};
\ee
\be\label{c2rho2}
\frac{|c_2'(t)|}{c_2(t)} \leq K \int_\R e^{-\mu |y|}|z|^2(t,x) dx  + Ke^{-\mu\ve t}\|z(t)\|_{H^1(\R)}, 
\ee
and finally
\be\label{ga2}
|\tilde \ga_2'(t)| \leq K \Big[ \int_\R e^{-\mu |y|}|z|^2(t,x) dx\Big]^{\frac 12} +  K \int_\R e^{-\mu |y|}|z|^2(t,x) dx  + K e^{-\mu\ve t}\|z(t)\|_{H^1(\R)} + K e^{-\ve\mu t}.
\ee
\end{lem}
 
\begin{rem}
Note that from (\ref{11a}) and taking $\ve$ small enough we have an improved the bound on $\rho_2(t)$. Indeed, for all $t\in [t_1, T^*]$,
$$
|\rho_2' (t) | +  |\rho_2 (t_1) +v_\infty t_1 - X_0 |  \leq 2D_0\ve^{p_m}.
$$
Thus, in order to reach a contradiction, we only need to show (\ref{onehalf}). Observe that these inequalities impliy that the soliton position is far away from the interaction region. 
\end{rem}

\begin{proof}[Proof of Lemma \ref{3Dr}]
As in Lemma \ref{FM} and \ref{DEFZ}, the proof of (\ref{eta1a})-(\ref{12a}) are based in a Implicit Function Theorem application. 

On the other hand, equation (\ref{13a}) is a simple computation, completely similar to (\ref{EqZ0}) and (\ref{Eqz1}).
Finally, estimates (\ref{rho2c2})-(\ref{ga2}) are similar to the proof of (\ref{rho0}). We skip the details.
\end{proof}


\medskip
\noindent
{\bf Step 2. Almost conserved quantities and monotonicity.}

\begin{lem}[Almost conservation of modified mass, energy and momentum]\label{C2}~

Consider $ M= M[\tilde R]$, $E_a=E_a[\tilde R ]$ and $P[\tilde R]$ the mass, energy and momentum of the \emph{soliton} $\tilde R$ (cf. (\ref{eta1a})). Then for all $t\in [t_1, T^*]$ we have
\bea\label{dE0}
& &  M[\tilde R](t)  =  c_2^{2\theta}(t) M[Q]; \\
& & E_a[\tilde R](t)  = c_2^{2\theta}(t) (\frac 14 v_\infty^2 - \la_0 c_2(t) ) M[Q] +O(e^{-\ve \mu t}); \label{dE01}\\
& & P[\tilde R](t) =  \frac 12v_\infty c_2^{2\theta}(t)  M[Q].\label{dE0P} 
\eea
Furthermore, we have the bound
\bea\label{dE02}
& & \big| E_a[\tilde R](t_1) -E_a[\tilde R](t)  +  (c_2(t_1) +\frac 14 v_\infty^2 ) (M[\tilde R](t_1) -M[\tilde R](t)) - v_\infty (P[\tilde R](t_1) -P[\tilde R](t)) \big|   \nonumber \\
& & \qquad \qquad \qquad  \leq K \abs{ \Big[\frac{c_2(t)}{c_2(t_1)}\Big]^{2\theta}-1}^2 +K e^{-\ve\mu t_1}. 
\eea
\end{lem}

\begin{proof}
The first and third identities, namely (\ref{dE0}) and (\ref{dE0P}), are direct computations. We consider (\ref{dE01}). Here we have
\bee
E_a[\tilde R](t) & =&\frac 12\int_\R |\tilde R_x|^2  - \frac 1{2(m+1)} \int_\R a_\ve(x) |\tilde  R|^{m+1}  \\
& =&   c_2^{2\theta}(t)  \Big[ c_2(t) (\frac 12 \int_\R Q'^2 - \frac 1{m+1} \int_\R  Q^{m+1} )  +  \frac 18 v_\infty^2 \int_\R Q^2 \Big]  \\
& & \qquad \qquad \qquad  +\frac 1{m+1} \int_\R (1-\frac {a_\ve}2 )|\tilde R|^{m+1}. 
\eee
Similarly to (\ref{dec3}), we have
$$
\abs{ \int_\R (1- \frac 12 a_\ve)|\tilde R|^{m+1}} \leq Ke^{-\mu \ve t }, 
$$
for some constants $K,\mu>0$. On the other hand, from Appendix \ref{AidQ} we have that 
$$
\frac 12 \int_\R Q'^2 - \frac 1{m+1} \int_\R  Q^{m+1} = - \frac {\la_0}2 \int_\R Q^2, \quad  \la_0= \frac{5-m}{m+3},
$$
and thus
$$
E_a[\tilde R](t) =   c_2^{2\theta}(t)  ( \frac 14 v_\infty^2 -\la_0 c_2(t) ) M[Q] + O( e^{-\mu\ve t}).
$$
Summing up (\ref{dE0}), (\ref{dE01}) and (\ref{dE0P}), we obtain
$$
E_a[\tilde R](t) + (c_2(t_1) +\frac 14 v_\infty^2 ) M[u](t) - v_\infty P[\tilde R](t) = c_2^{2\theta}(t) ( c_2(t_1) -\la_0 c_2(t) ) M[Q]  + O(e^{-\ve \mu t}).
$$
In particular,
\bee
& & E_a[\tilde R](t_1) -E_a[\tilde R](t) + (c_2(t_1) + \frac 14 v_\infty^2) ( M[\tilde R](t_1) - M[\tilde R](t) ) -v_\infty [P[\tilde R](t_1) -P[\tilde R](t) ] = \\
& & \; = \la_0 M[Q] \Big[  c_2^{2\theta +1}(t) -  c_2^{2\theta +1}(t_1)   - \frac{c_2(t_1)}{ \la_0} [ c_2^{2\theta}(t) -c_2^{2\theta}(t_1) ] \Big] 
+ O(e^{-\ve \mu t_1}).
\eee
To obtain the last estimate (\ref{dE02}) we perform a Taylor development up to the second order (around $y=y_0$) of the function $g(y):= y^{\frac {2\theta+1}{2\theta}}$; and where $y:= c_2^{2\theta}(t)$ and $y_0 := c_2^{2\theta}(t_1)$. Note that $\frac{2\theta+1}{2\theta} = \frac{1}{\la_0}$ and $y_0^{1/2\theta} = c_2(t_1)$. The conclusion follows at once.
\end{proof}

Now our objective is to estimate the quadratic term involved in (\ref{dE02}). Following \cite{MMT}, we use  the mass conservation law identity. From (\ref{eta1a}) -(\ref{10a}) we have
\be\label{Kc2}
c_2^{2\theta}(t)M[Q] + \frac 12 \int_\R |z(t)|^2 = c_2^{2\theta}(t_1)M[Q] + \frac 12 \int_\R |z(t_1)|^2.
\ee
From here we obtain
\be\label{dE023}
(\ref{dE02}) \leq K \|z(t)\|_{L^2(\R)}^4 +  \|z(t_1)\|_{L^2(\R)}^4 + Ke^{-\ve\mu t} ,
\ee
for some $K,\mu>0$, independent of $D_0 $ and $\ve$.

\medskip
\noindent
{\bf Step 3. Energy estimates.} Let us now introduce the second order functional 
\bee
\mathcal F_2(t) & := & \frac 12\int_\R \Big\{ |z_x|^2 + (c_2(t_1) + \frac 14 v_\infty^2 )|z|^2 \Big\}- \frac 12 v_\infty  \ima \int_\R z_x \bar z \\
& & \quad  -\frac 1{2(m+1)} \int_\R a_\ve(x) [ |\tilde R+z|^{m+1} - |\tilde R|^{m+1} - (m+1)|\tilde R|^{m-1}\re( \tilde R\bar z) ].
\eee
This functional have the following properties.
\begin{lem}[Energy expansion]\label{EE3}~

Consider $M[u]$, $E_a[u]$ and $P[u]$ the mass, energy and momentum defined in (\ref{M}), (\ref{tEa}) and (\ref{P}). Then we have for all $t\in [t_1,T^*]$,
\bee
& & E_a[u](t) +  (c_2(t_1) + \frac 14 v_\infty^2) M[u](t)  -v_\infty P[u](t)  =  \\
& & \qquad E_a[\tilde R](t) + ( c_2(t_1) + \frac 14 v_\infty^2) M[\tilde R](t) -v_\infty P[\tilde R](t)  + \mathcal F_2(t) +  O(e^{-\mu \ve  t}\|z(t)\|_{H^1(\R)}).
\eee
\end{lem}
\begin{proof}
Using the orthogonality condition (\ref{10a}), we have
\bee
E_a[u](t) &= & E_a[\tilde R] + \re \int_\R \bar z [ -\tilde R_{xx}  - |\tilde R|^{m-1} \tilde R ]  + \frac 12 \int_\R |z_x|^2   + \re \int_\R (1- \frac{a_\ve}2 )|\tilde R |^{m-1}\tilde R \bar z  \\
 & & \quad   -\frac 1{2(m+1)} \int_\R a_\ve(x) [ |\tilde R+z|^{m+1} -|\tilde R|^{m+1} - (m+1) |\tilde R|^{m-1} \re (\tilde R \bar z) ].
\eee
Moreover, following (\ref{dec1}), we easily get
$$
\abs{\re \int_\R \bar z (1-\frac 12 a_\ve) |\tilde R|^{m-1} \tilde R } \leq K e^{-\mu \ve  t}\|z(t)\|_{H^1(\R)}.
$$
Similarly, by using (\ref{10a}),
$$
M[u](t)  = M[\tilde R] +  \frac 12\int_\R |z|^2,
$$
and
$$
P[u](t)  = P[\tilde R](t) +\ima \int_\R \tilde R_x \bar z + \frac 12 \ima \int_\R z_x\bar z.
$$
Collecting the above estimates, we have
\bee
& & E_a[u](t) +  (c_2(t_1) + \frac 14 v_\infty^2) M[u](t) -v_\infty P[u](t)= \\
& & \;  E_a[\tilde R](t) + ( c_2(t_1) + \frac 14 v_\infty^2) M[\tilde R](t) -v_\infty P[\tilde R](t)  + \mathcal F_2(t) +  O(e^{-\mu \ve  t}\|z(t)\|_{H^1(\R)}).
\eee
Here we have used (\ref{10a}), the equation satisfied by $Q_{c_2}$ and the identity
$$
\re \int_\R \bar z [ -\tilde R_{xx} - |\tilde R|^{m-1} \tilde R + iv_\infty \tilde R_x ] =0. 
$$
This concludes the proof.
\end{proof}

\begin{lem}[Modified coercivity for $\mathcal F_2$]\label{Coer3}~

There exists $\ve_0>0$ such that for all $0<\ve<\ve_0$ the following hold. There exist $K,\nu, \mu>0$, independent of $K^*$ such that for every $t\in [t_1, T^*]$
\be\label{Co3}
\mathcal F_2(t) \geq \nu \|z (t)\|_{H^1(\R)}^2   - K e^{-\mu\ve t} \|z (t)\|_{L^2(\R)}^2 +O( \|z(t)\|_{L^2(\R)}^3). 
\ee
\end{lem}

\begin{proof}
First of all, note that 
\bee
\mathcal F_2(t) & = & \frac 1{2}\int_\R \big\{ z_x^2 + (c_2(t_1)  + \frac 14 v_\infty^2)z^2\big\} - \frac 12 v_\infty \ima \int_\R \bar z z_x\\
& & \quad -  \int_\R  [ |\tilde R|^{m-1} |z|^2 + (m-1) |\tilde R|^{m-3} [\re(\tilde R \bar z)]^2 ]\\
& & \quad - \frac 1{2} \int_\R (a_\ve (x)-2) [ |\tilde R|^{m-1} |z|^2 + (m-1) |\tilde R|^{m-3} [\re(\tilde R \bar z)]^2 ]+ O(\|z(t)\|_{H^1(\R)}^3 )
\eee
Since $(a_\ve(x) -2)$ is exponentially decreasing along the region where the soliton $\tilde R$ is supported, we have
$$
\abs{\int_\R (a_\ve (x)-2) [ |\tilde R|^{m-1} |z|^2 + (m-1) |\tilde R|^{m-3} [\re(\tilde R \bar z)]^2 ] } \leq K e^{-\ve\mu t} \|z(t)\|_{L^2(\R)}.
$$
(cf. (\ref{dec1} for a similar computation.) From Lemma \ref{surL} and (\ref{10a}) we have for $t\geq t_1$,
$$
\mathcal F_2(t)  \geq    \nu \|z(t)\|_{H^1(\R)}^2 -Ke^{-\ve\mu t}\|z(t)\|_{L^2(\R)}^2 -K\|z(t)\|_{H^1(\R)}^3,
$$ 
as desired.
\end{proof}

\medskip
\noindent
{\bf End of the proof.} Now we prove that our assumption $T^*<+\infty$ leads inevitably to a contradiction. Indeed, from Lemmas \ref{EE3} and \ref{Coer3}, the mass and energy conservation, and the positivity of (\ref{dPa}), we have for all $t\in [t_1, T^*]$ and for some constant $K>0,$
\bee
 & & \|z(t)\|_{H^1(\R)}^2 \leq  \quad  K \mathcal F(t_1) + Ke^{-\mu \ve t_1}\sup_{t\in [t_1, T^*]} \|z(t)\|_{L^2(\R)}  + K\sup_{t\in [t_1, T^*]} \|z(t)\|_{L^2(\R)}^3 \\
 & & +  \big| E_a[\tilde R](t_1) -E_a[\tilde R](t)  +  (c_2(t_1) +\frac 14 v_\infty^2 ) (M[\tilde R](t_1) -M[\tilde R](t)) - v_\infty (P[\tilde R](t_1) -P[\tilde R](t)) \big|. 
\eee
From Lemmas \ref{3Dr} and \ref{dE023} we have
$$
\|z(t)\|_{H^1(\R)}^2  \leq  K \ve^{2p_m}  +  K \sup_{t\in [t_1, T^*]} \|z(t)\|_{H^1(\R)}^4 + K e^{-\ve\mu t_1} D_0 \ve^{p_m}. 
$$
Collecting the preceding estimates we have for $\ve>0$ small and $D_0=D_0(K)$ large enough
$$
\|z(t)\|_{H^1(\R)}^2 \leq \frac 14D_0^2 \ve^{2p_m}.
$$
This estimate together with (\ref{Kc2}) and (\ref{12a}) gives $|c_2(t)-c_\infty| \leq K \ve^{p_m}$, independent of $D_0$, which contradicts the definition of $T^*$. The conclusion is that 
 $$
 \sup_{t\geq t_1} \big\| u(t)- Q_{c_\infty}(\cdot - v_\infty t- \rho_2(t)) \exp \big\{ \frac i2 x v_\infty -\frac i4v_\infty^2 t +i\ga_2(t)\big\}\big\|_{H^1(\R)} \leq K \ve^{p_m}.
 $$
 This finishes the proof of (\ref{S}).

\begin{rem}
Note that from the proof and the mass conservation law, we have the following additional information: 
\be\label{Pc}
\|z(t)\|_{L^2(\R)}^2  =  2 (1- \la_\infty^2 c_2^{  \frac 2{m-1} -\frac 12}(t)) M[Q], \quad \hbox{ for all } t\geq t_1.
\ee
and thus 
$$
 \limsup_{t\to +\infty }c_2(t) \leq  c_\infty.
$$
We believe that there is no equality in the above property, in any case, provided $0<\ve<\ve_0$ small enough.
\end{rem}

\bigskip

\section{Proof of Proposition \ref{prop:decomp}}\label{CDE}

In this section we prove the decomposition result for the error $S[\tilde u]$ associated to the approximate solution $\tilde u$.
First of all, it is easy to verify that 
$$
S[\tilde u] = S[\tilde R] + \mathcal L [w]  + \tilde N[w], 
$$
where 
$$
\mathcal L[w] := iw_t + w_{xx}  +  \frac{a(\ve x)}{2a(\ve\rho)} Q_c^{m-1}(y) [ (m+1) w + e^{2i\Theta}(m-1) \bar w],  
$$
and
$$
\tilde N[w] := a(\ve x) \big\{ |\tilde R + w|^{m-1}(\tilde R + w) -|\tilde R|^{m-1} \tilde R -  \frac{Q_c^{m-1}(y) }{2a(\ve\rho)} [ (m+1) w +  e^{2i\Theta}(m-1) \bar w] \big\}.
$$

In the next Claim we expand the first term, $S[\tilde R]$.

\begin{Cl}[Decomposition of $S{[}\tilde R{]}$]\label{lem:SQ}~

\ben
\item Suppose $2\leq m<3$. Then one has
\be\label{eq:SQm}
S[\tilde R] = \big[  F_0^R(t, y)   + \ve  F_1^R(t, y) + \ve^2 F_2^R(t, y)  + \ve^3  f^R (t) F_c^R (y) \big] e^{i\Theta}, 
\ee
where 
\bea
F_0^R(t, y) & := &    - \frac 12 (v'(t) -  \ve f_1(t) ) \frac {yQ_c(y)}{\tilde a(\ve \rho(t))}   + i( c'(t)- \ve f_2(t)) \frac{\Lambda Q_c(y)}{\tilde a(\ve \rho(t))}  \nonu \\
& & \quad   -(\ga'(t) + \frac 12 v'(t)\rho (t)) \frac{Q_c(y)}{\tilde a(\ve \rho(t))}  \nonu \\
& & \quad - i(\rho'(t) -v(t) ) [  \frac{Q_c'(y)}{\tilde a(\ve \rho(t))} - \frac{\ve \tilde a'(\ve \rho(t))}{\tilde a^2(\ve\rho(t))} Q_c(y)]  \in \mathcal Y, \label{F0R}
\eea
$f_1, f_2$ are given by (\ref{f1f2}),
\bea
 F_1^R(t ,y) & := & \frac{a' (\ve \rho(t))}{\tilde a^m (\ve \rho(t))}yQ_c(y)\big[ Q_c^{m-1}(y) -\frac {4c(t)}{m+3}\big]  \nonu  \\
  & & \qquad + \frac{ia' (\ve \rho(t)) v(t) }{\tilde a^m(\ve \rho (t))}\big[ \frac{4c(t)}{5-m}\Lambda Q_c (y)-\frac{1}{m-1}Q_c(y)\big], \label{F1R}
\eea
and $\abs{f^R(t)}\leq K$, $ F_c^R \in \mathcal Y$. Finally, for every $t\in [-T_\ve, T_\ve]$
$$
\| \ve^2 F_2^R(t,y) + \ve^3 f^R(t)  F_c^R (y) \|_{H^1(\R)} \leq K\ve^2( e^{-\ve\mu|\rho(t)|} +\ve).
$$

\item Now suppose $3\leq m<5$. Then one has

\be\label{eq:SQ34}
S[\tilde R] = \big[   F_0^R(t, y)   + \ve F_1^R(t, y) + \ve^2 F_2^R(t, y)  +  \ve^3 F_{3}^R(t, y)  +\ve^4  f(t) F_c^R(y) \big] e^{i\Theta}, 
\ee
where $F_0^R$, is given now by the expression
\bea
F_0^R(t, y) & := &    - \frac 12 (v'(t) -  \ve f_1(t) ) \frac {yQ_c(y)}{\tilde a(\ve \rho(t))}   + i( c'(t)- \ve f_2(t)) \frac{\Lambda Q_c(y)}{\tilde a(\ve \rho(t))}  \nonu \\
& & \quad   -(\ga'(t) + \frac 12 v'(t)\rho (t) -\ve^2 f_3(t)) \frac{Q_c(y)}{\tilde a(\ve \rho(t))}  \nonu \\
& & \quad - i(\rho'(t) -v(t) -\ve^2 f_4(t)) [  \frac{Q_c'(y)}{\tilde a(\ve \rho(t))} - \frac{\ve \tilde a'(\ve \rho(t))}{\tilde a^2(\ve\rho(t))} Q_c(y)]  \in \mathcal Y, \label{F0R2}
\eea
$F_1^R$ is given by (\ref{F1R}), 
\be\label{F2R}
 F_2^R(t,y) := \frac{a''(\ve \rho(t))}{2 \tilde a^{m} (\ve \rho(t)) }y^2 Q_c^m (y) -\frac{f_3(t)}{\tilde a(\ve\rho(t))} Q_c(y) -i \frac{f_4(t)}{\tilde a(\ve\rho(t))} Q_c'(y),
\ee
and $\abs{f^R(t)}\leq K$, $ F_c^R \in \mathcal Y$. Moreover, for every $t\in [-T_\ve, T_\ve]$
$$
\| \ve^3 F_3^R (t,y) + \ve^4 f(t) F_c^R (y) \|_{H^1(\R)} \leq K\ve^3( e^{-\ve\mu|\rho(t)|} +\ve).
$$
\een
\end{Cl}

\begin{proof}[Proof of Claim \ref{lem:SQ}.]
We prove the worst case, namely $3\leq m<5$. The remaining case is easier to handle and we skip the details. 

Recall the definitions of $\tilde R$, $y$ and $\Theta$ from (\ref{defALPHA})-(\ref{param0}). We  have 
\bea
S[\tilde R] & = & i\tilde R_t + \tilde R_{xx}  + a_\ve(x) |\tilde R|^{m-1} \tilde R \nonu \\
&= &  -\big[\frac 12  x v'  + c + \ga' - \frac 14 v^2 \big]\frac 1{\tilde a}Q_c(y)e^{i\Theta} - \frac{i\rho'}{\tilde a} Q_c'(y)e^{i\Theta} +  \frac{ i c'}{\tilde a} \Lambda Q_c e^{i\Theta} \nonu \\
& & + \frac 1{\tilde a}\big[ Q_c'' + iv Q_c' -\frac 14 v^2 Q_c +  \frac {a(\ve x)}{a(\ve \rho)} Q_c^m  \big] e^{i\Theta}   - \frac{ i\ve a' \rho' }{(m-1)\tilde a^m} Q_c e^{i\Theta} . \label{SR}  
\eea
Now we perform a Taylor expansion of the term $a(\ve x)$ based at $x=\rho(t)$, as in \cite{Mu2}. From (\ref{SR}),
\bee
S[\tilde R] & = & \frac 1{\tilde a}\big[ \ve \frac{a' }{a}Q_c^{m-1} -\frac 12 v'  \big] yQ_c e^{i\Theta} + \frac{i}{\tilde a}\big[ c'\Lambda Q_c -\frac{ \ve a' v}{(m-1)a} Q_c\big] e^{i\Theta} -\frac 1{\tilde a}(\ga' + \frac 12 v' \rho) Q_ce^{i\Theta} \\
& &  - \frac{i}{\tilde a} (\rho' -v)\big[ Q_c'(y) + \frac{ \ve a' }{(m-1)a} Q_c \big] e^{i\Theta}  + \frac{\ve^2a''}{2 \tilde a^{m}} y^2 Q_c^m e^{i\Theta} \\
& & \qquad +  \frac{\ve^3a^{(3)}}{6 \tilde a^{m}} y^3 Q_c^m e^{i\Theta} +  \ve^4 f(t)F_c(y)e^{i\Theta} \\
& = : &  \big[ F_0^R(t, y) + \ve F_1^R(t, y) + \ve^2 F_2^R(t, y) + \ve^3 F_3^R (t,y)  + \ve^4 f^R(t)F_c^R(y)\big] e^{i\Theta}.
\eee
Additionally, we have $\abs{f(t)}\leq K$ and $ F_c^R (y)\in \mathcal Y$. In conclusion,
$$
\| \ve^3F_4(t,y) + \ve^4 f^R(t) F_c^R (y) \|_{H^1(\R)} \leq K\ve^4( e^{-\ve\mu|\rho(t)|} +\ve).
$$
This finishes the proof.
\end{proof}

Next, we consider the linear term. As above, we need to consider three different cases. Recall  that $\Lambda A_c(t,y) =\partial_c A_c(t, y) $.

\begin{Cl}[Decomposition of $\mathcal L{[}w{]}$]\label{lem:dSKdVw}~

\ben
\item Suppose $2\leq m< 3$. Then
\bea
\mathcal L[w] & = & - \ve \big[ \mathcal L_+ (A_{1,c})  + i \mathcal L_- (B_{1,c}) \big] e^{i\Theta}  -  (\ga' + \frac 12 v' \rho)w  -\frac 12  (v' -\ve f_1) y w \nonumber \\
& &   - i (\rho' -v)  w_y  + i\ve (c' - \ve f_2 )\partial_c w  + \ve^2 f^L(t)  F^{L}_c(y) e^{i\Theta}  .\label{Lwm}
\eea
Furthermore, suppose that $(A_{1,c}, B_{1,c}) $ satisfy (\ref{IP}). Then there exist $K, \mu>0$ such that
\be\label{23}
\| \ve^2 f^{L}(t)F^{L}_c e^{i\Theta}\|_{H^1(\R)}\leq K \ve^2(e^{-\ve\mu |\rho(t)|} +\ve) .
\ee

\item 
Consider now the case $3\leq m<5$. Here one has
\bea
\mathcal L[w] & = & -\sum_{k=1}^2 \ve^{k} \big[ \mathcal L_+ (A_{k,c})  + i \mathcal L_- (B_{k,c}) \big] e^{i\Theta}  -\frac12  ( v' - \ve f_1) y w + i ( c' - \ve f_2) \partial_c w   \nonumber \\
& &  \quad -  (\ga' + \frac 12 v' \rho -\ve^2 f_3 )w -  i(\rho' -v -\ve^2 f_4) w_y \nonumber \\
& & \quad + \ve^2 [ F_{2}^{L}(t, y) +  iG_{2}^{L}(t, y)] e^{i\Theta} + \ve^3 f^L(t)  F^{L}_c(y) e^{i\Theta}  .\label{Lw}
\eea
Here
\be\label{F2L}
 F_{2}^{L}(t, y) := m\frac{a' }{a}Q_c^{m-1}y A_{1,c} - \frac 12 f_1 y A_{1,c} - \frac 1\ve (B_{1,c})_t - f_2 \Lambda B_{1,c},
\ee
and
\be\label{G2L}
 G_{2}^{L}(t, y) :=  \frac 1\ve (A_{1,c})_t + f_2 \Lambda A_{1,c}  +\frac{a' }{a} Q_c^{m-1}y B_{1,c}  - \frac 12f_1y B_{1,c}.
\ee
In addition, suppose that $(A_{k,c}(t,y), B_{k,c}(t,y)) $,  satisfy (\ref{IP}) $k=1,2$. Then there exist $K, \mu>0$ such that
\be\label{34}
\| \ve^3 f^{L}(t)F^{L}_c e^{i\Theta}\|_{H^1(\R)}\leq K \ve^3(e^{-\ve\mu |\rho(t)|} +\ve) .
\ee
\een
\end{Cl}

\begin{proof}
From the linear character of $w$ we are reduced to handle only two different kind of terms: $\mathcal L[A_c( t, y) e^{i\Theta}]$ and $\mathcal L[ i B_c(t, y) e^{i\Theta}]$. In addition, we expand in several order of $\ve$ to consider the case $m\in [3, 5)$. Otherwise, the computations are simpler and one does not need an accurate expression for these terms. We left the details to the reader.  

First we compute $\mathcal L[A_c(t, y) e^{i\Theta}]$, for a given smooth real valued function $A$. We have (the subscript $()_t$ means derivative on the first variable)
\bee
\mathcal L[A_c(t, y) e^{i\Theta}] & =&  i (A_c)_t e^{i\Theta} + i c' \Lambda A_c e^{i\Theta}  -(\frac 12 xv'  + c -\frac 14 v^2 +\ga') A_c e^{i\Theta} -i \rho'  (A_c)_x e^{i\Theta} \\
& & \quad  +   \big[ (A_c)_{xx} + iv (A_c)_x  -\frac 14 v^2 A_c  \big] e^{i\Theta}  + \frac{ m a(\ve x)}{a(\ve \rho)} Q_c^{m-1} A_c e^{i\Theta} \\
& = &  - \mathcal L_+ (A_c) e^{i\Theta} +  (\ve \frac{m a'}{a} Q_c^{m-1}- \frac 12 v')y A_c e^{i\Theta} -  (\ga' + \frac 12 v' \rho)A_c e^{i\Theta}  \\
& & \quad  -i(\rho' -v)(A_c)_y e^{i\Theta} + i  ( (A_c)_t  +  c' \Lambda A_c) e^{i\Theta} +  \frac{ m\ve^2 a''}{2a} y^2Q_c^{m-1} A_c e^{i\Theta}   \\
& & \quad +  \frac{ m\ve^3 a^{(3)}}{6a} y^3Q_c^{m-1} A_c e^{i\Theta} +  \ve^4 f(t)  y^4Q_c^{m-1} A_c e^{i\Theta}\\
& =&  - \mathcal L_+ (A_c) e^{i\Theta} -\frac 12 ( v'- \ve f_1 ) y A_c e^{i\Theta}  -  (\ga' + \frac 12 v' \rho -\ve^2 f_3 )A_c e^{i\Theta}  \\
& & \quad  - i(\rho' -v -\ve^2f_4 ) (A_c)_y e^{i\Theta}   + i( c' - \ve f_2) \Lambda A_c e^{i\Theta}  \\
& & \quad + \frac{\ve a' }{a} mQ_c^{m-1} y A_c e^{i\Theta} - \frac \ve 2 f_1 y A_c e^{i\Theta} + i\big[ (A_c)_t  +  \ve f_2  \Lambda A_c \big] e^{i\Theta} \\
& & \quad   +  \frac{ m\ve^2 a''}{2a} y^2Q_c^{m-1} A_c e^{i\Theta}  - \ve^2 f_3 A_c e^{i\Theta}   - i \ve^2f_4 (A_c)_y e^{i\Theta}\\
& & \quad  + \ve^3  \frac{m  a^{(3)}}{6a} y^3Q_c^{m-1} A_c e^{i\Theta}  e^{i\Theta} +  \ve^4 f(t)  F^{\bf II}_c(y) e^{i\Theta},
\eee
where  $F^{\bf II}_c(y) \in \mathcal Y$ and $f(t)$ is exponentially decaying in time. Therefore,
$$
\|\ve^4 f^{\bf II}(t)F^{\bf II}_c(y)\|_{H^1(\R)}\leq K \ve^4 e^{-\mu\ve |\rho(t)|} .
$$
With a similar computation,
\bee
\mathcal{L}[ iB_c(t, y) e^{i\Theta}] & =&  -i \mathcal L_- (B_c) e^{i\Theta} - \frac i 2  ( v' -\ve f_1 ) y B_c e^{i\Theta}  - i (\ga' + \frac 12 v' \rho -\ve^2 f_3 )B_c e^{i\Theta} \\
& & \quad + (\rho' -v-\ve^2 f_4) (B_c)_y e^{i\Theta} - ( c' - \ve f_2 )  \Lambda B_c e^{i\Theta}  \\
& & \quad  + \frac{i \ve  a' }{a} Q_c^{m-1} y B_c e^{i\Theta}  -\frac i2  \ve f_1 y B_c e^{i\Theta}   -  \big[ (B_c)_t +  \ve f_2 \Lambda B_c\big] e^{i\Theta}   \\
& & \quad + i \frac{ \ve^2 a''}{2a} y^2Q_c^{m-1} B_c e^{i\Theta} +i \ve^3  \frac{a^{(3)}}{6a} y^3Q_c^{m-1} B_c e^{i\Theta} \\
& & \quad  + \ve^2 f_4(t) (B_c)_y e^{i\Theta}   -i \ve^2 f_3 B_c e^{i\Theta} + i \ve^4 g^{\bf II}(t)  G_c^{II} (y) e^{i\Theta},
\eee
with $\|\ve^4g^{\bf II}(t)  G_c^{II} (y) e^{i\Theta} \|_{H^1(\R)} \leq K\ve^4e^{-\mu\ve |\rho(t)|} .$
Collecting the above calculations, we finally obtain (\ref{Lw}). Estimate (\ref{34}) can be directly verified. 

\end{proof}

For the final term $\tilde N[w]$ we have the following

\begin{Cl}[Decomposition of  $\tilde N{[}w{]}$]\label{lem:SintIII}~
\ben
\item
Suppose that $2\leq m< 3$ and (\ref{IP}) holds for $(A_{1,c}, B_{1,c})$. Then there exists $K,\mu>0$ such that 
$$ 
\| \tilde N[w] \|_{H^1(\R)} \leq K \ve^2 e^{-\mu\ve |\rho(t)|},
$$
uniformly for every $t\in [-T_\ve, T_\ve]$. 
\item Suppose now $3\leq m<5$, and that (\ref{IP}) holds for each $(A_{k,c}, B_{k,c})$, $k=1,2$. Then one has
$$ 
\tilde N[w] =  \ve^2 (N^{2,1}(t, y) + i N^{2,2}(t, y)) e^{i\Theta} + O_{H^1(\R)}( \ve^3 e^{-\ve\mu|\rho(t)|}),
$$
with 
\be\label{N21}
N^{2,1} := \frac 12 (m-1)\tilde a(\ve \rho)Q_c^{m-2} (mA_{1,c}^2 +B_{1,c}^2 ) , \quad N^{2,2} :=   (m-1) \tilde  a (\ve \rho)Q_c^{m-2}  A_{1,c}B_{1,c}.
\ee
\een
\end{Cl}

\begin{proof}
First we prove the case $2\leq m< 3$. Recall that $w = \ve [A_{1,c} (t , y)  +i B_{1,c}(t, y) ] e^{i\Theta}$, with the functions $A_c(t, \cdot), B_c(t, \cdot) \in \mathcal Y$, real valued. Here we have
$$
\tilde N[w] = O(Q_c^{m-2} |w|^2 + |w|^3) = O_{H^1(\R)}(\ve^2 e^{-\ve\mu |\rho(t)|}),
$$
uniformly in time.

\medskip

Finally, let us consider the case $3\leq m<5$. From (\ref{defv}) we have $ w(t,x) = \sum_{k=1}^2 \ve^k (A_{k,c}(t, y)  + i B_{k,c}(t,y))e^{i\Theta}$. In order to simplify the computations, we assume $(A_{k,c}, B_{k,c})_{k=1,2}$ satisfy (\ref{IP}) on the interval $[-T_\ve, T_\ve]$ (which is indeed the case). We have
\bea\label{N2}
 \tilde N[w]  &= & \frac{(m-1)a(\ve x)}{2 a^{\frac{m-2}{m-1}}(\ve\rho)}Q_c^{m-2}(y)  \big\{  e^{i\Theta} |w|^2 + 2 \re (e^{i\Theta} \bar w) w  +     (m-3)e^{i\Theta} (\re (e^{i\Theta} \bar w))^2 \big\} \nonu \\
& & + O_{H^1(\R)}(\ve^3 e^{-\ve \mu |\rho(t)|}).
\eea
Now we replace $w$  in the above expression and we arrange the terms obtained according to the power of $\ve$ and between real and imaginary parts. We perform this computation in several steps. First, note that
$$
a(\ve x) = a(\ve\rho) + O(\ve y). 
$$
On the other hand,
$$
|w|^2 = \ve^2\{ A_{1,c}^2 + B_{1,c}^2  \} + O_{H^1(\R)}(\ve^3 e^{-\ve\mu |\rho(t)|}).
$$
Similarly $\re (e^{i\Theta} \bar w) = \ve A_{1,c} + \ve^2 A_{2,c}$. Therefore
$$
\re (e^{i\Theta} \bar w) w   =  \ve^2 (A_{1,c}^2  + iA_{1,c} B_{1,c}) e^{i\Theta} + O_{H^1(\R)}(\ve^3 e^{-\ve\mu |\rho(t)|}),
$$
and
$$
e^{i\Theta}(\re (e^{i\Theta} \bar w))^2 = \ve^2  A_{1,c}^2e^{i\Theta}  + O_{H^1(\R)}(\ve^3 e^{-\ve\mu |\rho(t)|}).
$$
Collecting these expansions and replacing in (\ref{N2}) we obtain
$$
\tilde N[w]  =   \frac 12 \ve^2 (m-1)\tilde a (\ve\rho) Q_c^{m-2} \big\{  mA_{1,c}^2 + B_{1,c}^2   + 2i A_{1,c}B_{1,c}  \big\}e^{i\Theta}  +  O_{H^1(\R)}(\ve^3 e^{-\ve\mu |\rho(t)|}). 
$$
We are done.
\end{proof}

Collecting the estimates from Claims \ref{lem:SQ}, \ref{lem:dSKdVw} and \ref{lem:SintIII}, we obtain Proposition \ref{prop:decomp}. The proof is now complete.

\bigskip

\section{Some identities related to the soliton $Q$}\label{AidQ}

This section has been taken in part from Appendix C in \cite{MMcol1}.

\begin{lem}[Identities for the soliton $Q$]\label{IdQ}~

Suppose $m>1$ and denote by $Q_c := c^{\frac 1{m-1}} Q(\sqrt{c} x)$ the scaled soliton, with $Q$ solution of $-Q'' +Q -Q^m=0$ in $\R$. Then
\ben
\item \emph{Energy}. Let $E_1 [u]:= E_{a\equiv 1}[u]$. Then
$$
E_1[Q]= -\frac 12  \la_0\int_\R Q^2 =-\la_0M[Q],  \qquad \hbox{with }\quad   \la_0 = \frac{5-m}{m+3}.
$$
\item \emph{Integrals}. Recall $\theta = \frac 1{m-1} -\frac 14$. Then
$$
\int_\R Q_c = c^{\theta-\frac 14} \int_\R Q, \quad \int_\R Q_c^{2} = c^{2\theta} \int_\R Q^2, \quad E_1[Q_c] = c^{2\theta +1}E_1[Q].
$$
and finally
$$
\int_\R Q_c^{m+1} = \frac{2(m+1)c^{2\theta +1}}{m+3} \int_\R Q^2, \quad \int_\R \Lambda Q_c = (\theta -\frac 14) c^{\theta -\frac 54} \int_\R Q, \quad  \int_\R \Lambda Q_c Q_c =\theta c^{2\theta -1} \int_\R Q^2.
$$
\item \emph{Integrals with powers}.
$$
\int_\R Q'^2 = \frac{m-1}{m+3}\int_\R Q^2, \quad \int_\R y^2 Q^{m+1} = \frac{m+1}{m+3}\big[2\int_\R y^2 Q^2 -\int_\R Q^2 \big],
$$
and
$$
\int_\R y^4 Q^{m+1} = \frac{m+1}{m+3}\big[2\int_\R y^4 Q^2 - 6\int_\R y^2 Q^2 \big], \quad \int_\R y^2 Q'^2 = \frac 2{m+3}\int_\R Q^2 +\frac{m-1}{m+3}\int_\R y^2 Q^2.
$$
\een
\end{lem} 

\bigskip

{\bf Acknowledgments}. The author wishes to thank Y. Martel and F. Merle for presenting him this problem and for  their continuous encouragement during the elaboration of this work; and the DIM members at Universidad de Chile for their kind hospitality during the elaboration of part of this work.

\end{document}